\documentclass[a4paper, 10pt,reqno]{amsart}
\usepackage[margin=1.5cm,top=2.2cm,bottom=2.2cm]{geometry}
\usepackage{amsmath,amssymb,amsthm,mathtools}
\usepackage[T1]{fontenc}
\usepackage[utf8]{inputenc}
\usepackage{hyperref}
\usepackage{booktabs}
\usepackage{enumitem}

\usepackage{xcolor}

\newtheorem{theorem}{Theorem}[section]
\newtheorem{proposition}[theorem]{Proposition}
\newtheorem{lemma}[theorem]{Lemma}
\newtheorem{corollary}[theorem]{Corollary}
\theoremstyle{definition}
\newtheorem{assumption}[theorem]{Assumption}
\newtheorem{definition}[theorem]{Definition}

\theoremstyle{remark}
\newtheorem{remark}[theorem]{Remark}
\newcommand{\E}{\mathbb{E}}
\newcommand{\R}{\mathbb{R}}
\newcommand{\Rp}{\mathbb{R}_+}

\newcommand{\eps}{\varepsilon}
\newcommand{\half}{\tfrac12}

\newcommand{\lp}[1]{L^{#1}(\mathbb{R}_+)}
\DeclareMathOperator{\Var}{Var}
\DeclareMathOperator{\Cov}{Cov}

\hypersetup{
  colorlinks=true,
  linkcolor=blue!60!black,
  citecolor=green!50!black,
  urlcolor=blue!70!black
}

\providecommand{\doi}[1]{\textsc{doi}:\,\href{https://doi.org/#1}{#1}}
\providecommand{\arxiv}[1]{\textsc{arXiv}:\,\href{https://arxiv.org/abs/#1}{#1}}
\begin{document}

\title{Scaling Limits of Bivariate Nearly-Unstable Hawkes Processes
and Applications to Rough Volatility}

\author{Sohaib El Karmi}

\begin{abstract}
We study a pair of nearly-unstable Hawkes processes coupled through a
one-directional, or triangular, cross-excitation: the first component
evolves autonomously and excites the second, but not conversely.  Each
component is self-exciting through a heavy-tailed memory kernel, and the two
kernels are allowed to have different tail indices, so that the components
carry genuinely different degrees of roughness.  As the system approaches
criticality, we prove that the suitably rescaled intensity vector converges
to the unique solution of a coupled system of stochastic Volterra equations
of rough-volatility type.  The first limiting component is an autonomous
rough process; the second is driven both by its own noise and by an
inherited noise transmitted through an effective cross-kernel that combines
the two memory structures.  As a consequence, we obtain a short-time
cross-decorrelation law: although coupled, the two components become
asymptotically uncorrelated at small time scales, at an explicit rate
determined by both tail indices.  This vanishing, time-dependent correlation
distinguishes the bivariate limit from a pair of independent rough processes
and from classical models with constant Brownian correlation.  To our
knowledge, this is the first functional limit theorem of this kind in the
heterogeneous setting.
\end{abstract}


\maketitle

\newpage
\setcounter{tocdepth}{1}
\tableofcontents

\newpage

\section{Introduction}
\label{sec:intro}

\subsection{Background and motivation}

Hawkes processes~\cite{Hawkes1971} are self-exciting point processes whose
intensities depend on past arrivals through memory kernels.  They are a
standard microscopic model for order flow and high-frequency financial
activity.  When the kernel is heavy-tailed and the branching ratio is close
to one, the accumulated effect of many small excitations produces long-memory
macroscopic limits.

The univariate theory is by now well understood.  Jaisson and
Rosenbaum~\cite{JR2016} proved that a nearly unstable Hawkes process with a
power-law kernel of index $\alpha\in(1/2,1)$ converges, after the correct
renormalization, to a rough fractional diffusion with Hurst parameter
$H=\alpha-1/2$.  This mechanism is one of the microscopic explanations for
rough volatility, whose empirical relevance was documented in
Gatheral--Jaisson--Rosenbaum~\cite{GJR2018}.  Subsequent work, including
Horst--Xu--Zhang~\cite{HXZ2023} and Horst--Xu~\cite{HX2024}, refined the
functional convergence and tightness theory for heavy-tailed Hawkes
processes.

The present paper studies the corresponding bivariate problem when two
nearly unstable Hawkes components have \emph{different} heavy-tail indices
and interact through a one-directional cross-excitation.  The goal is to
identify the scaling limit of the two renormalized intensity processes and
to understand the cross-dependence created by the interaction.

\subsection{Notation and interpretation of the Hawkes model}
\label{ssec:hawkes-reminder}

We briefly fix the notation and interpretation used throughout the paper.
A linear Hawkes process is a counting process $N$ whose stochastic intensity
$\lambda$ satisfies
\[
\lambda_t=\mu+\int_0^t \phi(t-s)\,dN_s,
\]
or, equivalently, such that
\[
M_t:=N_t-\int_0^t\lambda_s\,ds
\]
is a martingale.  The constant $\mu>0$ is the baseline intensity, while the
non-negative kernel $\phi$ measures the excitation created by past events.
Its mass $\|\phi\|_1$ is the branching ratio; the process is subcritical when
$\|\phi\|_1<1$ and nearly unstable when the scaled kernels satisfy
$\|\phi_T\|_1=a_T\uparrow1$.

The memory of the process is encoded in the tail of $\phi$.  In the
heavy-tailed case
\[
1-\widehat\phi(z)\sim \delta z^\alpha,
\qquad z\downarrow0,
\qquad \alpha\in(1/2,1),
\]
the near-critical resolvent produces a rough limiting intensity with Hurst
index $H=\alpha-1/2$.  The present paper studies the bivariate situation in
which the two components have distinct tail indices and interact through a
triangular cross-excitation.

\subsection{Model and object of study}
\label{ssec:model-intro}

We introduce the model already in the introduction, since this is the object
of the whole paper.  Fix $T\geq1$.  Let $(N^{T,1},N^{T,2})$ be a pair of
point processes on $[0,T]$ with stochastic intensities
$(\lambda^{T,1},\lambda^{T,2})$.  We work in the triangular case: component
$1$ excites component $2$, but not conversely.  Thus
\begin{align}
\lambda_t^{T,1}
  &= \mu_T^1 + \int_0^t \phi_T^1(t-s)\,dN_s^{T,1},
\label{eq:intro-int1}\\
\lambda_t^{T,2}
  &= \mu_T^2 + \int_0^t \phi_T^2(t-s)\,dN_s^{T,2}
     + \int_0^t \phi_T^{12}(t-s)\,dN_s^{T,1}.
\label{eq:intro-int2}
\end{align}
The self-exciting kernels are $\phi_T^i=a_T^i\phi^i$, where
$\|\phi^i\|_1=1$ and $\phi^i$ has heavy-tail index
$\alpha_i\in(1/2,1)$.  The cross-kernel is
$\phi_T^{12}=b_T^{12}\psi^{12}$.  Throughout the main part of the paper we
order the exponents as
\[
   \alpha_1<\alpha_2.
\]
This is not a symmetry assumption: because the interaction is triangular,
component $1$ is the driver of component $2$, and we choose the labelling so
that the driver is the rougher component.

The system is nearly unstable.  The branching ratios satisfy
\begin{equation}\label{eq:intro-crit}
  T^{\alpha_i}(1-a_T^i)\longrightarrow \lambda_i\in(0,\infty),
  \qquad i=1,2.
\end{equation}
The baseline intensities are scaled so that
$T^{1-\alpha_i}\mu_T^i\to m_i>0$.  The cross-amplitude is scaled at the
triangular critical order
\[
T^{2\alpha_1-\alpha_2}b_T^{12}\longrightarrow b_\infty^{12}\in(0,\infty),
\]
which defines the limiting cross-excitation strength $b_\infty^{12}$.  The process whose limit we study is the
renormalized intensity vector
\begin{equation}\label{eq:intro-V}
V_t^{T,i}
 := \frac{1-a_T^i}{m_iT^{\alpha_i-1}}\,\lambda_{Tt}^{T,i},
\qquad t\in[0,1],\quad i=1,2.
\end{equation}
This normalization combines the macroscopic time change $t\mapsto Tt$ with
the critical intensity scale.  It is chosen so that each component has a
non-degenerate limit: without this scaling the intensities either collapse
or diverge in the nearly unstable regime.

\subsection{Main result and cross-kernel mechanism}

The main theorem states that $(V^{T,1},V^{T,2})$ converges weakly in
$C([0,1];\mathbb R^2)$ to the unique triangular weak solution of a stochastic
Volterra system.  The limiting kernels $K_i$ are the Mittag-Leffler kernels
associated with the two self-excitation kernels; they are defined precisely
in Section~\ref{sec:setup}.  The effective cross-kernel is
\[
  L_{12}=b_\infty^{12}\|\psi^{12}\|_1\,(K_1*K_2),
\]
and the deterministic functions $b_1,b_2$ are the corresponding limiting
mean profiles.  With the constants
\[
\nu_i=(m_i\lambda_i)^{-1/2},
\qquad
\gamma_{12}=\frac{m_1}{m_2\lambda_1},
\]
the limiting system has the form
\begin{align*}
V_t^1
  &= b_1(t)+\nu_1\int_0^t K_1(t-s)\sqrt{V_s^1}\,d\beta_s^1,\\
V_t^2
  &= b_2(t)+\nu_2\int_0^t K_2(t-s)\sqrt{V_s^2}\,d\beta_s^2
     +\gamma_{12}\nu_1\int_0^t L_{12}(t-s)\sqrt{V_s^1}\,d\beta_s^1,
\end{align*}
where $\beta^1$ and $\beta^2$ are independent Brownian motions.  Thus the
first component is autonomous, whereas the second component is driven both
by its own rough noise and by the noise of the first component through the
effective cross-kernel $L_{12}$.

This cross-kernel is the main structural feature of the limit.  It is not a
constant instantaneous correlation between the Brownian motions.  Rather,
the dependence between the two limiting intensities is transmitted through
the convolution $K_1*K_2$, which contains both heavy-tail indices.  Since
\[
  (K_1*K_2)(t)\sim C\,t^{\alpha_1+\alpha_2-1},
  \qquad t\downarrow0,
\]
the cross-kernel is smoother at the origin than the individual self-kernels
$K_i(t)\sim C_i t^{\alpha_i-1}$.  This regularity hierarchy is the analytic
reason why the cross-excitation survives in the limit without destroying
$C$-tightness.

A consequence is the zero-start short-time decorrelation law
\begin{equation*}
  \varrho(t)\sim C_\varrho\,t^{(3\alpha_1-\alpha_2)/2},
  \qquad t\downarrow0,
\end{equation*}
with an explicit positive constant $C_\varrho$; see
Corollary~\ref{cor:rho}.  The exponent depends on both roughness indices and
reduces to $\alpha$ in the homogeneous case $\alpha_1=\alpha_2=\alpha$.
This is the sense in which the heterogeneous bivariate model differs from a
product of two univariate limits or from a bivariate rough model with a
constant Brownian correlation.

\subsection{Why triangularity and why the ordering $\alpha_1<\alpha_2$}

The triangular assumption is part of the mathematical structure of this
paper.  It is also the point at which the proof differs most from a formal
multivariate extension of the univariate theory.  In the triangular case,
component $1$ is autonomous, so its limit can first be identified by the
scalar nearly unstable Hawkes theory.  The second component is then treated
conditionally on the first component.

The fully coupled case, where also $2$ excites $1$, is not covered by the
present proof.  Three obstructions appear.  First, the uniform moment bounds
used for $C$-tightness become a genuinely coupled two-dimensional renewal
problem.  Second, the identification of the limiting martingale problem would
require joint second-moment estimates for the two interacting intensities,
not only separate scalar estimates.  Third, the conditional uniqueness
argument for the second Volterra equation uses the independence of
$\beta^2$ from $\sigma(V^1,\beta^1)$; this independence fails once $V^1$
also depends on $\beta^2$.  The third point is made precise in
Section~\ref{ssec:triangular-limit-points}, where the required independence
is derived from a one-sided Poisson embedding of the triangular system.
The triangular hypothesis therefore isolates the
case in which the full tightness--identification--uniqueness argument can be
closed.

The ordering $\alpha_1<\alpha_2$ fixes the driver/response convention used
throughout the paper.  Since component $1$ is the driver, this convention
means that the driving component is the rougher one.  The opposite ordering
can be handled by relabelling when the direction of excitation is reversed,
and the homogeneous case requires only notational simplifications; see
Remark~\ref{rem:equal-alpha}.

\subsection{Relation to the literature and contribution}
\label{ssec:literature-contribution}

The closest predecessors are the univariate nearly unstable Hawkes limits
of Jaisson--Rosenbaum~\cite{JR2015,JR2016}: the diffusive limit in the
light-tailed case and the rough fractional diffusion in the heavy-tailed
case.  El~Euch--Rosenbaum~\cite{ER2019} identify the rough Heston model in
this microstructural framework.  We keep the same near-critical mechanism
but introduce a triangular bivariate interaction with distinct tail indices.
The tightness inputs are related to Horst--Xu--Zhang~\cite{HXZ2023} and
Horst--Xu~\cite{HX2024}, while uniqueness of the limiting Volterra equations
uses the affine Volterra framework of Abi Jaber--Larsson--Pulido~\cite{ALP2019}.
The conditional uniqueness statement needed for the second component is
formulated here as Theorem~\ref{thm:ALP-extension}.

In the multivariate setting, scaling limits of Hawkes processes go back,
in the diffusive regime, to Bacry--Delattre--Hoffmann--Muzy~\cite{BDHM2013}.
In the heavy-tailed near-critical regime, the closest comparison is the
work of Rosenbaum--Tomas~\cite{TomasRosenbaum2019}, which derives
multidimensional rough volatility models from nearly unstable multivariate
Hawkes processes with a common tail exponent, so that all components share
a single limiting roughness.  Two recent related works are likewise driven
by a single tail statistic: Muhle-Karbe, Ouazzani~Chahdi, Rosenbaum and
Szymanski~\cite{MORS2026} study a two-layer Hawkes system with
one-directional excitation between core and reaction order flow, in which
all limiting exponents are deterministic functions of a single persistence
index $H_0$; and Wang--Cui~\cite{WangCui2025} obtain the one-dimensional
rough Heston model as the scaling limit of bivariate cumulative
INAR$(\infty)$ processes with a common exponent.  By contrast, the present
paper treats the genuinely heterogeneous regime: two independent tail
exponents $\alpha_1\neq\alpha_2$ produce two distinct limiting roughnesses
within a single coupled system.  Its main contribution is the
identification of the effective cross-kernel
$L_{12}=b_\infty^{12}\|\psi^{12}\|_1(K_1*K_2)$ and of the resulting
short-time cross-decorrelation exponent $(3\alpha_1-\alpha_2)/2$.

\subsection{Organization}

Section~\ref{sec:setup} gives the full assumptions and the precise statement
of the main theorem.  Section~\ref{sec:cross-kernel} identifies the
cross-kernel, proves the covariance formula, and derives the short-time
decorrelation law.  Section~\ref{sec:numerics} provides analytic and
numerical illustrations of the exponent and of the regularity hierarchy.
Sections~\ref{sec:tightness}, \ref{sec:identification}, and
\ref{sec:uniqueness} prove tightness, identify the subsequential limits, and
establish uniqueness in law.  Two appendices record the removal of the
scale-matching condition and the detailed asymptotic calculations underlying
the constants used in the main text.

\medskip
\noindent\emph{Conventions.}
Throughout, $\widehat f$ denotes the Laplace transform of $f$, $*$ the
convolution on $\mathbb R_+$, $\delta_0$ the Dirac mass at the origin, and
$\|\cdot\|_1,\|\cdot\|_2$ the $L^1(\mathbb R_+)$ and $L^2(\mathbb R_+)$
norms.  The superscript $i\in\{1,2\}$ indexes the two components and
$T\ge1$ is the observation horizon.  We write $a\wedge b=\min(a,b)$, and
$C$ for a finite positive constant that may change from line to line and is
uniform in $T\ge T_0$.  All other notation is introduced where it first
appears.

\section{Model, assumptions, and main results}
\label{sec:setup}

\subsection{The bivariate Hawkes model}

Throughout this paper, we work in the asymmetric (triangular) case:
component~1 excites component~2 but not conversely.  Equivalently, if
$\psi^{21}$ denotes a possible reverse cross-kernel from component~2 to
component~1, then we impose $\psi^{21}=0$.
The results below (Sections~\ref{sec:cross-kernel}--\ref{sec:uniqueness})
apply in this setting.
The three specific obstructions that arise when $\psi^{21}\neq 0$ are
identified in Section~\ref{sec:intro} above.

Because the coupling is one-directional, the labelling of the two
components is not interchangeable: we take component~1 to be both the
\emph{driver} of the interaction and the \emph{rougher} of the two
components, i.e.\ $\alpha_1<\alpha_2$.  This convention is what gives the
zero-start short-time decorrelation exponent its value
$(3\alpha_1-\alpha_2)/2$ (Corollary~\ref{cor:rho-intro}); the opposite ordering
$\alpha_1>\alpha_2$ is recovered by relabelling, and the equal-exponent
case $\alpha_1=\alpha_2$ requires only notational changes
(Remark~\ref{rem:equal-alpha}).

\begin{assumption}\label{ass:kernels}
The following five conditions hold.

\medskip
\noindent\textbf{(A1) Heavy-tail self-kernels.}
For each $i\in\{1,2\}$, the self-exciting kernel satisfies
$\phi^i\geq 0$, $\phi^i\in\lp{1}\cap\lp{2}$, $\|\phi^i\|_1=1$, and
\begin{equation}\label{eq:heavy-tail}
1 - \widehat{\phi^i}(z) = \delta_i z^{\alpha_i} + o(z^{\alpha_i}),
\qquad z\to 0,\ \Re z\ge 0,
\end{equation}
for some $\alpha_i\in(1/2,1)$, $\delta_i>0$, with $\alpha_1<\alpha_2$
(component~$1$ is the rougher of the two).
We assume in addition that each $\phi^i$ is strongly aperiodic,
\[
\forall\rho>0,\qquad
\sup_{|w|\geq \rho}|\widehat{\phi^i}(iw)|<1,
\]
and absolutely continuous with $(\phi^i)'\in L^1(\mathbb R_+)$. These
regularity assumptions ensure the high-frequency Fourier domination used
in Lemma~\ref{lem:domination}.

\medskip
\noindent\textbf{(A2) Near-criticality.}
The self-exciting amplitudes satisfy $a_T^i\in(0,1)$, $a_T^i\uparrow 1$, and
\begin{equation}\label{eq:crit}
T^{\alpha_i}(1-a_T^i)\longrightarrow\lambda_i\in(0,\infty).
\end{equation}

\medskip
\noindent\textbf{(A3) Cross-kernel.}
The cross-exciting kernel satisfies $\psi^{12}\geq 0$,
$\psi^{12}\in\lp{1}\cap\lp{2}$,
and $\phi_T^{12}=b_T^{12}\psi^{12}$ with
\begin{equation}\label{eq:A3-scaling}
T^{2\alpha_1-\alpha_2}\,b_T^{12} \longrightarrow b_\infty^{12}\in(0,\infty).
\end{equation}
Equivalently, $b_T^{12}\sim b_\infty^{12}\,T^{\alpha_2-2\alpha_1}$.
This is the critical cross-excitation scaling for the present normalization:
it makes the deterministic cross-compensator in $V^{T,2}$ finite and keeps
the stochastic inherited-noise term non-degenerate in the limit.

\begin{remark}[Interpretation of~(A3)]\label{rem:A3}
The condition~\eqref{eq:A3-scaling} balances the growth of the first
component, the resolvent amplification of the second component, and the
renormalization of $V^{T,2}$.  Indeed,
$\mu_T^1/(1-a_T^1)$ is of order $T^{2\alpha_1-1}$, while the second
resolvent contributes a factor of order $T^{\alpha_2}$ and the
normalization of $V^{T,2}$ contributes $T^{1-2\alpha_2}$.  Hence the
renormalized deterministic cross-contribution is of order
$b_T^{12}T^{2\alpha_1-\alpha_2}$, which is finite precisely under
\eqref{eq:A3-scaling}.  The same balance gives a non-trivial stochastic
cross term.

When $\alpha_1=\alpha_2=\alpha$ (homogeneous case), the condition reduces
to $T^\alpha b_T^{12}\to b_\infty^{12}$, i.e.\ the cross-amplitude is of
the same order as the distance to criticality $1-a_T^i$.
When $\alpha_1<\alpha_2$, the exponent $\alpha_2-2\alpha_1$ is negative,
so the raw cross-amplitude tends to zero; the effective coupling survives
after the near-critical resolvent amplification.
\end{remark}

\medskip
\noindent\textbf{(A4) Baseline scaling.}
The baseline intensities satisfy $T^{1-\alpha_i}\mu_T^i\to m_i\in(0,\infty)$.

\medskip
\noindent\textbf{(A5) Scale matching.}
\begin{equation}\label{eq:scale-match}
\lambda_1^{1/\alpha_1} = \lambda_2^{1/\alpha_2} =: \lambda^*.
\end{equation}
\end{assumption}

\begin{remark}[On condition (A1)]
Condition~(A1) is equivalent to saying that $1-\widehat{\phi^i}$ is regularly
varying at zero with exponent $\alpha_i$, i.e.\ $\phi^i$ belongs to the
domain of attraction of a positive $\alpha_i$-stable law
(Karamata representation theorem~\cite{BinghamGoldieTeugels1987}).
The Laplace-domain condition~\eqref{eq:heavy-tail} is the standard
form for the heavy-tail assumption in the Hawkes literature.
\end{remark}

We now restate the model of~\eqref{eq:intro-int1}--\eqref{eq:intro-int2}
with the notation fixed by Assumption~\ref{ass:kernels}.
As above, $N^{T,i}$ is the counting process of component~$i$, $\mu_T^i$ its
baseline intensity, $\phi_T^i:=a_T^i\phi^i$ the near-critical self-exciting
kernels (so $\|\phi_T^i\|_1=a_T^i\uparrow1$ by~(A1)--(A2)), and
$\phi_T^{12}=b_T^{12}\psi^{12}$ the cross-kernel from~(A3).
The intensity processes are:
\begin{align}
\lambda_t^{T,1} &= \mu_T^1 + \int_0^t \phi_T^1(t-s)\,dN_s^{T,1},
\label{eq:int1}\\
\lambda_t^{T,2} &= \mu_T^2 + \int_0^t \phi_T^2(t-s)\,dN_s^{T,2}
+ \int_0^t \phi_T^{12}(t-s)\,dN_s^{T,1}.
\label{eq:int2}
\end{align}
We denote by $M_t^{T,i}:=N_t^{T,i}-\int_0^t\lambda_s^{T,i}\,ds$
the compensated martingales, by $\psi_T^i:=\sum_{k\geq 1}(\phi_T^i)^{*k}$
the resolvents, and we define the renormalized intensity processes
\begin{equation}\label{eq:V}
V_t^{T,i} := \frac{1-a_T^i}{m_i\,T^{\alpha_i-1}}\,\lambda_{Tt}^{T,i},
\qquad t\in[0,1].
\end{equation}

The vector $(V^{T,1},V^{T,2})$ is the central object of this paper.
The normalization in~\eqref{eq:V} is the bivariate analogue of the
Jaisson--Rosenbaum rescaling~\cite{JR2016}: it is calibrated so that each
component has a non-degenerate weak limit as the system approaches
criticality ($T\to\infty$), neither exploding nor collapsing to its
deterministic mean.  The weak limit of $(V^{T,1},V^{T,2})$ is precisely the
bivariate rough-volatility model announced in the title; the heterogeneous
roughness, the cross-kernel mechanism, and the short-time decorrelation
described below are all properties of this limit.  The remainder of the
paper is devoted to establishing the weak convergence of
$(V^{T,1},V^{T,2})$ in $C([0,1];\mathbb{R}^2)$ and identifying its limit.

The Mittag-Leffler kernels $K_i$ are defined through their Laplace transform:
\begin{equation}\label{eq:ML}
\widehat{K_i}(z) = \frac{1}{1+\tilde\delta_i z^{\alpha_i}},
\qquad\tilde\delta_i:=\frac{\delta_i}{\lambda_i},
\end{equation}
so that $K_i(t)\sim(\tilde\delta_i\,\Gamma(\alpha_i))^{-1}t^{\alpha_i-1}$
as $t\downarrow 0$, and $K_i\in\lp{2}$ since $\alpha_i>1/2$.

\begin{remark}[The deterministic mean profiles]\label{rem:three-means}
The normalization~\eqref{eq:V} separates the baseline scales $m_i$ from the
renormalized mean profiles.  With the notation
\[
\gamma_{12}:=\frac{m_1}{m_2\lambda_1},
\qquad
L_{12}:=b_\infty^{12}\|\psi^{12}\|_1(K_1*K_2),
\]
we write
\begin{equation}\label{eq:b1-def}
b_1(t):=\int_0^t K_1(w)\,dw
\end{equation}
for the deterministic profile of the autonomous first component.  For the
second component the deterministic profile contains both its self part and
the finite cross-compensator:
\begin{equation}\label{eq:b2-full}
b_2(t):=\int_0^t K_2(w)\,dw+
\gamma_{12}\int_0^t L_{12}(w)\,dw.
\end{equation}
Thus $b_2$ is not merely the univariate deterministic profile of component~2;
it also records the deterministic effect of the triangular excitation
$1\to2$.  The stochastic cross-excitation is encoded separately through the
same kernel $L_{12}$ in the martingale part of the limiting equation.
\end{remark}

\begin{remark}[Equal exponents]\label{rem:equal-alpha}
Assumption~\ref{ass:kernels} requires $\alpha_1<\alpha_2$
to ensure that component~1 is the rougher component and that the
zero-start decorrelation exponent $(3\alpha_1-\alpha_2)/2$ is interpreted
under a fixed driver/response ordering.
The results of Sections~\ref{sec:cross-kernel}--\ref{sec:uniqueness}
extend without change to the case $\alpha_1=\alpha_2=:\alpha$;
in that case the cross-kernel reduces to
$L_{12}=b_\infty^{12}\|\psi^{12}\|_1\cdot(K*K)$
and the decorrelation exponent becomes $\alpha$.
The case $\alpha_1>\alpha_2$ is obtained by relabelling the components.
\end{remark}

\subsection{The main theorem}

\begin{theorem}[Main result]\label{thm:main-intro}
Under Assumption~\ref{ass:kernels},
the bivariate renormalized intensity process $(V^{T,1},V^{T,2})_{T\geq 1}$
converges weakly in $C([0,1];\mathbb{R}^2)$
to the unique triangular weak solution $(V^1,V^2)$
(Definition~\ref{def:triangular-solution}) of the coupled stochastic Volterra system.
Set
\[
\nu_i=(m_i\lambda_i)^{-1/2},\qquad
\gamma_{12}=\frac{m_1}{m_2\lambda_1},
\qquad
L_{12}:=b_\infty^{12}\|\psi^{12}\|_1(K_1*K_2).
\]
Then
\begin{align}
V_t^1 &= b_1(t)
  + \nu_1\int_0^t K_1(t-s)\sqrt{V_s^1}\,d\beta_s^1,
\label{eq:main-limit1}\\
V_t^2 &= b_2(t)
  + \nu_2\int_0^t K_2(t-s)\sqrt{V_s^2}\,d\beta_s^2
  + \gamma_{12}\nu_1
    \int_0^t L_{12}(t-s)\sqrt{V_s^1}\,d\beta_s^1,
\label{eq:main-limit2}
\end{align}
where $\beta^1,\beta^2$ are independent standard Brownian motions,
$b_1$ is defined by~\eqref{eq:b1-def}, and $b_2$ is defined by~\eqref{eq:b2-full}.

The limiting process has the following properties.
\begin{enumerate}
\item[\textup{(i)}] \textbf{Heterogeneous roughness.}
Set $H_i:=\alpha_i-\tfrac12\in(0,\tfrac12)$.  Almost surely, $V^i$ is
H\"older continuous of every order $\gamma<H_i$, and the exponent is
sharp: away from the time origin, the mean-square increments satisfy
$\E[(V^i_t-V^i_s)^2]\asymp(t-s)^{2H_i}$
(Proposition~\ref{prop:regularity}).  In this sense each $V^i$ has Hurst
parameter $H_i$, and $H_1<H_2$: the components carry distinct
roughnesses.

\item[\textup{(ii)}] \textbf{Covariance formula.}
The exact covariance of the limit process is
\[
\mathrm{Cov}(V_t^1,V_t^2)
= \gamma_{12}\nu_1^2
  \int_0^t K_1(t-s)\,L_{12}(t-s)\,b_1(s)\,ds,
\]
where $b_1(s)=\mathbb{E}[V_s^1]$.
In the stationary regime $b_1(s)\equiv\bar V^1$, this simplifies to
$\gamma_{12}\nu_1^2\bar V^1\int_0^t K_1(u)L_{12}(u)\,du$.
The effective cross-kernel is $L_{12}:=b_\infty^{12}\|\psi^{12}\|_1\cdot(K_1*K_2)$,
with $(K_1*K_2)(t)\sim C_{12}\,t^{\alpha_1+\alpha_2-1}$ as $t\downarrow 0$
(Proposition~\ref{prop:main}).
\end{enumerate}

A self-contained statement with all explicit constants appears as
Theorem~\ref{thm:main-final}.
\end{theorem}

\begin{corollary}[Short-time decorrelation, informal preview]
\label{cor:rho-intro}
Under the assumptions of Theorem~\ref{thm:main-intro} and in the
zero-start scaling regime,
the functional correlation satisfies
\[
\varrho(t)\sim C_\varrho\,t^{(3\alpha_1-\alpha_2)/2},
\qquad t\downarrow0.
\]
Since $\alpha_1>1/2$ and $\alpha_2<1$, the exponent is strictly positive.
In the homogeneous case $\alpha_1=\alpha_2$, it reduces to $\alpha_1$.
In particular $\varrho(t)\to 0$ as $t\downarrow 0$: the two components are decorrelated
at short time scales.
The explicit constant $C_\varrho$ and its proof are in Corollary~\ref{cor:rho}.
A stationary comparison, where the mean profiles are locally constant near
zero, gives the exponent $\alpha_1$; this is not the zero-start regime used
in the main theorem.
\end{corollary}

\begin{remark}[Interpretation]
The structure of~\eqref{eq:main-limit1}--\eqref{eq:main-limit2} can be read
as follows.
The first component $V^1$ is an autonomous rough process with Hurst parameter
$H_1$ (Proposition~\ref{prop:regularity}); it is driven entirely by $\beta^1$.
The second component $V^2$ has two noise sources: its own rough noise
$\beta^2$ at scale $H_2$, and an inherited noise from $V^1$ through the
cross-kernel $K_1*K_2$.
Since $K_1*K_2$ has exponent $\alpha_1+\alpha_2-1>\max(\alpha_1,\alpha_2)-1$,
the cross-noise is smoother than either self-noise: cross-excitation is
a coarser-scale phenomenon than self-excitation; this is made quantitative
in Proposition~\ref{prop:regularity}.
The functional correlation in the zero-start regime satisfies
$\varrho(t)\sim C_\varrho t^{(3\alpha_1-\alpha_2)/2}$.
The vanishing correlation is caused by the smoother cross-kernel together
with the different variance normalizations of the two components.
In the existing literature on correlated rough volatility
(e.g.\ the multidimensional models of~\cite{TomasRosenbaum2019} or the
multifactor models of~\cite{JMP2025}),
the correlation is typically taken to be a constant $\varrho\in(-1,1)$;
the time-dependent, vanishing correlation found here is a distinctive
consequence of the near-critical bivariate microstructure.
\end{remark}

\begin{table}[htbp]
\renewcommand{\arraystretch}{1.55}
\setlength{\tabcolsep}{7pt}
\centering
\caption{%
Summary of the main results.
The five contributions of the paper, with their location,
the key formula or object, and the proof technique.
All results hold under Assumption~\ref{ass:kernels}
with $\alpha_1<\alpha_2\in(1/2,1)$.%
}
\label{tab:summary}
\small
\begin{tabular}{@{} l l l l @{}}
\hline\hline
\textbf{Result} & \textbf{Location} & \textbf{Key formula / object} & \textbf{Method} \\
\hline
\begin{tabular}{@{}l@{}}
Functional limit\\
theorem
\end{tabular}
&
\begin{tabular}{@{}l@{}}
Theorem~\ref{thm:main-final}\\
(Sections~\ref{sec:tightness}--\ref{sec:uniqueness})
\end{tabular}
&
\begin{tabular}{@{}l@{}}
$(V^{T,1},V^{T,2})\Rightarrow(V^1,V^2)$\\
weakly in $C([0,1];\mathbb{R}^2)$
\end{tabular}
&
\begin{tabular}{@{}l@{}}
$C$-tightness + Rebolledo\\
+ ALP uniqueness
\end{tabular}
\\
\hline
\begin{tabular}{@{}l@{}}
Cross-kernel\\
identification
\end{tabular}
&
\begin{tabular}{@{}l@{}}
Proposition~\ref{prop:main}\\
(Section~\ref{sec:cross-kernel})
\end{tabular}
&
$L_{12} = b_\infty^{12}\|\psi^{12}\|_1\cdot(K_1*K_2)$
&
\begin{tabular}{@{}l@{}}
$L^2$ kernel convergence\\
+ It\^o isometry
\end{tabular}
\\
\hline
\begin{tabular}{@{}l@{}}
Covariance\\
formula
\end{tabular}
&
\begin{tabular}{@{}l@{}}
Proposition~\ref{prop:main}(iii)\\
(Section~\ref{ssec:main-prop})
\end{tabular}
&
\begin{tabular}{@{}l@{}}
$\mathrm{Cov}(V^1_t,V^2_t)$\\[2pt]
$= \gamma_{12}\nu_1^2\!\int_0^t\!K_1(t{-}s)L_{12}(t{-}s)\,b_1(s)\,ds$
\end{tabular}
&
\begin{tabular}{@{}l@{}}
$\beta^1\perp\beta^2$\\
+ It\^o isometry
\end{tabular}
\\
\hline
\begin{tabular}{@{}l@{}}
Bivariate\\
$C$-tightness
\end{tabular}
&
\begin{tabular}{@{}l@{}}
Proposition~\ref{prop:tightness}\\
(Section~\ref{sec:tightness})
\end{tabular}
&
$2(\alpha_1+\alpha_2)-1 > 2\alpha_2 - 1$
&
\begin{tabular}{@{}l@{}}
HXZ criterion\\
+ H\"older bound on $\mathcal{C}^T$
\end{tabular}
\\
\hline
\begin{tabular}{@{}l@{}}
Short-time\\
decorrelation
\end{tabular}
&
\begin{tabular}{@{}l@{}}
Corollary~\ref{cor:rho}\\
(Section~\ref{ssec:main-prop})
\end{tabular}
&
\begin{tabular}{@{}l@{}}
$\varrho(t)\sim C_\varrho\,t^{(3\alpha_1-\alpha_2)/2}$\\[2pt]
explicit constant; see Corollary~\ref{cor:rho}
\end{tabular}
&
\begin{tabular}{@{}l@{}}
Power-law asymptotics\\
of $\mathrm{Cov}$ and $\mathrm{Var}$
\end{tabular}
\\
\hline\hline
\end{tabular}
\end{table}

\subsection{On the scale-matching condition}
\label{ssec:A5}

Assumption~(A5) is used only to put the two components on a common
macroscopic time scale.  It is not a structural restriction on the
cross-kernel mechanism.  If (A5) is dropped, the same proof gives the same
limit system with the cross-kernel $L_{12}$ replaced by a deterministic
time-rescaled version
\[
L_{12}^{(\rho)}
= b_\infty^{12}\|\psi^{12}\|_1\bigl(K_1*K_2^{(\rho_{12})}\bigr),
\qquad
K_2^{(\rho_{12})}(t)=\rho_{12}^{-1}K_2(t/\rho_{12}),
\]
where
\[
\rho_{12}:=\lambda_1^{1/\alpha_1}/\lambda_2^{1/\alpha_2}.
\]
For readability we impose (A5) in the main text and defer the precise
statement and proof of the time-rescaled version to
Proposition~\ref{prop:A5} in Appendix~\ref{app:scale-matching}.

\section{Cross-kernel structure and functional correlation}
\label{sec:cross-kernel}

This section records the analytic mechanism behind the limiting dependence
between the two components.  The main message is simple: the triangular
microscopic excitation $1\to2$ survives in the limit through the deterministic
kernel
\[
L_{12}:=b_\infty^{12}\|\psi^{12}\|_1\,(K_1*K_2),
\]
and the covariance of the limiting processes is obtained by the It\^o isometry
from the pointwise product $K_1L_{12}$.  The detailed Fourier estimates and
constant computations are collected in Appendix~\ref{app:detailed-asymptotic-calculations};
we keep here only the statements needed for the main proof.

\subsection{Resolvent representation and kernel limits}
\label{ssec:analytic-lemmas}

Throughout this section we work under Assumption~\ref{ass:kernels}.  Set
\[
\varepsilon_T^i := (1-a_T^i)^{1/\alpha_i},\qquad i=1,2.
\]
Then $(1-a_T^i)=(\varepsilon_T^i)^{\alpha_i}$ and, by~(A2),
\begin{equation}\label{eq:scale-ratio}
\varepsilon_T^i\sim\lambda_i^{1/\alpha_i}T^{-1},
\qquad
\varepsilon_T^iT\longrightarrow\lambda^*:=\lambda_i^{1/\alpha_i}
\quad(i=1,2,\ \text{equal under (A5)}).
\end{equation}
We shall repeatedly use the following elementary dilation principle: if
$f_T\to f$ in $L^2(\mathbb R_+)$ and $\theta_T\to\lambda^*$, then, for every
$t>0$,
\begin{equation}\label{eq:dilation}
\int_0^{\theta_Tt}f_T(v)\,dv
\longrightarrow
\int_0^{\lambda^*t}f(v)\,dv
=\int_0^t\lambda^*f(\lambda^*w)\,dw .
\end{equation}
Thus the microscopic kernel limits $K_i^\circ$ produce the macroscopic
Mittag--Leffler kernels
\begin{equation}\label{eq:macro-kernels}
K_i(t):=\lambda^*K_i^\circ(\lambda^*t),
\qquad
L_{12}(t):=\lambda^*L_{12}^\circ(\lambda^*t)
=b_\infty^{12}\|\psi^{12}\|_1(K_1*K_2)(t).
\end{equation}

\begin{lemma}[Exact resolvent--martingale representation]\label{lem:rep}
Let
$M_t^{T,i}:=N_t^{T,i}-\int_0^t\lambda_s^{T,i}\,ds$.
Then $d\langle M^{T,i}\rangle_t=\lambda_t^{T,i}\,dt$,
$\langle M^{T,1},M^{T,2}\rangle\equiv0$, and
\begin{align}
\lambda_t^{T,1}
  &=\mu_T^1\Bigl(1+\int_0^t\psi_T^1(u)\,du\Bigr)
   +\int_0^t\psi_T^1(t-s)\,dM_s^{T,1},
   \label{eq:rep1}\\[2pt]
\lambda_t^{T,2}
  &=\E[\lambda_t^{T,2}]
   +\int_0^t\psi_T^2(t-s)\,dM_s^{T,2}
   +\int_0^t\Psi_T^{12}(t-s)\,dM_s^{T,1},
   \label{eq:rep2}
\end{align}
where
\begin{equation}\label{eq:Psi}
\widehat{\Psi_T^{12}}(z)
=\widehat{\phi_T^{12}}(z)
\frac{1}{1-a_T^1\widehat{\phi^1}(z)}
\frac{1}{1-a_T^2\widehat{\phi^2}(z)}
\end{equation}
and
\begin{equation}\label{eq:mean2}
\E[\lambda_t^{T,2}]
=\mu_T^2\Bigl(1+\int_0^t\psi_T^2\Bigr)
+\mu_T^1(\Psi_T^{12}*\mathbf 1)(t).
\end{equation}
Consequently the renormalized martingale parts are
\begin{equation}\label{eq:SC}
\mathcal S_t^{T,i}=
\frac{1-a_T^i}{m_iT^{\alpha_i-1}}
\int_0^{Tt}\psi_T^i(Tt-s)\,dM_s^{T,i},
\qquad
\mathcal C_t^{T}=\frac{1-a_T^2}{m_2T^{\alpha_2-1}}
\int_0^{Tt}\Psi_T^{12}(Tt-s)\,dM_s^{T,1}.
\end{equation}
Moreover,
\begin{equation}\label{eq:inv}
\psi_T^i(r)=\frac{\varepsilon_T^i}{1-a_T^i}
 g_T^i(\varepsilon_T^i r),
\qquad
\Psi_T^{12}(r)=
\frac{\varepsilon_T^1}{T^{2\alpha_1-\alpha_2}(1-a_T^1)(1-a_T^2)}
\tilde h_T(\varepsilon_T^1r).
\end{equation}
\end{lemma}

\begin{proof}
All assertions, including the bracket identities, the mean
formula~\eqref{eq:mean2}, and the rescaled and inverted
forms~\eqref{eq:SC}--\eqref{eq:inv}, are proved in
Appendix~\ref{app:detailed-asymptotic-calculations}
(Lemma~\ref{lem:rep-detail} and its proof).  The argument is elementary
and pathwise at each fixed time: it combines the resolvent identity
$\psi_T^i*\phi_T^i=\psi_T^i-\phi_T^i$ with Tonelli's theorem for
non-negative kernels and the compensator identity; first moments and
non-explosion are established there by a stopping argument.
\end{proof}

Define the renormalized self- and cross-resolvents by
\begin{align}
g_T^i(s)
  &:=\frac{1-a_T^i}{\varepsilon_T^i}\,
     \psi_T^i\!\left(\frac{s}{\varepsilon_T^i}\right),
  \qquad i=1,2,
  \label{eq:gTi-def}\\[3pt]
\tilde h_T(s)
  &:=T^{2\alpha_1-\alpha_2}
     \frac{(1-a_T^1)(1-a_T^2)}{\varepsilon_T^1}
     \Psi_T^{12}\!\left(\frac{s}{\varepsilon_T^1}\right).
  \label{eq:hT-def}
\end{align}

\begin{lemma}[Renormalized kernel limits]\label{lem:L2-self}\label{lem:L2-cross}
For each $i\in\{1,2\}$,
\begin{equation}\label{eq:L2-self}
g_T^i\longrightarrow K_i^\circ
\quad\text{in }L^2(\mathbb R_+).
\end{equation}
Under~(A5),
\begin{equation}\label{eq:L2-cross}
\tilde h_T\longrightarrow
L_{12}^\circ:=b_\infty^{12}\|\psi^{12}\|_1(K_1^\circ*K_2^\circ)
\quad\text{in }L^2(\mathbb R_+).
\end{equation}
Equivalently, after the dilation~\eqref{eq:macro-kernels}, the macroscopic
cross-kernel is
$L_{12}=b_\infty^{12}\|\psi^{12}\|_1(K_1*K_2)$.
Moreover, the Fourier transforms satisfy the uniform bounds
\begin{equation}\label{eq:gTi-bound}
|\widehat{g_T^i}(i\xi)|\leq\frac{C_i}{1+c_i|\xi|^{\alpha_i}},\qquad i=1,2,
\end{equation}
and
\begin{equation}\label{eq:hT-bound}
|\widehat{\tilde h_T}(i\xi)|\leq
\frac{C}{(1+c_1|\xi|^{\alpha_1})(1+c_2|\xi|^{\alpha_2})},
\end{equation}
with constants independent of $T$.
\end{lemma}

\begin{proof}
The proof is by Plancherel.  The pointwise transform limits follow from
$1-\widehat{\phi^i}(z)\sim\delta_i z^{\alpha_i}$ at the origin, while the
strong aperiodicity and absolute continuity assumptions give the uniform
Fourier domination needed for convergence in $L^2$.  The cross-kernel follows
by multiplying the two scalar resolvent limits and using the scaling
$T^{2\alpha_1-\alpha_2}b_T^{12}\to b_\infty^{12}$.  Full details are in
Appendix~\ref{app:detailed-asymptotic-calculations}.
\end{proof}

\begin{lemma}[Product convergence]\label{lem:L1-product}
The family $(\tilde h_T)_T$ is bounded in $L^2(\mathbb R_+)$ and the
convergence~\eqref{eq:L2-cross} holds.  In addition, for every $t>0$,
\begin{equation}\label{eq:L1-product}
\bigl\|g_T^1\tilde h_T-K_1^\circ L_{12}^\circ\bigr\|_{L^1([0,t])}
\longrightarrow0.
\end{equation}
\end{lemma}

\begin{proof}
The $L^1$ convergence is a direct Cauchy--Schwarz consequence of the two
$L^2$ convergences in Lemma~\ref{lem:L2-self}.  The same proof, after applying
\eqref{eq:dilation}, yields the corresponding macroscopic convergence with
$K_1$ and $L_{12}$.
\end{proof}

\begin{lemma}[First-moment bounds]\label{lem:moments-first}
The following estimates hold uniformly in $t\ge0$:
\begin{align}
\E[\lambda_t^{T,1}]&\le \frac{\mu_T^1}{1-a_T^1},
\label{eq:mean-upper-1}\\
\E[\lambda_t^{T,2}]&\le
\frac{\mu_T^2}{1-a_T^2}
+\frac{b_T^{12}\|\psi^{12}\|_1}{1-a_T^2}
\frac{\mu_T^1}{1-a_T^1}.
\label{eq:mean-upper-2}
\end{align}
Consequently $\sup_{T,t\in[0,1]}\E[V_t^{T,i}]<\infty$ for $i=1,2$.
\end{lemma}

\begin{proof}
Take expectations in the Hawkes equations and use the positivity of the
resolvents.  For component~1 the equation is scalar and gives
\eqref{eq:mean-upper-1}.  For component~2, applying the resolvent of
$\phi_T^2$ and then using~\eqref{eq:mean-upper-1} gives
\eqref{eq:mean-upper-2}.  The stated renormalized bounds follow from
(A2)--(A4) and the cross-scaling~\eqref{eq:A3-scaling}.
\end{proof}

\subsection{Cross-kernel, covariance and decorrelation}
\label{ssec:main-prop}

\begin{proposition}[Cross-kernel structure and covariance formula]
\label{prop:main}
Under Assumption~\ref{ass:kernels}, set
\[
\nu_i=(m_i\lambda_i)^{-1/2},
\qquad
\gamma_{12}=\frac{m_1}{m_2\lambda_1},
\qquad
L_{12}=b_\infty^{12}\|\psi^{12}\|_1(K_1*K_2).
\]
Then:
\begin{enumerate}[label=\textup{(\roman*)}]
\item $V^{T,2}$ admits the decomposition
\begin{equation}\label{eq:decomp}
V_t^{T,2}=D_t^T+\mathcal S_t^T+\mathcal C_t^T,
\end{equation}
where
\begin{equation}\label{eq:self-term}
\mathcal S_t^T=\frac{1-a_T^2}{m_2T^{\alpha_2-1}}
\int_0^{Tt}\psi_T^2(Tt-u)\,dM_u^{T,2},
\end{equation}
\begin{equation}\label{eq:cross-term}
\mathcal C_t^T=\frac{1-a_T^2}{m_2T^{\alpha_2-1}}
\int_0^{Tt}\Psi_T^{12}(Tt-u)\,dM_u^{T,1}.
\end{equation}
Moreover $\langle\mathcal S^T,\mathcal C^T\rangle\equiv0$.

\item The effective cross-kernel is
\begin{equation}\label{eq:L12}
L_{12}=b_\infty^{12}\|\psi^{12}\|_1(K_1*K_2).
\end{equation}
It belongs to $L^2(\mathbb R_+)$ and satisfies
\begin{equation}\label{eq:L12-asymp}
L_{12}(t)\sim C_{12}t^{\alpha_1+\alpha_2-1},
\qquad t\downarrow0,
\end{equation}
with
$C_{12}=b_\infty^{12}\|\psi^{12}\|_1/
(\tilde\delta_1\tilde\delta_2\Gamma(\alpha_1+\alpha_2))$.

\item The covariance of the limiting process is
\begin{equation}\label{eq:cov-exact-prop}
\mathrm{Cov}(V_t^1,V_t^2)
=\gamma_{12}\nu_1^2
\int_0^tK_1(t-s)L_{12}(t-s)b_1(s)\,ds.
\end{equation}
In the stationary comparison $b_1(s)\equiv\bar V^1$ this becomes
\begin{equation}\label{eq:cov-explicit}
\mathrm{Cov}(V_t^1,V_t^2)
=\gamma_{12}\nu_1^2\bar V^1
\int_0^tK_1(u)L_{12}(u)\,du.
\end{equation}
\end{enumerate}
\end{proposition}

\begin{proof}
The decomposition follows from Lemma~\ref{lem:rep}.  The $L^2$ convergence of
$\tilde h_T$ identifies $L_{12}$, while the asymptotic
\eqref{eq:L12-asymp} follows either from Karamata's theorem or directly from
Beta convolution of the two Mittag--Leffler singularities.  Finally,
$V^1$ and the inherited part of $V^2$ share only the Brownian motion
$\beta^1$, so the It\^o isometry gives the product integral
\eqref{eq:cov-exact-prop}.  The detailed constant calculation is in
Appendix~\ref{app:detailed-asymptotic-calculations}.
\end{proof}

\begin{remark}[Product integral versus triple convolution]
\label{rem:product-vs-conv}
The covariance formula contains
$\int_0^tK_1(u)L_{12}(u)\,du$, the integral of a pointwise product.  This is
not the triple convolution $(K_1*K_1*K_2)(t)$.  The two objects have the same
power-law order at the origin but different constants; the product integral
is the one dictated by the It\^o isometry.
\end{remark}

\begin{corollary}[Short-time decorrelation law]
\label{cor:rho}
In the zero-start scaling regime, the functional correlation
\[
\varrho(t):=\frac{\mathrm{Cov}(V_t^1,V_t^2)}
{\sqrt{\mathrm{Var}(V_t^1)\mathrm{Var}(V_t^2)}}
\]
satisfies
\begin{equation}\label{eq:rho-asymp}
\varrho(t)\sim C_\varrho t^{(3\alpha_1-\alpha_2)/2},
\qquad t\downarrow0,
\end{equation}
where
\begin{equation}\label{eq:C-rho}
C_\varrho
=
\frac{
\gamma_{12}\nu_1^2
b_\infty^{12}\|\psi^{12}\|_1
\kappa_1^2\kappa_2
\beta_1
\mathrm{B}(\alpha_1,\alpha_2)
\mathrm{B}(2\alpha_1+\alpha_2-1,\alpha_1+1)
}{
\sqrt{
\nu_1^2\kappa_1^2\beta_1\mathrm{B}(2\alpha_1-1,\alpha_1+1)
\,
\nu_2^2\kappa_2^2\beta_2\mathrm{B}(2\alpha_2-1,\alpha_2+1)
}
}.
\end{equation}
Here $\kappa_i=(\tilde\delta_i\Gamma(\alpha_i))^{-1}$ and
$\beta_i=\kappa_i/\alpha_i$.
\end{corollary}

\begin{proof}
Use $K_i(t)\sim\kappa_it^{\alpha_i-1}$,
$b_i(t)\sim\beta_it^{\alpha_i}$, and
$L_{12}(t)\sim b_\infty^{12}\|\psi^{12}\|_1\kappa_1\kappa_2
\mathrm B(\alpha_1,\alpha_2)t^{\alpha_1+\alpha_2-1}$ in the variance and
covariance integrals of Proposition~\ref{prop:main}.  The Mellin convolution
calculation is written out in Appendix~\ref{app:detailed-asymptotic-calculations}.
\end{proof}

\begin{remark}[Stationary comparison]\label{rem:stationary-comparison}
If the mean profiles are locally constant near zero, then the same calculation
gives $\varrho(t)\sim C_\varrho^{\mathrm{stat}}t^{\alpha_1}$.  This is a
different regime from the zero-start scaling used in the main theorem.
\end{remark}

\section{Analytic and numerical illustrations}
\label{sec:numerics}

This short section gives two checks on the formulas of
Section~\ref{sec:cross-kernel}.  It is not an empirical calibration section:
its purpose is only to make the exponent and the role of the cross-kernel
transparent before the proof of tightness and identification begins.

\subsection{A numerical quadrature check}

We first verify the short-time exponent at the level of the limiting covariance
and variance formulas.  For fixed
$\alpha_1,\alpha_2\in(1/2,1)$, set, up to irrelevant multiplicative constants,
\[
K_1(t)=t^{\alpha_1-1},\qquad K_2(t)=t^{\alpha_2-1},
\qquad L_{12}(t)=(K_1*K_2)(t).
\]
We then compute directly, by one-dimensional quadrature,
\[
\operatorname{Var}(V_t^1)\simeq
\int_0^t K_1(t-u)^2u^{\alpha_1}\,du,
\qquad
\operatorname{Var}(V_t^2)\simeq
\int_0^t K_2(t-u)^2u^{\alpha_2}\,du,
\]
\[
\operatorname{Cov}(V_t^1,V_t^2)\simeq
\int_0^t K_1(t-u)L_{12}(t-u)u^{\alpha_1}\,du,
\]
and estimate the slope of $\log\varrho(t)$ as a function of $\log t$ for
small $t$.  The resulting slopes agree with the theoretical value
$(3\alpha_1-\alpha_2)/2$; see Table~\ref{tab:numerical-exponent-verification}.
Some rows are included only as formula diagnostics outside the standing
ordering convention $\alpha_1<\alpha_2$.

\begin{table}[htbp]
\centering
\begin{tabular}{c c c c c}
\hline
$\alpha_1$ & $\alpha_2$ & Theory & Numerical & Absolute error \\
\hline
0.60 & 0.60 & 0.600000 & 0.600000 & $4.3\times 10^{-8}$ \\
0.65 & 0.75 & 0.600000 & 0.599998 & $1.8\times 10^{-6}$ \\
0.55 & 0.90 & 0.375000 & 0.374996 & $4.4\times 10^{-6}$ \\
0.80 & 0.60 & 0.900000 & 0.899990 & $1.0\times 10^{-5}$ \\
0.51 & 0.99 & 0.270000 & 0.269996 & $4.1\times 10^{-6}$ \\
\hline
\end{tabular}
\caption{Numerical quadrature check of the zero-start short-time exponent
$(3\alpha_1-\alpha_2)/2$.  The quadrature uses the limiting covariance and
variance integrals, not a simulation of the microscopic Hawkes process.}
\label{tab:numerical-exponent-verification}
\end{table}

As a robustness check, the same computation over 100 random pairs in
$(0.51,0.99)^2$ gives a root mean squared error of order $10^{-5}$ between
the numerical and theoretical slopes.  This confirms the Mellin-convolution
asymptotics used in Corollary~\ref{cor:rho}.

\subsection{Interpretation of the parameter regime}

The empirically relevant rough-volatility range $H\approx0.1$--$0.3$
corresponds to $\alpha=H+1/2\approx0.6$--$0.8$
\cite{GJR2018}.  For instance, taking
$\alpha_1=0.6$ and $\alpha_2=0.8$ gives
\[
\frac{3\alpha_1-\alpha_2}{2}=0.5.
\]
Thus, in the zero-start regime, the limiting functional correlation behaves
as $\varrho(t)\asymp t^{1/2}$ at small times.  The same pair of parameters also
makes the regularity hierarchy transparent:
\[
K_1(t)\sim t^{-0.4},\qquad K_2(t)\sim t^{-0.2},\qquad
(K_1*K_2)(t)\sim t^{0.4}.
\]
The cross-kernel is therefore smoother than either self-kernel.  This is the
analytic reason why cross-excitation creates a non-trivial dependence between
the two components without destroying the local roughness of the second
component.

\section{\texorpdfstring{$C$}{C}-tightness of the bivariate intensity processes}
\label{sec:tightness}

This section proves tightness of the renormalized intensity vector in
$C([0,1];\mathbb R^2)$.  The proof is deliberately separated from the
identification argument: tightness uses only kernel estimates, moment bounds,
and the triangular structure of the cross-excitation.

We use the decomposition obtained from the resolvent representation,
\[
V^{T,1}=D^{T,1}+\mathcal S^{T,1},\qquad
V^{T,2}=D^{T,2}+\mathcal S^{T,2}+\mathcal C^T,
\]
where $D^{T,i}$ are deterministic terms, $\mathcal S^{T,i}$ are the
self-excited martingale convolutions, and $\mathcal C^T$ is the martingale
convolution inherited by component~2 from component~1.  The deterministic
terms converge locally uniformly to the profiles $b_i$ defined in
Section~\ref{sec:setup}; therefore only the stochastic convolutions require
work.

\subsection{A stochastic-convolution tightness criterion}

We shall use the following standard tightness criterion for stochastic
convolutions.  It is the version of the criterion of
Horst--Xu--Zhang~\cite{HXZ2023} needed below; its proof is the usual
BDG--Kolmogorov argument, together with the vanishing-jump estimates for the
renormalized compensated counting martingales.

\begin{proposition}[Stochastic-convolution tightness criterion]
\label{prop:HXZ}
Let $M^T$ be square-integrable compensated point-process martingales on
$[0,1]$, with predictable brackets
\[
 d\langle M^T\rangle_t=q_t^T\,dt,
\]
and assume that the jumps of the renormalized martingales vanish in
probability uniformly on $[0,1]$.  Let
\[
I_t^T=\int_0^t K_T(t-s)\,dM_s^T.
\]
Suppose that for some $\eta>0$,
\begin{equation}\label{eq:L2-shift}
\sup_T\int_0^\infty |K_T(u+h)-K_T(u)|^2\,du
\le C h^\eta,
\qquad h\in(0,1],
\end{equation}
and that, for some $p>1/\eta$,
\begin{equation}\label{eq:moment-density}
\sup_T\E\left[\sup_{t\in[0,1]} |q_t^T|^p\right]<\infty .
\end{equation}
Then $(I^T)_T$ is tight in $D([0,1])$, and every limit point is continuous.
Equivalently, $(I^T)_T$ is $C$-tight.
\end{proposition}

\begin{proof}[Proof sketch]
For $0\le s\le t\le1$, write
\[
I_t^T-I_s^T
=\int_s^t K_T(t-r)\,dM_r^T
+\int_0^s\bigl(K_T(t-r)-K_T(s-r)\bigr)\,dM_r^T .
\]
The Burkholder--Davis--Gundy inequality and H\"older's inequality give, at
moment order $2p$,
\[
\E\bigl[|I_t^T-I_s^T|^{2p}\bigr]
\le C |t-s|^{\eta p},
\]
using~\eqref{eq:L2-shift} and~\eqref{eq:moment-density}.  Since
$\eta p>1$, Kolmogorov--Chentsov gives tightness with continuous limiting
paths.  The point-process jump contribution is negligible because the jumps
of the renormalized compensated martingales vanish uniformly in probability;
this is the same argument as in~\cite{HXZ2023}.
\end{proof}

\subsection{Kernel and moment inputs}

The next two lemmas collect the two inputs needed to apply
Proposition~\ref{prop:HXZ}.

\begin{lemma}[Uniform $L^2$-shift bounds]
\label{lem:L2-shift-pre}
For each $i\in\{1,2\}$, the renormalized self-kernels satisfy
\begin{equation}\label{eq:self-shift-bound}
\sup_T\int_0^\infty |g_T^i(u+h)-g_T^i(u)|^2\,du
\le C h^{2\alpha_i-1},
\qquad h\in(0,1].
\end{equation}
The renormalized cross-kernels satisfy
\begin{equation}\label{eq:cross-shift-bound}
\sup_T\int_0^\infty |\tilde h_T(u+h)-\tilde h_T(u)|^2\,du
\le C h^{\eta_{12}},
\qquad h\in(0,1],
\end{equation}
where
\[
\eta_{12}:=\min\{2(\alpha_1+\alpha_2)-1,\,2\}>1.
\]
\end{lemma}

\begin{proof}
By Plancherel,
\[
\int_{\mathbb R}|g_T^i(u+h)-g_T^i(u)|^2\,du
=\frac1{2\pi}\int_{\mathbb R}|e^{ih\xi}-1|^2
|\widehat g_T^i(i\xi)|^2\,d\xi .
\]
The uniform Fourier bound~\eqref{eq:gTi-bound} and the inequality
$|e^{ih\xi}-1|\le \min(2,h|\xi|)$ give~\eqref{eq:self-shift-bound} by splitting
the integral at $|\xi|=1/h$.
The proof of~\eqref{eq:cross-shift-bound} is identical, using the product
bound~\eqref{eq:hT-bound}.  If $\alpha_1+\alpha_2\le3/2$, the exponent is
$2(\alpha_1+\alpha_2)-1$; if $\alpha_1+\alpha_2>3/2$, the low-frequency
part gives the sharper bound $Ch^2$.  This yields the exponent
$\eta_{12}=\min\{2(\alpha_1+\alpha_2)-1,2\}$.
\end{proof}

The shift bounds above control increments in $L^2$ once the integrands are
known to have bounded moments.  We record the required uniform moment bounds
next.

\begin{lemma}[Uniform moment bounds]
\label{lem:sup-bound}
For every $p\ge1$,
\begin{equation}\label{eq:uniform-moments}
\sup_T\E\left[\sup_{t\in[0,1]} (V_t^{T,1})^p\right]
+
\sup_T\E\left[\sup_{t\in[0,1]} (V_t^{T,2})^p\right]
<\infty .
\end{equation}
\end{lemma}

\begin{proof}
We give the argument because this is the only point in the tightness proof
where the triangular structure of the model enters.  Throughout the proof
$C_p$ denotes a constant independent of $T$ and of $t\in[0,1]$.

We use the following standard maximal estimate for Volterra stochastic
convolutions, which follows from the BDG inequality, the $L^2$ bounds and the
shift bounds on the kernels; it is the estimate used in
\cite{HXZ2023}, Lemma~3.1.  If
\[
Y_t^T=\int_0^t K_T(t-s)\,d\widehat M_s^T,
\qquad d\langle\widehat M^T\rangle_s=q_s^T\,ds,
\]
where $K_T$ is one of the families $g_T^i$ or $\tilde h_T$, then for every
$p\ge1$,
\begin{equation}\label{eq:volterra-maximal-moment}
\E\left[\sup_{r\le t}|Y_r^T|^p\right]
\le C_p\left(1+\int_0^t
\E\left[\sup_{u\le s}|q_u^T|^p\right]ds\right).
\end{equation}
The important point is that the constant is uniform in $T$; this uniformity
comes from Lemmas~\ref{lem:L2-self}, \ref{lem:L2-cross}, and
\ref{lem:L2-shift-pre}.  Estimate~\eqref{eq:volterra-maximal-moment} is the
BDG--Kolmogorov estimate behind Proposition~\ref{prop:HXZ}, written here in
maximal-moment form.

\medskip
\noindent\textbf{First component.}
The first component is autonomous.  In normalized form,
\[
V^{T,1}=D^{T,1}+\mathcal S^{T,1},
\]
where $D^{T,1}$ is uniformly bounded on $[0,1]$, and
$\mathcal S^{T,1}$ is a stochastic convolution with kernel $g_T^1$ driven by
$\widehat M^{T,1}$, whose bracket density is $V^{T,1}$.  Hence, defining
\[
F_1^T(t):=\E\left[\sup_{r\le t}(V_r^{T,1})^p\right],
\]
we get from~\eqref{eq:volterra-maximal-moment}
\[
F_1^T(t)
\le C_p + C_p\int_0^t F_1^T(s)\,ds .
\]
Gronwall's lemma gives
\[
\sup_T F_1^T(1)<\infty .
\]
This is precisely the scalar moment estimate for nearly unstable
heavy-tailed Hawkes intensities; the above display recalls how the estimate
is used in the present normalization.

\medskip
\noindent\textbf{Second component.}
We use the triangular decomposition
\[
V^{T,2}=D^{T,2}+\mathcal S^{T,2}+\mathcal C^T .
\]
The deterministic term $D^{T,2}$ is uniformly bounded on $[0,1]$ by
Assumptions~(A2)--(A4) and the convergence of the deterministic profiles.
The self term $\mathcal S^{T,2}$ is a stochastic convolution with kernel
$g_T^2$ and bracket density $V^{T,2}$.  Therefore, with
\[
F_2^T(t):=\E\left[\sup_{r\le t}(V_r^{T,2})^p\right],
\]
\eqref{eq:volterra-maximal-moment} gives
\begin{equation}\label{eq:self-moment-v2}
\E\left[\sup_{r\le t}|\mathcal S_r^{T,2}|^p\right]
\le C_p + C_p\int_0^t F_2^T(s)\,ds .
\end{equation}

The cross term is different but easier because its bracket involves only the
autonomous first component.  Indeed,
\[
\mathcal C_t^T=\int_0^t \tilde h_T(t-s)\,d\widehat M_s^{T,1},
\qquad d\langle\widehat M^{T,1}\rangle_s=V_s^{T,1}\,ds .
\]
Applying~\eqref{eq:volterra-maximal-moment} with $K_T=\tilde h_T$ and using
the first-component bound already proved yields
\begin{equation}\label{eq:cross-moment-v2}
\E\left[\sup_{r\le t}|\mathcal C_r^T|^p\right]
\le C_p + C_p\int_0^t F_1^T(s)\,ds
\le C_p .
\end{equation}
Combining the deterministic bound with~\eqref{eq:self-moment-v2} and
\eqref{eq:cross-moment-v2}, we obtain
\[
F_2^T(t)
\le C_p + C_p\int_0^t F_2^T(s)\,ds .
\]
A second application of Gronwall's lemma gives
\[
\sup_T F_2^T(1)<\infty .
\]
This proves~\eqref{eq:uniform-moments}.  The argument closes because the
cross martingale is one-sided: its bracket is controlled by $V^{T,1}$, which
has already been estimated.  In a fully mutually exciting system the same
step would require joint moment estimates for the coupled intensities.
\end{proof}

\subsection{The main tightness result}

\begin{proposition}[$C$-tightness]
\label{prop:tightness}
Under Assumption~\ref{ass:kernels}, the sequence
$(V^{T,1},V^{T,2})_{T\ge1}$ is $C$-tight in
$C([0,1];\mathbb R^2)$.
\end{proposition}

\begin{proof}
It is enough to prove $C$-tightness of the two coordinates separately.

\medskip
\noindent\textbf{First component.}
The martingale part of $V^{T,1}$ is, up to deterministic constants already
absorbed in the normalization, a stochastic convolution with kernel $g_T^1$
and bracket density comparable to $V^{T,1}$.  By
Lemma~\ref{lem:L2-shift-pre}, the shift exponent is $2\alpha_1-1>0$.
Choose $p>1/(2\alpha_1-1)$.  Lemma~\ref{lem:sup-bound} gives the required
moment bound~\eqref{eq:moment-density}.  Proposition~\ref{prop:HXZ} gives
$C$-tightness of the martingale part, and the deterministic part converges
uniformly.  Hence $(V^{T,1})_T$ is $C$-tight.

\medskip
\noindent\textbf{Second component: self part.}
The self martingale part $\mathcal S^{T,2}$ is treated in the same way, with
kernel $g_T^2$ and shift exponent $2\alpha_2-1>0$.  Taking
$p>1/(2\alpha_2-1)$ and using Lemma~\ref{lem:sup-bound},
Proposition~\ref{prop:HXZ} yields $C$-tightness of $(\mathcal S^{T,2})_T$.

\medskip
\noindent\textbf{Second component: cross part.}
The cross martingale part has the representation
\begin{equation}\label{eq:CT-renorm-rep}
\mathcal C_t^T
=\int_0^t \tilde h_T(t-s)\,d\widetilde M_s^{T,1},
\end{equation}
where the bracket density of $\widetilde M^{T,1}$ is proportional to
$V^{T,1}$.  Applying It\^o's isometry and Lemma~\ref{lem:sup-bound}(i) gives,
for $0\le s\le t\le1$ and $h=t-s$,
\begin{align}
\E\bigl[|\mathcal C_t^T-\mathcal C_s^T|^2\bigr]
&\le C\left(
\int_0^\infty |\tilde h_T(u+h)-\tilde h_T(u)|^2\,du
+\int_0^h \tilde h_T(u)^2\,du
\right)\notag\\
&\le C h^{\eta_{12}},
\label{eq:cross-holder}
\end{align}
where $\eta_{12}=\min\{2(\alpha_1+\alpha_2)-1,2\}>1$.  Therefore the cross
part is $C$-tight by Kolmogorov--Chentsov.  Notice that this exponent is
strictly larger than $2\alpha_2-1$, so the cross-excitation is smoother than
the self-excitation of the second component.

\medskip
\noindent\textbf{Conclusion.}
The deterministic part $D^{T,2}$ converges uniformly, the self part is
$C$-tight, and the cross part is $C$-tight.  Hence $(V^{T,2})_T$ is
$C$-tight.  Tightness of the two coordinates implies tightness of the vector
in the product space $C([0,1];\mathbb R^2)$.
\end{proof}

\begin{remark}[Role of triangularity in tightness]
The only place where the triangular structure is used in this section is the
cross estimate~\eqref{eq:cross-holder}.  Since component~1 is autonomous, the
bracket of the cross martingale is controlled by $V^{T,1}$ alone.  With a
feedback term $2\to1$, the analogous estimate would involve coupled moments
of both intensities, and the above one-sided argument would no longer close.
\end{remark}

\section{Identification of limit points as solutions of the Volterra system}
\label{sec:identification}

Having established $C$-tightness, we now identify every subsequential limit.
The only point requiring care is that the compensated Hawkes martingales are
pure-jump martingales before the limit.  Their jumps vanish under the
normalization, so Rebolledo's theorem yields continuous limiting martingales.
The constants $\nu_1$, $\nu_2$ and $\gamma_{12}$ enter through the scalar
normalizations of the stochastic convolution kernels.

\subsection{Renormalized martingales}

For $i=1,2$ define the martingales
\begin{equation}\label{eq:Mhat-def}
\widehat M_t^{T,i}
:=\left(\frac{1-a_T^i}{m_iT^{\alpha_i}}\right)^{1/2}M_{Tt}^{T,i},
\qquad t\in[0,1].
\end{equation}
Then
\begin{equation}\label{eq:Mhat-bracket}
\langle \widehat M^{T,i}\rangle_t
=\int_0^t V_s^{T,i}\,ds,
\qquad
\sup_{t\le1}|\Delta \widehat M_t^{T,i}|
=\left(\frac{1-a_T^i}{m_iT^{\alpha_i}}\right)^{1/2}\longrightarrow0 .
\end{equation}
Indeed, each jump of the counting martingale $M^{T,i}$ has size one
(the compensator is continuous), so the deterministic prefactor in
\eqref{eq:Mhat-def} is exactly the jump size of $\widehat M^{T,i}$.
Since $1-a_T^i\sim\lambda_iT^{-\alpha_i}$, this jump size is of order
$T^{-\alpha_i}$ and hence vanishes.  Thus the Lindeberg condition in
Rebolledo's theorem is automatic.  Moreover
$[\widehat M^{T,1},\widehat M^{T,2}]\equiv0$, because the two counting
processes have no simultaneous jumps in the standard Poisson embedding of
the Hawkes system.

\begin{lemma}[Rebolledo limit of the driving martingales]
\label{lem:rebolledo}
Let $(V^1,V^2)$ be a limit point of $(V^{T,1},V^{T,2})$ in
$C([0,1];\mathbb R^2)$.  Along the same subsequence, the martingales
$(\widehat M^{T,1},\widehat M^{T,2})$ are tight in $D([0,1];\mathbb R^2)$, and
any joint limit $(W^1,W^2)$ is a pair of continuous local martingales satisfying
\begin{equation}\label{eq:W-brackets}
\langle W^i\rangle_t=\int_0^t V_s^i\,ds,
\qquad
\langle W^1,W^2\rangle_t=0.
\end{equation}
Consequently, possibly after enlarging the probability space, there exist two
independent Brownian motions $\beta^1,\beta^2$ such that
\begin{equation}\label{eq:W-brownian-rep}
W_t^i=\int_0^t\sqrt{V_s^i}\,d\beta_s^i,
\qquad i=1,2.
\end{equation}
\end{lemma}

\begin{proof}
We verify the hypotheses of the martingale functional central limit theorem.
First, by~\eqref{eq:Mhat-bracket} and Lemma~\ref{lem:sup-bound},
\[
\sup_T\E\bigl[\langle \widehat M^{T,i}\rangle_1^p\bigr]
\le
\sup_T\E\left[\left(\int_0^1V_s^{T,i}\,ds\right)^p\right]<\infty
\]
for every fixed $p\ge1$.  Consequently Aldous' criterion applies to the
martingales $\widehat M^{T,i}$: for stopping times $\tau_T\le1$ and
$\theta_T\downarrow0$,
\[
\E\left[\left|\widehat M^{T,i}_{(\tau_T+\theta_T)\wedge1}
        -\widehat M^{T,i}_{\tau_T}\right|^2\right]
=
\E\left[\int_{\tau_T}^{(\tau_T+\theta_T)\wedge1}V_s^{T,i}\,ds\right]
\le
\theta_T\,\sup_T\E\left[\sup_{s\le1}V_s^{T,i}\right]
\longrightarrow0.
\]
Together with the boundedness of
$\E[\langle\widehat M^{T,i}\rangle_1]$, this gives tightness in
$D([0,1])$.

Second, along a subsequence on which $V^{T,i}\to V^i$ in law in
$C([0,1])$, Skorokhod representation and the uniform integrability supplied
by Lemma~\ref{lem:sup-bound} yield
\[
\sup_{t\le1}\left|\langle\widehat M^{T,i}\rangle_t
-\int_0^tV_s^i\,ds\right|
=\sup_{t\le1}\left|\int_0^t(V_s^{T,i}-V_s^i)\,ds\right|
\longrightarrow0
\]
in probability.  Finally, the jump estimate in~\eqref{eq:Mhat-bracket}
shows that
$\sup_{t\le1}|\Delta\widehat M_t^{T,i}|\to0$ deterministically.  Hence the
Lindeberg condition holds, and Rebolledo's theorem gives a continuous local
martingale limit $W^i$ with
$\langle W^i\rangle_t=\int_0^tV_s^i\,ds$.

For the joint limit, the pre-limit cross-bracket is identically zero:
$[\widehat M^{T,1},\widehat M^{T,2}]\equiv0$. Since the limiting martingales
are continuous, convergence of quadratic covariations gives
$\langle W^1,W^2\rangle\equiv0$.  The representation
\eqref{eq:W-brownian-rep} then follows from the Dambis--Dubins--Schwarz
representation, with the usual enlargement on the set where $V^i=0$.  By
Knight's theorem for strongly orthogonal continuous local martingales, the
Brownian drivers may be chosen independent.
\end{proof}

\subsection{Stable convergence of the stochastic convolutions}

For the identification step it is useful to absorb all deterministic scalar
normalizations into the kernels.  Define deterministic kernels
$\mathfrak K_T^i$ and $\mathfrak L_T^{12}$ by writing the martingale parts as
\begin{align}
\mathcal S_t^{T,i}
&=\int_0^t \mathfrak K_T^i(t-s)\,d\widehat M_s^{T,i},
\qquad i=1,2,\label{eq:renorm-self-kernel}\\
\mathcal C_t^{T}
&=\int_0^t \mathfrak L_T^{12}(t-s)\,d\widehat M_s^{T,1}.
\label{eq:renorm-cross-kernel}
\end{align}
Equivalently, $\mathfrak K_T^i$ is the dilated version of $g_T^i$ multiplied by
its scalar normalization, and $\mathfrak L_T^{12}$ is the dilated version of
$\tilde h_T$ multiplied by its scalar normalization.  The calculations of
Appendix~\ref{app:detailed-asymptotic-calculations} give
\begin{equation}\label{eq:macro-kernel-limits-identification}
\mathfrak K_T^i\longrightarrow \nu_iK_i
\quad\text{in }L^2([0,1]),
\qquad
\mathfrak L_T^{12}\longrightarrow \gamma_{12}\nu_1L_{12}
\quad\text{in }L^2([0,1]),
\end{equation}
with the corresponding uniform $L^2$-shift bounds inherited from
Lemma~\ref{lem:L2-shift-pre}.

\begin{lemma}[Stability of deterministic stochastic convolutions]
\label{lem:stoch-conv}
Let $K_T\to K$ in $L^2([0,1])$, and assume that $K_T$ satisfies the uniform
$L^2$-shift bound~\eqref{eq:L2-shift}.  Let $M^T$ be square-integrable
martingales whose jumps vanish in probability, and suppose that
$(M^T,\langle M^T\rangle)$ converges in law to $(M,\langle M\rangle)$ in the
Skorokhod topology, where $M$ is continuous.  If
$d\langle M^T\rangle_t=q_t^Tdt$ and the densities $q^T$ satisfy the moment
bound~\eqref{eq:moment-density}, then
\begin{equation}\label{eq:stoch-conv-stability}
\int_0^\cdot K_T(\cdot-s)\,dM_s^T
\Longrightarrow
\int_0^\cdot K(\cdot-s)\,dM_s
\qquad\text{in }C([0,1]).
\end{equation}
\end{lemma}

\begin{proof}
We give the argument because this is the only stability input used in the
identification step.  Since the limiting martingale is continuous, convergence
in the Skorokhod topology is equivalent, along a Skorokhod representation, to
uniform convergence on $[0,1]$.  We therefore argue on such a representation
and prove convergence in probability; the conclusion in law follows.

Write
\[
I_t^T:=\int_0^tK_T(t-s)\,dM_s^T,
\qquad
I_t:=\int_0^tK(t-s)\,dM_s .
\]
We split
\[
I_t^T-I_t
= A_t^T+B_t^T,
\qquad
A_t^T:=\int_0^t(K_T-K)(t-s)\,dM_s^T,
\]
where $B^T$ denotes the same expression with the fixed kernel $K$ and the
martingales $M^T$ and $M$.

\medskip
\noindent\emph{Step 1: the kernel error.}
For every fixed $t$, It\^o's isometry gives
\[
\E\bigl[|A_t^T|^2\bigr]
=\E\left[\int_0^t |K_T(t-s)-K(t-s)|^2 q_s^T\,ds\right]
\le C\|K_T-K\|_{L^2([0,1])}^2,
\]
where $C:=\sup_T\sup_{s\le1}\E[q_s^T]<\infty$, a consequence of
\eqref{eq:moment-density}.  Hence $A_t^T\to0$ in $L^2$ for each fixed $t$.
Moreover, the family $(A^T)_T$ is $C$-tight by Proposition~\ref{prop:HXZ}
applied to the kernel family $K_T-K$: the shift estimate follows from the
shift estimates for $K_T$ and $K$ (the latter by Fatou's lemma from
\eqref{eq:L2-shift}).  Every subsequential limit of $A^T$ is therefore
continuous and has all finite-dimensional distributions equal to zero; hence
$A^T\to0$ in probability in $C([0,1])$.

\medskip
\noindent\emph{Step 2: reduction of the fixed-kernel term to bounded-variation
kernels.}
Let $(K^{(n)})_{n\ge1}\subset C^1([0,1])$ be a sequence of deterministic
bounded-variation kernels such that $K^{(n)}\to K$ in $L^2([0,1])$ and, after
possibly replacing the shift exponent by any smaller positive exponent, the
BDG--Kolmogorov estimate of Proposition~\ref{prop:HXZ} applies uniformly to
$K-K^{(n)}$.  Such an approximation is obtained by truncating the singularity
near zero and then mollifying.
Define
\[
I_t^{T,n}:=\int_0^tK^{(n)}(t-s)\,dM_s^T,
\qquad
I_t^n:=\int_0^tK^{(n)}(t-s)\,dM_s .
\]
The same BDG estimate as in Step~1 gives
\begin{equation}\label{eq:fixed-kernel-trunc-error}
\lim_{n\to\infty}\,\limsup_{T\to\infty}
\mathbb P\Bigl(\sup_{t\le1}|I_t^T-I_t^{T,n}|>\varepsilon\Bigr)=0,
\qquad
\sup_{t\le1}|I_t-I_t^n|\xrightarrow[n\to\infty]{\mathbb P}0 .
\end{equation}
Indeed, the bracket densities are uniformly controlled by
\eqref{eq:moment-density}, while $K^{(n)}\to K$ in $L^2$ and the approximation
is chosen so that the corresponding shift seminorm tends to zero.

\medskip
\noindent\emph{Step 3: convergence for bounded-variation kernels.}
For fixed $n$, the stochastic integral against $K^{(n)}$ is a pathwise
continuous functional of the martingale path.  Since $M_0^T=M_0=0$, integration
by parts gives, for every c\`adl\`ag path $x$,
\[
\int_0^tK^{(n)}(t-s)\,dx_s
=K^{(n)}(0)x_t+
\int_0^t x_s\,(K^{(n)})'(t-s)\,ds .
\]
Consequently,
\[
\sup_{t\le1}\left|\int_0^tK^{(n)}(t-s)\,d(M_s^T-M_s)\right|
\le
\Bigl(\|K^{(n)}\|_\infty+\|(K^{(n)})'\|_{L^1}\Bigr)
\sup_{s\le1}|M_s^T-M_s|,
\]
which tends to zero in probability on the Skorokhod representation.  Hence
$I^{T,n}\to I^n$ in probability in $C([0,1])$ for every fixed $n$.

Combining Step~2 and Step~3 yields $B^T\to0$ in probability in
$C([0,1])$.  Together with Step~1, this proves
$I^T\to I$ in probability on the Skorokhod representation, and therefore the
convergence in law stated in~\eqref{eq:stoch-conv-stability}.  The limiting
process is continuous by the same tightness estimate and by the continuity of
$M$.
\end{proof}

\subsection{Main identification theorem}

The stability lemma lets us pass the defining equations to the limit along
the convergent subsequence.  We can now identify every limit point as a
solution of the Volterra system.

\begin{theorem}[Identification of limit points]
\label{thm:identification}
Under Assumption~\ref{ass:kernels}, every limit point $(V^1,V^2)$ of
$(V^{T,1},V^{T,2})$ in the path space $C([0,1];\mathbb R^2)$ satisfies
\begin{align}
V_t^1
&=b_1(t)+\nu_1\int_0^tK_1(t-s)\sqrt{V_s^1}\,d\beta_s^1,
\label{eq:limit1}\\
V_t^2
&=b_2(t)+\nu_2\int_0^tK_2(t-s)\sqrt{V_s^2}\,d\beta_s^2
+\gamma_{12}\nu_1\int_0^tL_{12}(t-s)\sqrt{V_s^1}\,d\beta_s^1,
\label{eq:limit2}
\end{align}
where $\beta^1$ and $\beta^2$ are independent Brownian motions.
\end{theorem}

\begin{proof}
Let $(V^1,V^2)$ be a subsequential limit, and let $(W^1,W^2)$ be the limiting
martingales supplied by Lemma~\ref{lem:rebolledo}.  The deterministic parts of
$V^{T,1}$ and $V^{T,2}$ converge locally uniformly to $b_1$ and $b_2$ by the
mean-profile calculation in Appendix~\ref{app:detailed-asymptotic-calculations}.
It remains only to pass to the limit in the martingale convolutions.

For the first component, combine~\eqref{eq:renorm-self-kernel}, the convergence
$\mathfrak K_T^1\to\nu_1K_1$ from
\eqref{eq:macro-kernel-limits-identification}, and
Lemma~\ref{lem:stoch-conv}.  This gives
\[
\mathcal S_t^{T,1}\Longrightarrow
\nu_1\int_0^tK_1(t-s)\,dW_s^1
=\nu_1\int_0^tK_1(t-s)\sqrt{V_s^1}\,d\beta_s^1,
\]
which proves~\eqref{eq:limit1}.

For the second component, the self martingale part is identical, using
$\mathfrak K_T^2\to\nu_2K_2$, and yields
\[
\mathcal S_t^{T,2}\Longrightarrow
\nu_2\int_0^tK_2(t-s)\sqrt{V_s^2}\,d\beta_s^2.
\]
The cross martingale part uses~\eqref{eq:renorm-cross-kernel} and
$\mathfrak L_T^{12}\to\gamma_{12}\nu_1L_{12}$, again with
Lemma~\ref{lem:stoch-conv}, and yields
\[
\mathcal C_t^T\Longrightarrow
\gamma_{12}\nu_1\int_0^tL_{12}(t-s)\,dW_s^1
=\gamma_{12}\nu_1\int_0^tL_{12}(t-s)\sqrt{V_s^1}\,d\beta_s^1.
\]
Combining the deterministic part, the self martingale part and the cross
martingale part gives~\eqref{eq:limit2}.
\end{proof}

\begin{remark}[Why the triangular structure is visible in the limit]
The first equation is autonomous, whereas the second equation contains an
additional noise inherited from component~1.  The two Brownian motions are
independent; the dependence between $V^1$ and $V^2$ is generated by the shared
$\beta^1$ term in~\eqref{eq:limit2}, transported through the cross-kernel
$L_{12}=b_\infty^{12}\|\psi^{12}\|_1(K_1*K_2)$.
\end{remark}

\begin{remark}[Short-time order of the deterministic profiles]
The mean profile $b_1(t)=\int_0^tK_1(w)\,dw$ satisfies
$b_1(t)\sim \kappa_1t^{\alpha_1}/\alpha_1$ as $t\downarrow0$.  Similarly,
\[
 b_2(t)=\int_0^tK_2(w)\,dw+\gamma_{12}\int_0^tL_{12}(w)\,dw
\]
has leading order $\kappa_2t^{\alpha_2}/\alpha_2$, since the cross-compensator
is of higher order $t^{\alpha_1+\alpha_2}$.
\end{remark}

\section{Uniqueness for the limiting Volterra system}
\label{sec:uniqueness}

This section proves uniqueness for the triangular limiting system identified in
Section~\ref{sec:identification}.  The point of the argument is simple.  The
first equation is scalar and autonomous, hence its law is fixed by a scalar
uniqueness theorem proved here in a self-contained form, in the spirit of the
affine Volterra theory of Abi Jaber--Larsson--Pulido~\cite{ALP2019}.  Once
this first component and its inherited noise are fixed, the second equation is
again a scalar affine Volterra equation, now with a random drift independent of
its Brownian driver.  This conditional scalar reduction is the reason why the
triangular case is tractable.

\begin{definition}[Triangular weak solution]
\label{def:triangular-solution}
A \emph{triangular weak solution} of
\eqref{eq:limit1}--\eqref{eq:limit2} consists of a standard stochastic
basis $(\Omega,\mathcal F,(\mathbb F_t)_{t\in[0,1]},\mathbb P)$,
continuous non-negative adapted processes $(V^1,V^2)$ with
\[
\E\Bigl[\sup_{t\le1}\bigl(V_t^1+V_t^2\bigr)\Bigr]<\infty,
\]
and $\mathbb F$-Brownian motions $\beta^1,\beta^2$ satisfying
\eqref{eq:limit1}--\eqref{eq:limit2} almost surely, such that $\beta^2$ is
a Brownian motion with respect to the initially enlarged filtration
\[
\mathbb F_t':=\sigma\bigl(V_s^1,\beta_s^1:0\le s\le1\bigr)\vee\mathbb F_t,
\qquad t\in[0,1].
\]
In particular
$\beta^2\perp\!\!\!\perp\sigma(V_s^1,\beta_s^1:0\le s\le1)$.
Proposition~\ref{prop:limit-triangular} below shows that every limit point
of the renormalized Hawkes system is of this type.
\end{definition}

\subsection{Limit points are triangular weak solutions}
\label{ssec:triangular-limit-points}

Theorem~\ref{thm:identification} shows that every limit point of
$(V^{T,1},V^{T,2})$ solves the system
\eqref{eq:limit1}--\eqref{eq:limit2} with independent Brownian drivers.
The uniqueness statement of this section, however, is formulated in the
class of triangular weak solutions, which requires the stronger property
$\beta^2\perp\!\!\!\perp\sigma(V^1_s,\beta^1_s:0\le s\le1)$.  This does not
follow from the independence of $\beta^1$ and $\beta^2$ alone: pathwise
uniqueness for the first limiting equation is not known, so $V^1$ need not
be measurable with respect to $\sigma(\beta^1)$, and
$\beta^1\perp\beta^2$ does not by itself imply
$\beta^2\perp\sigma(V^1,\beta^1)$.  The purpose of this subsection is to
prove that every limit point is nevertheless a triangular weak solution
(Proposition~\ref{prop:limit-triangular}).

The argument has three steps.  First, realising the bivariate Hawkes
system by the standard Poisson thinning construction, we show that at the
prelimit level the compensated martingale of the second component remains
a martingale in the filtration \emph{initially enlarged} by the entire
trajectory of the first component; this is the precise probabilistic
expression of triangularity.  Second, we pass this enlarged martingale
property to the limit through a martingale-problem formulation whose test
functionals are allowed to depend on the whole path of the first
component.  Third, we construct the Brownian drivers so that $\beta^2$ is
a Brownian motion with respect to a filtration whose initial
$\sigma$-field contains $\sigma(V^1,\beta^1)$; a Brownian motion is
independent of the initial $\sigma$-field of its filtration, which yields
the required independence.

\begin{remark}[Canonical thinning construction]\label{rem:thinning}
All statements of this paper concern laws, so we may and do assume that,
for each $T$, the pair $(N^{T,1},N^{T,2})$ is realised by the standard
thinning construction.  Let $\pi^1,\pi^2$ be two independent Poisson
random measures on $(0,\infty)\times(0,\infty)$ with intensity
$ds\,dz$, and let
\begin{equation}\label{eq:thinning}
N_t^{T,1}=\int_{(0,t]\times\Rp}\mathbf 1_{\{z\le\lambda_{s}^{T,1}\}}\,
\pi^1(ds,dz),
\qquad
N_t^{T,2}=\int_{(0,t]\times\Rp}\mathbf 1_{\{z\le\lambda_{s}^{T,2}\}}\,
\pi^2(ds,dz),
\end{equation}
where $\lambda^{T,1},\lambda^{T,2}$ are given by
\eqref{eq:int1}--\eqref{eq:int2}.  Because the system is triangular,
$\lambda^{T,1}$ is a functional of the past of $N^{T,1}$ alone, so the
first equation in \eqref{eq:thinning} is an autonomous equation driven by
$\pi^1$ only; it admits a pathwise unique, non-exploding solution,
constructed by induction over its jump times (the expected number of
points on compacts is finite since $\|\phi_T^1\|_1=a_T^1<1$), and the
construction reproduces the law of the Hawkes system; see
Br\'emaud--Massouli\'e~\cite{BM1996}.  In particular
\begin{equation}\label{eq:N1-measurable}
\sigma\bigl(N_s^{T,1}:s\ge0\bigr)\subseteq\sigma(\pi^1),
\end{equation}
and the same holds for $\lambda^{T,1}$, $M^{T,1}$, $\widehat M^{T,1}$ and
$V^{T,1}$, which are path functionals of $N^{T,1}$.  This step uses the
triangular structure in an essential way: with a feedback kernel $2\to1$,
the first component would also depend on $\pi^2$, and
\eqref{eq:N1-measurable} would fail.
\end{remark}

\begin{lemma}[One-sided initial enlargement at the prelimit]
\label{lem:prelimit-enlargement}
Work under the construction of Remark~\ref{rem:thinning} and define the
initially enlarged filtration
\[
\mathbb G_t^T:=\sigma(\pi^1)\vee
\sigma\bigl(\pi^2((0,s]\times B):s\le t,\ B\in\mathcal B(\Rp)\bigr),
\qquad t\ge0,
\]
augmented by the $\mathbb P$-null sets.  Then $M^{T,2}$ is a
square-integrable $\mathbb G^T$-martingale with predictable bracket
$\langle M^{T,2}\rangle_t=\int_0^t\lambda_s^{T,2}\,ds$ and optional
bracket $[M^{T,2}]_t=N^{T,2}_t$.  Consequently, with
$c_T:=\bigl((1-a_T^2)/(m_2T^{\alpha_2})\bigr)^{1/2}$ as in
\eqref{eq:Mhat-def}, the renormalized martingale
$\widehat M^{T,2}_t=c_TM^{T,2}_{Tt}$ is a square-integrable martingale for
the filtration $(\mathbb G_{Tt}^T)_{t\in[0,1]}$, and both
\begin{equation}\label{eq:G-martingales}
\widehat M^{T,2}
\qquad\text{and}\qquad
(\widehat M^{T,2})^2-\int_0^\cdot V_u^{T,2}\,du
\end{equation}
are $(\mathbb G_{Tt}^T)$-martingales.
\end{lemma}

\begin{proof}
Since $\pi^2$ has independent increments and $\pi^2\perp\pi^1$, for every
$0\le s<t$ the restriction of $\pi^2$ to $(s,t]\times\Rp$ is independent
of $\mathbb G_s^T=\sigma(\pi^1)\vee\sigma(\pi^2|_{(0,s]\times\Rp})$.
Hence $\pi^2$ is a Poisson random measure with respect to
$(\mathbb G_t^T)_{t\ge0}$, and the compensated measure
$\tilde\pi^2(ds,dz):=\pi^2(ds,dz)-ds\,dz$ has the usual martingale
property for bounded $(\mathbb G^T)$-predictable integrands.  The
intensity $\lambda^{T,2}$ is left-continuous and
$\mathbb G^T$-adapted: it is a functional of
$(N^{T,1},N^{T,2})$ restricted to $[0,t)$, and
$\sigma(N^{T,1})\subseteq\mathbb G_0^T$ by~\eqref{eq:N1-measurable};
hence $\lambda^{T,2}$ is $\mathbb G^T$-predictable.  Therefore
\[
M_t^{T,2}
=\int_{(0,t]\times\Rp}\mathbf 1_{\{z\le\lambda_s^{T,2}\}}\,
\tilde\pi^2(ds,dz)
\]
is a $\mathbb G^T$-local martingale with
$\langle M^{T,2}\rangle_t=\int_0^t\lambda_s^{T,2}\,ds$ and
$[M^{T,2}]=N^{T,2}$; square-integrability follows from
$\E\int_0^T\lambda_s^{T,2}\,ds<\infty$, a consequence of
Lemma~\ref{lem:moments-first}.  The first statement in
\eqref{eq:G-martingales} follows by deterministic rescaling.  For the
second, note that $(\widehat M^{T,2})^2-[\widehat M^{T,2}]$ is a
$(\mathbb G_{Tt}^T)$-martingale for any square-integrable martingale,
while
\[
[\widehat M^{T,2}]_t-\int_0^tV_u^{T,2}\,du
=c_T^2\Bigl(N_{Tt}^{T,2}-\int_0^{Tt}\lambda_s^{T,2}\,ds\Bigr)
=c_T\,\widehat M_t^{T,2}
\]
is also a $(\mathbb G_{Tt}^T)$-martingale; here we used
$c_T^2\int_0^{Tt}\lambda_s^{T,2}\,ds=\int_0^tV_u^{T,2}\,du$, which is
\eqref{eq:bracket-audit}.  Summing the two gives
\eqref{eq:G-martingales}.
\end{proof}

The prelimit martingale property of Lemma~\ref{lem:prelimit-enlargement}
survives the passage to the limit.  We now make this precise, recording that
the limiting martingale $W^2$ retains its bracket against the enlarged
filtration generated by the whole path of the first component.

\begin{lemma}[Enlarged martingale property of the limit]
\label{lem:limit-enlargement}
Let $(V^1,W^1,V^2,W^2)$ be a joint limit point of the quadruple
$(V^{T,1},\widehat M^{T,1},$ $V^{T,2},\widehat M^{T,2})$ along the
subsequence used in Lemma~\ref{lem:rebolledo} and
Theorem~\ref{thm:identification}, and define
\[
\mathcal H_t:=\sigma\bigl(V_u^1,W_u^1:u\in[0,1]\bigr)\vee
\sigma\bigl(V_u^2,W_u^2:u\le t\bigr),
\qquad t\in[0,1],
\]
augmented by the null sets.  Then $W^2$ is a continuous square-integrable
$(\mathcal H_t)$-martingale with
$\langle W^2\rangle_t=\int_0^tV_s^2\,ds$.
\end{lemma}

\begin{proof}
The four coordinate sequences are tight (Proposition~\ref{prop:tightness}
and Lemma~\ref{lem:rebolledo}), so after passing to a further subsequence
we may assume joint convergence in law of the quadruple in
$D([0,1];\mathbb R^4)$; by the vanishing-jump estimate
\eqref{eq:Mhat-bracket} and $C$-tightness, every limit is supported on
continuous paths.  By the Skorokhod representation theorem we work on a
probability space on which the convergence holds almost surely; since the
limits are continuous, the $J_1$-convergence is uniform convergence on
$[0,1]$.

Fix $0\le s\le t\le1$, a bounded functional
$\Phi:D([0,1];\mathbb R^2)\to\mathbb R$ continuous with respect to the
uniform norm, and a bounded functional
$\Psi:D([0,s];\mathbb R^2)\to\mathbb R$ continuous with respect to the
uniform norm.  Write
$\Phi^T:=\Phi(V^{T,1},\widehat M^{T,1})$ and
$\Psi^T:=\Psi\bigl((V^{T,2},\widehat M^{T,2})|_{[0,s]}\bigr)$.  Then
$\Phi^T$ is $\mathbb G_0^T$-measurable by~\eqref{eq:N1-measurable}, and
$\Psi^T$ is $\mathbb G_{Ts}^T$-measurable.  Lemma~\ref{lem:prelimit-enlargement}
and optional sampling therefore give, for every $T$,
\begin{align}
\E\Bigl[\bigl(\widehat M_t^{T,2}-\widehat M_s^{T,2}\bigr)\,
\Phi^T\Psi^T\Bigr]&=0,
\label{eq:prelimit-mart}\\
\E\Bigl[\Bigl((\widehat M_t^{T,2})^2-(\widehat M_s^{T,2})^2
-\int_s^tV_u^{T,2}\,du\Bigr)\Phi^T\Psi^T\Bigr]&=0.
\label{eq:prelimit-bracket}
\end{align}

We pass to the limit in
\eqref{eq:prelimit-mart}--\eqref{eq:prelimit-bracket}.  Almost-sure
uniform convergence and the continuity of $\Phi$ and $\Psi$ give
\[
\bigl(\widehat M_t^{T,2}-\widehat M_s^{T,2}\bigr)\Phi^T\Psi^T
\longrightarrow
\bigl(W_t^2-W_s^2\bigr)\,\Phi(V^1,W^1)\,
\Psi\bigl((V^2,W^2)|_{[0,s]}\bigr)
\quad\text{a.s.},
\]
and similarly for the integrand of \eqref{eq:prelimit-bracket}.  For
uniform integrability, $\Phi$ and $\Psi$ are bounded;
$\E[(\widehat M_t^{T,2})^2]=\E\int_0^tV_u^{T,2}\,du$ is bounded uniformly
in $T$ by Lemma~\ref{lem:sup-bound}; and the fourth moments are uniformly
bounded as well: by the Burkholder--Davis--Gundy inequality,
\[
\E\Bigl[\sup_{t\le1}(\widehat M_t^{T,2})^4\Bigr]
\le C\,\E\bigl[[\widehat M^{T,2}]_1^2\bigr]
= C\,\E\Bigl[\Bigl(c_T\widehat M_1^{T,2}
+\int_0^1V_u^{T,2}\,du\Bigr)^{\!2}\Bigr]
\le C\Bigl(c_T^2\,\E\bigl[(\widehat M_1^{T,2})^2\bigr]
+\E\Bigl[\Bigl(\int_0^1V_u^{T,2}\,du\Bigr)^{\!2}\Bigr]\Bigr),
\]
which is bounded uniformly in $T$ by Lemma~\ref{lem:sup-bound} and
$c_T\to0$.  Hence the integrands in
\eqref{eq:prelimit-mart}--\eqref{eq:prelimit-bracket} are uniformly
integrable, and taking limits yields
\begin{align}
\E\Bigl[\bigl(W_t^2-W_s^2\bigr)\,\Phi(V^1,W^1)\,
\Psi\bigl((V^2,W^2)|_{[0,s]}\bigr)\Bigr]&=0,
\label{eq:limit-mart}\\
\E\Bigl[\Bigl((W_t^2)^2-(W_s^2)^2-\int_s^tV_u^2\,du\Bigr)
\Phi(V^1,W^1)\,\Psi\bigl((V^2,W^2)|_{[0,s]}\bigr)\Bigr]&=0.
\label{eq:limit-bracket}
\end{align}
The class of random variables
$\Phi(V^1,W^1)\Psi((V^2,W^2)|_{[0,s]})$ with $\Phi,\Psi$ as above is
stable under multiplication and generates $\mathcal H_s$ up to null sets,
because the paths involved are continuous and bounded uniformly
continuous functionals generate the Borel $\sigma$-field of the
corresponding spaces of continuous paths.  A functional monotone class
argument extends \eqref{eq:limit-mart}--\eqref{eq:limit-bracket} to all
bounded $\mathcal H_s$-measurable multipliers.  Hence $W^2$ and
$(W^2)^2-\int_0^\cdot V^2_u\,du$ are $(\mathcal H_t)$-martingales, which
is the claim; square-integrability follows from Fatou's lemma and the
uniform second-moment bound above.
\end{proof}

Lemma~\ref{lem:limit-enlargement} controls $W^2$ relative to the full
history of the first component; to use it we must know that this history is
itself measurable with respect to the noise $W^1$.  The next lemma supplies
this, identifying $V^1$ and the inherited noise $Z^1$ as functionals of $W^1$
alone.

\begin{lemma}[Measurability of the inherited functionals]
\label{lem:G-measurable}
In the setting of Lemma~\ref{lem:limit-enlargement}, let
$\mathcal G:=\sigma(W_u^1:u\in[0,1])$, augmented by the null sets, and
define $Z^1_t:=\gamma_{12}\nu_1\int_0^tL_{12}(t-u)\,dW_u^1$.  Then $V^1$
and $Z^1$ are indistinguishable from $\mathcal G$-measurable processes;
in particular $\sigma(V^1,Z^1)\subseteq\mathcal G\subseteq\mathcal H_0$.
\end{lemma}

\begin{proof}
By Theorem~\ref{thm:identification},
$V^1-b_1=\nu_1\int_0^\cdot K_1(\cdot-u)\,dW_u^1$, so both processes are
stochastic convolutions of deterministic $L^2([0,1])$ kernels against
$W^1$.  Fix $t$ and let $(q^{(n)})_n$ be deterministic step functions
with $q^{(n)}\to K_1(t-\cdot)\mathbf 1_{[0,t]}$ in $L^2([0,t])$.  Each
$\int_0^tq^{(n)}\,dW^1$ is a finite linear combination of increments of
$W^1$, hence $\mathcal G$-measurable, and by the It\^o isometry,
\[
\E\Bigl[\Bigl|\int_0^t\bigl(q^{(n)}(u)-K_1(t-u)\bigr)dW_u^1\Bigr|^2\Bigr]
=\int_0^t\bigl|q^{(n)}(u)-K_1(t-u)\bigr|^2\,\E[V_u^1]\,du
\le C\,\bigl\|q^{(n)}-K_1(t-\cdot)\bigr\|_{L^2([0,t])}^2
\longrightarrow0,
\]
since $\E[V_u^1]=b_1(u)$ is bounded on $[0,1]$.  Passing to an almost
surely convergent subsequence shows that $V_t^1$ is measurable with
respect to the completion of $\mathcal G$, for every fixed $t$;
continuity of the paths upgrades this to indistinguishability from a
$\mathcal G$-measurable process.  The argument for $Z^1$ with the kernel
$L_{12}$ is identical.
\end{proof}

The two preceding lemmas combine: the enlarged martingale property places
$W^2$ in a filtration whose initial $\sigma$-field already contains $V^1$ and
$Z^1$, which is exactly the independence required of a triangular solution.
We now assemble them.

\begin{proposition}[Limit points are triangular weak solutions]
\label{prop:limit-triangular}
Let $(V^1,V^2)$ be any limit point of $(V^{T,1},V^{T,2})$ in
$C([0,1];\mathbb R^2)$, realised jointly with the martingales
$(W^1,W^2)$ of Lemma~\ref{lem:rebolledo}.  Then, on a standard extension
of the probability space, there exist Brownian motions $\beta^1,\beta^2$
such that \eqref{eq:limit1}--\eqref{eq:limit2} hold and
\begin{equation}\label{eq:beta2-indep}
\beta^2\ \perp\!\!\!\perp\ \sigma\bigl(V_s^1,\beta_s^1:0\le s\le1\bigr).
\end{equation}
In particular $\beta^1\perp\beta^2$, and every limit point of
$(V^{T,1},V^{T,2})$ is a triangular weak solution in the sense of
Definition~\ref{def:triangular-solution}, and
$\E[\sup_{t\le1}(V_t^1+V_t^2)^p]<\infty$ for every $p\ge1$.
\end{proposition}

\begin{proof}
Enlarge the probability space by a product extension carrying two
independent Brownian motions $B^1,B^2$, independent of
$\mathcal H_1$.  Define
\begin{equation}\label{eq:beta-construction}
\beta_t^i:=\int_0^t\mathbf 1_{\{V_u^i>0\}}\,(V_u^i)^{-1/2}\,dW_u^i
+\int_0^t\mathbf 1_{\{V_u^i=0\}}\,dB_u^i,
\qquad i=1,2.
\end{equation}
The first integral is well defined because
$\int_0^t\mathbf 1_{\{V^i>0\}}(V^i)^{-1}\,d\langle W^i\rangle
=\int_0^t\mathbf 1_{\{V^i>0\}}\,du\le t$.  This is the classical
construction of the representing Brownian motion for a continuous local
martingale with absolutely continuous bracket (see
\cite[\S3.4]{KS1991}): each $\beta^i$ is a continuous local martingale
with $\langle\beta^i\rangle_t=t$, hence a Brownian motion by L\'evy's
characterization, and
\[
\int_0^t\sqrt{V_u^i}\,d\beta_u^i
=\int_0^t\mathbf 1_{\{V_u^i>0\}}\,dW_u^i
=W_t^i-\int_0^t\mathbf 1_{\{V_u^i=0\}}\,dW_u^i=W_t^i,
\]
since the bracket of $\int\mathbf 1_{\{V^i=0\}}\,dW^i$ equals
$\int\mathbf 1_{\{V^i=0\}}V^i\,du=0$.  Substituting
$dW^i=\sqrt{V^i}\,d\beta^i$ into the equations of
Theorem~\ref{thm:identification} yields
\eqref{eq:limit1}--\eqref{eq:limit2} with these drivers.

For the independence~\eqref{eq:beta2-indep}, define
\[
\mathbb F_t:=\mathcal H_t\vee\sigma\bigl(B_u^1:u\in[0,1]\bigr)\vee
\sigma\bigl(B_u^2:u\le t\bigr),
\qquad t\in[0,1],
\]
augmented by the null sets.  Note that $\mathbb F_0$ contains the entire
paths of $V^1$, $W^1$ and $B^1$, hence also $\beta^1$ and, by
Lemma~\ref{lem:G-measurable}, $Z^1$.  We claim that $\beta^2$ is an
$(\mathbb F_t)$-Brownian motion.

First, $W^2$ remains an $(\mathbb F_t)$-martingale with bracket
$\int_0^\cdot V^2_u\,du$.  Indeed, $\sigma(B^1,B^2)$ is independent of
$\mathcal H_1$, so for bounded random variables $\xi\in\mathcal H_s$,
$\eta\in\sigma(B^1)$ and $\zeta\in\sigma(B_u^2:u\le s)$,
Lemma~\ref{lem:limit-enlargement} gives
\[
\E\bigl[(W_t^2-W_s^2)\,\xi\eta\zeta\bigr]
=\E\bigl[(W_t^2-W_s^2)\,\xi\bigr]\,\E[\eta\zeta]=0,
\]
and the same factorization applies to
$(W^2)^2-\int_0^\cdot V^2$.  Second, $B^2$ is an
$(\mathbb F_t)$-Brownian motion, because its increments after time $s$
are independent of
$\sigma(B^2_u:u\le s)\vee\sigma(B^1)\vee\mathcal H_1\supseteq\mathbb F_s$.
Third, $\langle W^2,B^2\rangle\equiv0$: writing
$W_t^2B_t^2-W_s^2B_s^2
=(W_t^2-W_s^2)(B_t^2-B_s^2)+W_s^2(B_t^2-B_s^2)+B_s^2(W_t^2-W_s^2)$
and using the independence of the three blocks
$\mathcal H_1$, $\sigma(B^1)$, $\sigma(B^2)$ together with the two
martingale properties just established, each term has vanishing
expectation against $\xi\eta\zeta$ as above, so $W^2B^2$ is an
$(\mathbb F_t)$-martingale.  Consequently $\beta^2$, defined
by~\eqref{eq:beta-construction}, is a continuous
$(\mathbb F_t)$-local martingale with
\[
\langle\beta^2\rangle_t
=\int_0^t\mathbf 1_{\{V_u^2>0\}}(V_u^2)^{-1}\,d\langle W^2\rangle_u
+\int_0^t\mathbf 1_{\{V_u^2=0\}}\,du
=t,
\]
hence an $(\mathbb F_t)$-Brownian motion by L\'evy's characterization.

A Brownian motion with respect to a filtration is independent of the
initial $\sigma$-field of that filtration: for
$0\le t_1<\cdots<t_k\le1$, the increments
$\beta^2_{t_1},\beta^2_{t_2}-\beta^2_{t_1},\ldots$ are successively
independent of $\mathbb F_0$ by conditioning, and they determine
$(\beta^2_{t_1},\ldots,\beta^2_{t_k})$.  Since
$\sigma(V^1,\beta^1)\subseteq\mathbb F_0$, this proves
\eqref{eq:beta2-indep}.  Moreover, $\beta^1$ is
$\mathbb F_0$-measurable, so $\beta^1\perp\beta^2$.
The moment bound follows from Lemma~\ref{lem:sup-bound} and Fatou's lemma
along the almost surely uniformly convergent Skorokhod representation.
It remains to exhibit a filtration as required by
Definition~\ref{def:triangular-solution}.  Let
$\mathbb G_t:=\sigma(V_u^1,W_u^1,V_u^2,W_u^2:u\le t)\vee
\sigma(B_u^1,B_u^2:u\le t)$.  Then $(V^1,V^2)$ is $\mathbb G$-adapted, and
$\beta^1,\beta^2$ are $\mathbb G$-Brownian motions: each is a continuous
$\mathbb G$-local martingale with bracket $t$ by the computation above,
since $W^1$ and $W^2$ remain martingales under the independent enlargement
by $(B^1,B^2)$.  Finally, $\beta^2$ is adapted to
$\sigma(V_s^1,\beta_s^1:s\le1)\vee\mathbb G_t\subseteq\mathbb F_t$, and a
Brownian motion with respect to a filtration is a Brownian motion with
respect to any smaller filtration to which it is adapted; hence $\beta^2$
is a Brownian motion for the initially enlarged filtration
$(\sigma(V_s^1,\beta_s^1:s\le1)\vee\mathbb G_t)_{t\in[0,1]}$, as required
by Definition~\ref{def:triangular-solution}.
\end{proof}

\begin{remark}[Role of triangularity, revisited]
\label{rem:triangularity-revisited}
The triangular structure enters the above argument exactly once, in
Remark~\ref{rem:thinning}: it is what allows the first component to be
constructed from $\pi^1$ alone, so that the initial enlargement of
Lemma~\ref{lem:prelimit-enlargement} costs nothing.  With a feedback
kernel $2\to1$, the trajectory of component~1 would carry information
about the future of $\pi^2$, the enlarged martingale property would fail,
and with it the independence~\eqref{eq:beta2-indep}.  This is the precise
form of the third obstruction listed in Section~\ref{sec:intro}.
\end{remark}

\subsection{A Laplace-functional uniqueness theorem with signed input}
\label{ssec:laplace-signed}

The second limiting equation, conditionally on the first component, is a
scalar Volterra square-root equation whose deterministic input
$b_2+Z^1$ is continuous but not necessarily non-negative.  The affine
Volterra uniqueness theory is usually stated for non-negative (and
suitably admissible) inputs, since the input also governs existence and
positivity.  In this subsection we prove, in a self-contained way, the
uniqueness statement that is actually needed: \emph{among continuous
non-negative solutions with finite first moments, the law is determined by
the input, whatever its sign}.  The conditional version is
Theorem~\ref{thm:ALP-extension} below.  Throughout this subsection, $\varpi_u$ (and $\chi_u$) denote the Riccati--Volterra objects associated with the test function $u$; the symbol is local to this subsection and is unrelated to the resolvent kernels $\psi_T^i,\Psi_T^{12}$ or to the cross-kernel shape $\psi^{12}$.

Throughout this subsection, $K:(0,1]\to[0,\infty)$ is non-increasing with
$K\in L^2(0,1)$, and $\nu>0$ is fixed.  The Mittag--Leffler kernels of
Section~\ref{sec:setup} satisfy these conditions: they are completely
monotone, hence non-increasing, and square-integrable since
$\alpha_i>1/2$.  We write $\|K\|_1:=\|K\|_{L^1(0,1)}$ and set
\begin{equation}\label{eq:rho-eps-def}
\rho:=\frac{1}{2\nu^2\|K\|_1},
\qquad
\varepsilon_0:=\frac{3\rho}{4\|K\|_1}=\frac{3}{8\nu^2\|K\|_1^2}.
\end{equation}

\begin{definition}[Integrable solution]\label{def:integrable-solution}
Let $(\Omega,\mathcal F,(\mathbb F_t)_{t\in[0,1]},\mathbb P)$ be a
stochastic basis carrying an $\mathbb F$-Brownian motion $W$, and let
$f\in C([0,1];\mathbb R)$.  An \emph{integrable solution} with input $f$
is a continuous, non-negative, $\mathbb F$-adapted process $X$ with
$\E[\sup_{t\le1}X_t]<\infty$ such that, almost surely,
\begin{equation}\label{eq:scalar-eq}
X_t=f(t)+\nu\int_0^tK(t-s)\sqrt{X_s}\,dW_s,
\qquad t\in[0,1].
\end{equation}
\end{definition}

Note that $M_t:=\int_0^t\sqrt{X_s}\,dW_s$ is then a true square-integrable
martingale, since $\E\langle M\rangle_1=\E\int_0^1X_s\,ds<\infty$; in
particular $\E[X_t]=f(t)$ for all $t$, so that
$\sup_{t\le1}\E[X_t]\le\|f\|_\infty<\infty$.

\subsubsection{The Riccati--Volterra equation}

For a bounded measurable $u:[0,1]\to\mathbb R$ we consider the
Riccati--Volterra equation
\begin{equation}\label{eq:riccati}
\varpi(t)=\int_0^tK(t-s)\,\chi(s)\,ds,
\qquad
\chi(s):=-u(1-s)+\frac{\nu^2}{2}\,\varpi(s)^2,
\qquad t\in[0,1].
\end{equation}
The time reversal $u(1-\cdot)$ matches the backward role that the test
function plays in the martingale construction below.

\begin{lemma}[Existence for small drives]\label{lem:riccati}
Let $w:[0,1]\to\mathbb R$ be bounded measurable with
\begin{equation}\label{eq:small-drive}
\sup_{t\le1}\int_0^tK(t-s)\,|w(s)|\,ds\ \le\ \frac{3\rho}{4}.
\end{equation}
Then the equation
$\varpi=K*\bigl(w+\tfrac{\nu^2}{2}\varpi^2\bigr)$ has a unique solution
$\varpi\in C([0,1])$ with $\|\varpi\|_\infty\le\rho$.  The smallness
condition~\eqref{eq:small-drive} holds in each of the following two cases:
\begin{enumerate}[label=\textup{(\roman*)}]
\item $\|w\|_\infty\le\varepsilon_0$;
\item $w=c\,\mathbf 1_J$ for a constant $c>0$ and an interval
$J\subseteq[0,1]$ with $c\int_0^{|J|}K(r)\,dr\le 3\rho/4$.
\end{enumerate}
In particular, \eqref{eq:riccati} has a unique solution with
$\|\varpi\|_\infty\le\rho$ whenever $\|u\|_\infty\le\varepsilon_0$.
\end{lemma}

\begin{proof}
Let $B_\rho:=\{\varpi\in C([0,1]):\|\varpi\|_\infty\le\rho\}$ and define
$\Phi\varpi:=K*(w+\tfrac{\nu^2}{2}\varpi^2)$.  Each $\Phi\varpi$ is continuous,
because the convolution of $K\in L^1$ with a bounded function is
continuous (continuity of translation in $L^1$).  For $\varpi\in B_\rho$,
\[
\|\Phi\varpi\|_\infty
\le\sup_{t\le1}\int_0^tK(t-s)|w(s)|\,ds
+\frac{\nu^2\rho^2}{2}\,\|K\|_1
\le\frac{3\rho}{4}+\frac{\rho}{4}=\rho,
\]
since $\tfrac{\nu^2\rho^2}{2}\|K\|_1=\tfrac{\rho}{2}\,(\nu^2\rho\|K\|_1)
=\tfrac{\rho}{4}$ by~\eqref{eq:rho-eps-def}.  For
$\varpi,\varphi\in B_\rho$,
\[
\|\Phi\varpi-\Phi\varphi\|_\infty
\le\frac{\nu^2}{2}\,\|K\|_1\,\|\varpi+\varphi\|_\infty\,
\|\varpi-\varphi\|_\infty
\le\nu^2\rho\|K\|_1\,\|\varpi-\varphi\|_\infty
=\frac12\|\varpi-\varphi\|_\infty .
\]
Banach's fixed point theorem on $B_\rho$ gives existence and uniqueness in
$B_\rho$.  Case~(i) is clear since
$\sup_t\int_0^tK(t-s)|w(s)|\,ds\le\varepsilon_0\|K\|_1=3\rho/4$.  For
case~(ii), the set $\{s\in J:s\le t\}$ has measure at most $|J|$, and
since $K$ is non-increasing the integral of $K(t-s)$ over a set of $s$ of
measure at most $|J|$ contained in $[0,t]$ is maximised when $t-s$ ranges
over $[0,|J|]$; hence
$\sup_t\int_0^tK(t-s)\,c\mathbf 1_J(s)\,ds\le c\int_0^{|J|}K(r)\,dr$.
\end{proof}

\subsubsection{The exponential-affine local martingale}

\begin{lemma}[Exponent identity]\label{lem:exp-affine}
Let $X$ be an integrable solution with input $f$, let
$u:[0,1]\to\mathbb R$ be bounded measurable (of arbitrary sign), and
suppose $\varpi\in C([0,1])$ solves~\eqref{eq:riccati}; set
$\chi:=-u(1-\cdot)+\tfrac{\nu^2}{2}\varpi^2$.  Define, for
$0\le t\le s\le 1$, the forward curve
\[
\xi_t(s):=f(s)+\nu\int_0^tK(s-r)\,dM_r,
\qquad M_r=\int_0^r\sqrt{X_v}\,dW_v,
\]
and the exponent process
\[
H_t:=-\int_0^tu(s)X_s\,ds+\int_t^1\chi(1-s)\,\xi_t(s)\,ds .
\]
Then, almost surely, for all $t\in[0,1]$,
\begin{equation}\label{eq:exponent-identity}
H_t=H_0+N_t-\tfrac12\langle N\rangle_t,
\qquad
N_t:=\nu\int_0^t\varpi(1-s)\sqrt{X_s}\,dW_s .
\end{equation}
Consequently $G:=e^{H}=e^{H_0}\,\mathcal E(N)$ is a positive local
martingale and supermartingale with
\[
G_0=\exp\Bigl(\int_0^1\chi(1-t)f(t)\,dt\Bigr),
\qquad
G_1=\exp\Bigl(-\int_0^1u(t)X_t\,dt\Bigr),
\]
so that
\begin{equation}\label{eq:super-ineq}
\E\Bigl[\exp\Bigl(-\int_0^1u(t)X_t\,dt\Bigr)\Bigr]
\ \le\
\exp\Bigl(\int_0^1\chi(1-t)f(t)\,dt\Bigr),
\end{equation}
with equality if and only if $\mathcal E(N)$ is a true martingale on
$[0,1]$.
\end{lemma}

\begin{proof}
All integrals below are well defined: $\varpi$ and $\chi$ are bounded,
$K\in L^2(0,1)$, and $\sup_{t\le1}\E[X_t]\le\|f\|_\infty$, so that for
each fixed $t$,
\[
\int_t^1|\chi(1-s)|\Bigl(\int_0^tK(s-r)^2\,\E[X_r]\,dr\Bigr)^{1/2}ds
\le\|\chi\|_\infty\,\|f\|_\infty^{1/2}\,\|K\|_{L^2(0,1)}<\infty,
\]
which is the integrability needed for the stochastic Fubini theorem (see,
e.g., \cite{Veraar2012}); the analogous bound holds on the triangle
$\{0\le r<s\le t\}$ used below.

Fix $t\in[0,1]$.  By the definition of $\xi_t$ and stochastic Fubini on
the rectangle $[0,t]\times[t,1]$,
\[
\int_t^1\chi(1-s)\,\xi_t(s)\,ds
=\int_t^1\chi(1-s)f(s)\,ds
+\nu\int_0^t\Theta(t,r)\,dM_r,
\qquad
\Theta(t,r):=\int_t^1\chi(1-s)K(s-r)\,ds .
\]
From~\eqref{eq:riccati} and the substitution $w=1-s$,
\[
\varpi(1-r)=\int_0^{1-r}K(1-r-w)\,\chi(w)\,dw
=\int_r^1K(s-r)\,\chi(1-s)\,ds,
\]
so that $\Theta(t,r)=\varpi(1-r)-\int_r^tK(s-r)\,\chi(1-s)\,ds$.  A second
application of stochastic Fubini, now on the triangle
$\{0\le r<s\le t\}$, together with~\eqref{eq:scalar-eq}, gives
\[
\nu\int_0^t\Bigl(\int_r^tK(s-r)\,\chi(1-s)\,ds\Bigr)dM_r
=\int_0^t\chi(1-s)\Bigl(\nu\int_0^sK(s-r)\,dM_r\Bigr)ds
=\int_0^t\chi(1-s)\bigl(X_s-f(s)\bigr)\,ds .
\]
Combining the three displays and using
$H_0=\int_0^1\chi(1-s)f(s)\,ds$,
\[
H_t-H_0
=-\int_0^t\bigl[u(s)+\chi(1-s)\bigr]X_s\,ds
+\nu\int_0^t\varpi(1-r)\,dM_r .
\]
By the definition of $\chi$,
$u(s)+\chi(1-s)=\tfrac{\nu^2}{2}\varpi(1-s)^2$, while
$\langle N\rangle_t=\nu^2\int_0^t\varpi(1-s)^2X_s\,ds$; this
proves~\eqref{eq:exponent-identity}.  Hence
$G=e^{H_0}\exp(N-\tfrac12\langle N\rangle)=e^{H_0}\mathcal E(N)$ is a
positive local martingale, therefore a supermartingale, and evaluating
$G$ at $t=0$ and $t=1$ (where the forward-curve term vanishes) gives
\eqref{eq:super-ineq}, with equality precisely when
$\E[\mathcal E(N)_1]=1$, i.e.\ when $\mathcal E(N)$ is a true martingale.
\end{proof}

\begin{remark}\label{rem:sign-free}
The proof uses neither the sign of $f$ nor the sign of $u$: both enter
only through the endpoint values of $H$.  Non-negativity of $X$ is used
only through the definition of $\sqrt X$.  This is the reason the
argument extends verbatim to signed inputs, and also the reason the
supermartingale inequality~\eqref{eq:super-ineq} can be exploited with
\emph{negative} test functions, as in the next lemma.
\end{remark}

\subsubsection{Local exponential moments and the true martingale property}

The supermartingale inequality of Lemma~\ref{lem:exp-affine} holds for
inputs of either sign, but to upgrade it to an identity we need the
exponential local martingale to be a true martingale.  The next lemma
provides the integrability for this, by applying the inequality with a
negative test function on a short interval.

\begin{lemma}[Local exponential moments]\label{lem:exp-moment}
For every $c>0$ choose $\delta=\delta(c,K,\nu)>0$ with
$c\int_0^\delta K(r)\,dr\le3\rho/4$.  Then every integrable solution $X$
with input $f$ satisfies, for every interval $I\subseteq[0,1]$ of length
at most $\delta$,
\begin{equation}\label{eq:exp-moment}
\E\Bigl[\exp\Bigl(c\int_IX_t\,dt\Bigr)\Bigr]
\ \le\
\exp\Bigl(\Bigl(c+\frac{\nu^2\rho^2}{2}\Bigr)\|f\|_\infty\Bigr)
\ <\ \infty .
\end{equation}
\end{lemma}

\begin{proof}
Take $u:=-c\,\mathbf 1_I$, so that the drive in~\eqref{eq:riccati} is
$w:=-u(1-\cdot)=c\,\mathbf 1_{1-I}$, where $1-I$ is an interval of the
same length as $I$.  By Lemma~\ref{lem:riccati}(ii) and the choice of
$\delta$, equation~\eqref{eq:riccati} has a solution $\varpi$ with
$\|\varpi\|_\infty\le\rho$, and
$\chi=w+\tfrac{\nu^2}{2}\varpi^2\ge0$ satisfies
$\|\chi\|_{L^1(0,1)}\le c\,|I|+\tfrac{\nu^2\rho^2}{2}
\le c+\tfrac{\nu^2\rho^2}{2}$.  The supermartingale
inequality~\eqref{eq:super-ineq} of Lemma~\ref{lem:exp-affine} then reads
\[
\E\Bigl[\exp\Bigl(c\int_IX_t\,dt\Bigr)\Bigr]
\le\exp\Bigl(\int_0^1\chi(1-t)f(t)\,dt\Bigr)
\le\exp\bigl(\|\chi\|_{L^1(0,1)}\,\|f\|_\infty\bigr),
\]
which is~\eqref{eq:exp-moment}.
\end{proof}

With the local exponential moments in hand, the partitioned Novikov
criterion turns the supermartingale inequality into the exact
Laplace-transform formula.  This is the central computation of the
subsection.

\begin{theorem}[Laplace functional with signed deterministic input]
\label{thm:laplace-deterministic}
Let $f\in C([0,1];\mathbb R)$ and let $X$ be an integrable solution with
input $f$.  Then, for every measurable $u:[0,1]\to[0,\varepsilon_0]$,
\begin{equation}\label{eq:laplace-formula}
\E\Bigl[\exp\Bigl(-\int_0^1u(t)X_t\,dt\Bigr)\Bigr]
=\exp\Bigl(\int_0^1\chi_u(1-t)\,f(t)\,dt\Bigr),
\end{equation}
where $\varpi_u$ is the solution of~\eqref{eq:riccati} provided by
Lemma~\ref{lem:riccati} and
$\chi_u=-u(1-\cdot)+\tfrac{\nu^2}{2}\varpi_u^2$.
\end{theorem}

\begin{proof}
By Lemma~\ref{lem:exp-affine} it suffices to prove that $\mathcal E(N)$,
$N=\nu\int_0^\cdot\varpi_u(1-s)\sqrt{X_s}\,dW_s$, is a true martingale on
$[0,1]$.  Since $\|\varpi_u\|_\infty\le\rho$, for any
$0=t_0<\cdots<t_m=1$,
\[
\tfrac12\bigl(\langle N\rangle_{t_i}-\langle N\rangle_{t_{i-1}}\bigr)
\le\frac{\nu^2\rho^2}{2}\int_{t_{i-1}}^{t_i}X_s\,ds .
\]
Choose the partition with mesh at most
$\delta\bigl(\tfrac{\nu^2\rho^2}{2},K,\nu\bigr)$ from
Lemma~\ref{lem:exp-moment}; then
$\E\bigl[\exp\bigl(\tfrac12(\langle N\rangle_{t_i}
-\langle N\rangle_{t_{i-1}})\bigr)\bigr]<\infty$ for every $i$, and
Novikov's criterion along a partition
(\cite[Corollary~3.5.14]{KS1991}) shows that $\mathcal E(N)$ is a
martingale on $[0,1]$.
\end{proof}

\begin{remark}[Consistency check]\label{rem:nu-zero}
For $\nu=0$ one has $\chi_u=-u(1-\cdot)$, and
\eqref{eq:laplace-formula} reduces to the deterministic identity
$\exp(-\int_0^1u(t)f(t)\,dt)$, as it must since $X=f$ in that case.
This is the reason the exponent in~\eqref{eq:laplace-formula} features
$\chi_u$ rather than $\varpi_u$.
\end{remark}

\subsubsection{Uniqueness in law}

The Laplace functional is determined by the input alone, so two solutions
with the same input cannot be distinguished by these transforms.  Uniqueness
in law follows once the transforms are shown to separate the finite-dimensional
distributions.

\begin{corollary}[Uniqueness in law, deterministic signed input]
\label{cor:unique-deterministic}
Let $f\in C([0,1];\mathbb R)$.  Any two integrable solutions with input
$f$, possibly defined on different stochastic bases, have the same law on
$C([0,1])$.
\end{corollary}

\begin{proof}
Let $X,\widetilde X$ be two integrable solutions with input $f$, and let
$u:[0,1]\to[0,\infty)$ be bounded measurable, not identically zero.  The
functions
\[
L(z):=\E\bigl[e^{-zA}\bigr],
\quad
\widetilde L(z):=\E\bigl[e^{-z\widetilde A}\bigr],
\qquad
A:=\int_0^1u(t)X_t\,dt,
\quad
\widetilde A:=\int_0^1u(t)\widetilde X_t\,dt,
\]
are analytic on the half-plane $\{\Re z>0\}$, since
$A,\widetilde A\ge0$ (differentiation under the expectation is justified
by $|A^ke^{-zA}|\le\sup_{x\ge0}x^ke^{-\varepsilon x}$ for
$\Re z\ge\varepsilon$).  For
$\theta\in(0,\varepsilon_0/\|u\|_\infty]$, the function $\theta u$ takes
values in $[0,\varepsilon_0]$, and
Theorem~\ref{thm:laplace-deterministic}, applied to both solutions with
the \emph{same} Riccati solution $\varpi_{\theta u}$, gives
$L(\theta)=\widetilde L(\theta)$.  The interval
$(0,\varepsilon_0/\|u\|_\infty]$ has accumulation points in the open
half-plane, so the identity theorem yields $L\equiv\widetilde L$ on
$\{\Re z>0\}$; in particular, taking $z=1$,
\begin{equation}\label{eq:all-u}
\E\Bigl[e^{-\int_0^1u(t)X_t\,dt}\Bigr]
=\E\Bigl[e^{-\int_0^1u(t)\widetilde X_t\,dt}\Bigr]
\qquad\text{for every bounded measurable }u\ge0 .
\end{equation}

Finite-dimensional distributions now follow by approximate identities.
Fix $0\le t_1<\cdots<t_m\le1$ and $\theta_1,\ldots,\theta_m\ge0$, and let
$u_n:=\sum_j\theta_ju_{j,n}$ with $u_{j,n}\ge0$ continuous, supported in
$[0,1]$, of unit mass, concentrating at $t_j$.  By path continuity,
$\int_0^1u_nX\,dt\to\sum_j\theta_jX_{t_j}$ almost surely; the
exponentials are bounded by one, and dominated convergence together
with~\eqref{eq:all-u} gives
\[
\E\Bigl[e^{-\sum_j\theta_jX_{t_j}}\Bigr]
=\E\Bigl[e^{-\sum_j\theta_j\widetilde X_{t_j}}\Bigr].
\]
Laplace transforms determine the laws of non-negative random vectors, so
all finite-dimensional distributions coincide; since both processes have
continuous paths, their laws on $C([0,1])$ coincide.
\end{proof}

\subsubsection{The conditional statement}

So far the input has been deterministic.  In the application the drift
$b_2+Z^1$ is random but measurable with respect to the conditioning
$\sigma$-field, so we transfer the deterministic theory under a regular
conditional probability.  The following lemma records that, conditionally,
the Brownian driver and the equation are preserved.

\begin{lemma}[Transfer to regular conditional probabilities]
\label{lem:conditional-transfer}
Let $(\Omega,\mathcal F,(\mathbb F_t),\mathbb P)$ be a standard
stochastic basis carrying an $\mathbb F$-Brownian motion $W$, let
$\mathcal G\subseteq\mathbb F_0$ be a sub-$\sigma$-field, let
$b:\Omega\times[0,1]\to\mathbb R$ be $\mathcal G$-measurable with
continuous paths, and let $X$ be a continuous non-negative
$\mathbb F$-adapted process with $\E[\sup_{t\le1}X_t]<\infty$ satisfying,
almost surely,
\begin{equation}\label{eq:scalar-eq-random}
X_t=b(t)+\nu\int_0^tK(t-s)\sqrt{X_s}\,dW_s,
\qquad t\in[0,1].
\end{equation}
Let $(\mathbb P^\omega)_{\omega\in\Omega}$ be a regular conditional
probability given $\mathcal G$.  Then for $\mathbb P$-almost every
$\omega$: under $\mathbb P^\omega$, the process $W$ is a Brownian motion
with respect to the filtration
$\mathbb F_t^\circ:=\sigma(b)\vee\sigma(X_r,W_r:r\le t)$;
$\E_{\mathbb P^\omega}[\sup_{t\le1}X_t]<\infty$; and $X$ is an integrable
solution in the sense of Definition~\ref{def:integrable-solution} with
the deterministic input $f=b(\cdot,\omega)$.
\end{lemma}

\begin{proof}
\emph{(i) $b$ is frozen.}  The basis is standard and $b$ is
$\mathcal G$-measurable with values in the Polish space $C([0,1])$, so
$\mathbb P^\omega\bigl(b=b(\cdot,\omega)\bigr)=1$ for almost every
$\omega$.

\emph{(ii) Conditional Brownian motion.}  Note
$\mathbb F_t^\circ\subseteq\mathbb F_t$, since
$\sigma(b)\subseteq\mathcal G\subseteq\mathbb F_0$.  Fix rationals
$s<t$ in $[0,1]$ and a countable convergence-determining family
$(\zeta_k)_k$ of bounded continuous functionals of
$\bigl(b,(X,W)|_{[0,s]}\bigr)$.  Each $\zeta_k$ is bounded and
$\mathbb F_s$-measurable, so the $\mathbb F$-Brownian property of $W$
gives
\[
\E\bigl[(W_t-W_s)\,\zeta_k\,\big|\,\mathcal G\bigr]=0,
\qquad
\E\bigl[\bigl((W_t-W_s)^2-(t-s)\bigr)\,\zeta_k\,\big|\,\mathcal G\bigr]=0
\qquad\text{a.s.}
\]
Collecting the countably many null sets, for almost every $\omega$ the
process $W$ is, under $\mathbb P^\omega$, a continuous square-integrable
$\mathbb F^\circ$-martingale with $\langle W\rangle_t=t$: the
integrability follows from $\E[W_t^2\mid\mathcal G]<\infty$ a.s., and the
identities extend from the $\zeta_k$ to all bounded
$\mathbb F_s^\circ$-measurable variables by a monotone class argument.
L\'evy's characterization makes $W$ an $\mathbb F^\circ$-Brownian motion
under $\mathbb P^\omega$.

\emph{(iii) Stochastic integrals transfer.}  The integrals
in~\eqref{eq:scalar-eq-random} are limits in $\mathbb P$-probability,
uniformly on $[0,1]$, of integrals of simple
$\mathbb F^\circ$-predictable integrands along a fixed approximating
sequence: first approximate $\sqrt X$ by simple adapted processes in
$L^2(\mathbb P\otimes dt)$, then the deterministic kernel by step
functions in $L^2(0,1)$, as in Lemma~\ref{lem:G-measurable}; the
required moment bounds hold since $\sup_t\E[X_t]<\infty$.  By
conditional Chebyshev,
$\E\bigl[\min\bigl(1,\sup_{t\le1}|\cdot|\bigr)\,\big|\,\mathcal G\bigr]
\to0$ in $L^1$ along this sequence, so along a further subsequence the
same simple integrals converge in $\mathbb P^\omega$-probability, for
almost every $\omega$, to the processes appearing
in~\eqref{eq:scalar-eq-random}.  Since the
$\mathbb P^\omega$-stochastic integrals are by definition the
$\mathbb P^\omega$-limits of the same simple integrals (legitimate by
step~(ii)), the integrals in~\eqref{eq:scalar-eq-random} are
$\mathbb P^\omega$-versions of the corresponding
$\mathbb P^\omega$-stochastic integrals.  Moreover
$\mathbb P(\eqref{eq:scalar-eq-random}\text{ holds})=1$ implies
$\mathbb P^\omega(\eqref{eq:scalar-eq-random}\text{ holds})=1$ for
almost every $\omega$.

\emph{(iv) Integrability.}
$\E_{\mathbb P^\omega}[\sup_tX_t]
=\E[\sup_tX_t\mid\mathcal G](\omega)<\infty$ for almost every $\omega$.
Combining (i)--(iv) proves the lemma.
\end{proof}

Combining the transfer lemma with the deterministic uniqueness corollary
gives the conditional statement, which is the form used for the second
component.

\begin{theorem}[Conditional uniqueness with signed input]
\label{thm:ALP-extension}
Let $K,\nu$ be as above.  Let
$(\Omega,\mathcal F,(\mathbb F_t),\mathbb P)$ be a standard stochastic
basis carrying an $\mathbb F$-Brownian motion $W$, let
$\mathcal G\subseteq\mathbb F_0$, let $b$ be $\mathcal G$-measurable
with continuous paths, and let $X$ be a continuous non-negative
$\mathbb F$-adapted process with $\E[\sup_{t\le1}X_t]<\infty$
satisfying~\eqref{eq:scalar-eq-random}.  Then there is a measurable
kernel $Q:C([0,1])\to\mathcal P(C([0,1]))$, depending only on $(K,\nu)$,
such that
\[
\mathcal L\bigl(X\mid\mathcal G\bigr)=Q(b)
\qquad\text{a.s.}
\]
In particular, if $(\mathcal G,b,X)$ and
$(\widetilde{\mathcal G},\widetilde b,\widetilde X)$ are two such
configurations, possibly on different bases, with
$b\overset{d}{=}\widetilde b$, then
$(b,X)\overset{d}{=}(\widetilde b,\widetilde X)$.
\end{theorem}

\begin{proof}
By Lemma~\ref{lem:conditional-transfer}, for almost every $\omega$ the
process $X$ is, under $\mathbb P^\omega$, an integrable solution with the
deterministic input $b(\cdot,\omega)$.  Hence
Theorem~\ref{thm:laplace-deterministic} applies under
$\mathbb P^\omega$: for every measurable $u:[0,1]\to[0,\infty)$ bounded,
every $\theta\in(0,\varepsilon_0/\|u\|_\infty]$, and almost every
$\omega$,
\begin{equation}\label{eq:cond-laplace}
\E_{\mathbb P^\omega}\Bigl[e^{-\theta\int_0^1u(t)X_t\,dt}\Bigr]
=\exp\Bigl(\int_0^1\chi_{\theta u}(1-t)\,b(t,\omega)\,dt\Bigr)
=:\Lambda_{\theta u}\bigl(b(\cdot,\omega)\bigr),
\end{equation}
where $\Lambda_{\theta u}:C([0,1])\to(0,\infty)$ is an explicit
continuous functional depending only on $(K,\nu,\theta u)$.

We claim that the conditional law of $X$ under $\mathbb P^\omega$ is, for
almost every $\omega$, a fixed Borel functional of $b(\cdot,\omega)$,
the same for every configuration satisfying the hypotheses.  Fix $u$ as
above and set
$L^\omega(\theta):=\E_{\mathbb P^\omega}[e^{-\theta\int uX}]$.  As the
Laplace transform of a non-negative random variable, $L^\omega$ is
analytic on $\{\Re\theta>0\}$, and its power series at any $\theta_0>0$
has radius of convergence at least $\theta_0$.  Its values on
$(0,\varepsilon_0/\|u\|_\infty]$ are given by~\eqref{eq:cond-laplace},
hence are Borel functions of $b(\cdot,\omega)$; its derivatives at any
interior point of that interval are limits of difference quotients of
those values, hence Borel functions of $b(\cdot,\omega)$ as well; and its
values at arbitrary $\theta>0$ are obtained by finitely many power-series
continuation steps with deterministic step sizes.  Consequently
$L^\omega(1)=\E_{\mathbb P^\omega}[e^{-\int uX}]$ is, for every bounded
measurable $u\ge0$, a fixed Borel functional of $b(\cdot,\omega)$.
Approximate identities and dominated convergence, exactly as in the proof
of Corollary~\ref{cor:unique-deterministic}, extend this to the
finite-dimensional Laplace transforms
$\E_{\mathbb P^\omega}[\exp(-\sum_j\theta_jX_{t_j})]$, and a functional
monotone class argument to $\E_{\mathbb P^\omega}[g(X)]$ for every
bounded Borel $g$ on $C([0,1])$.

Define $Q(b)$ as the conditional law of $X$ given $\sigma(b)$ for one
fixed reference configuration (such a configuration exists: the limit
construction of this paper provides one).  The preceding paragraph shows
$\mathcal L(X\mid\mathcal G)=Q(b)$ a.s.\ for \emph{every} configuration,
and that $Q$ depends only on $(K,\nu)$.  The final statement follows by
integrating the conditional laws against the common law of $b$.
\end{proof}
We record separately the statement for the first component, which has a
non-negative deterministic input.

\begin{theorem}[Scalar uniqueness for the first component]
\label{thm:ALP}
Let $K$ be one of the Mittag--Leffler kernels $K_i$, let $\nu>0$, and let
$b\in C([0,1];\mathbb R_+)$.  Among continuous non-negative weak solutions
of
\begin{equation}\label{eq:ALP}
X_t=b(t)+\nu\int_0^tK(t-s)\sqrt{X_s^+}\,dW_s
\end{equation}
with $\E[\sup_{t\le1}X_t]<\infty$, the law on $C([0,1])$ is unique.
\end{theorem}

\begin{proof}
Since $X\ge0$ one has $\sqrt{X^+}=\sqrt X$, so $X$ is an integrable
solution in the sense of Definition~\ref{def:integrable-solution} with
input $f=b$, and the claim is a special case of
Corollary~\ref{cor:unique-deterministic}.  The statement is also contained
in the affine Volterra theory of~\cite{ALP2019}.
\end{proof}

Finally, the second equation depends not only on $V^1$ but on the inherited
noise generated by the same Brownian motion.  We record explicitly the scalar
``enhanced'' uniqueness input needed for this step.

\begin{lemma}[Enhanced scalar uniqueness for the inherited noise]
\label{lem:enhanced-scalar-uniqueness}
Let $V^1$ be a continuous non-negative weak solution, with
$\E[\sup_{t\le1}V_t^1]<\infty$, of the first limiting Volterra equation,
\[
V_t^1=b_1(t)+\nu_1\int_0^tK_1(t-s)\,dM_s^1,
\qquad
M_t^1:=\int_0^t\sqrt{V_s^1}\,d\beta_s^1 .
\]
Define the inherited-noise process
\[
Z_t^1:=\gamma_{12}\nu_1\int_0^tL_{12}(t-s)\,dM_s^1 .
\]
Then the joint law of $(V^1,Z^1)$ on $C([0,1];\mathbb R^2)$ is uniquely
determined by the model parameters.
\end{lemma}

\begin{proof}
By Theorem~\ref{thm:ALP}, the law of $V^1$ is uniquely determined.  We first
show that the martingale driver $M^1$ is in fact recoverable from the path of
$V^1$.

Let $R_1$ be the resolvent of the first kind associated with $K_1$, namely
\begin{equation}\label{eq:first-kind-resolvent}
R_1*K_1=1,
\end{equation}
where $1$ denotes the constant kernel on $\mathbb R_+$.  Since
\[
\widehat K_1(z)=\frac{1}{1+\tilde\delta_1z^{\alpha_1}},
\]
we have
\[
\widehat R_1(z)\widehat K_1(z)=\frac1z,
\qquad\text{hence}\qquad
\widehat R_1(z)=\frac{1+\tilde\delta_1z^{\alpha_1}}{z}
=\frac1z+\tilde\delta_1 z^{\alpha_1-1}.
\]
Thus $R_1$ is the locally integrable kernel
\[
R_1(t)=1+\frac{\tilde\delta_1}{\Gamma(1-\alpha_1)}t^{-\alpha_1},
\qquad t>0,
\]
which is locally integrable because $\alpha_1<1$.
Applying $R_1$ to the Volterra equation for $V^1$ and using stochastic Fubini
with~\eqref{eq:first-kind-resolvent} gives, for all $t\in[0,1]$,
\begin{equation}\label{eq:recover-M-from-V}
M_t^1
=\frac1{\nu_1}\int_0^tR_1(t-s)\bigl(V_s^1-b_1(s)\bigr)\,ds.
\end{equation}
Consequently $M^1$ is a measurable functional of $V^1$, and the joint law of
$(V^1,M^1)$ is uniquely determined.

It remains to prove that $Z^1$ is determined in law by $(V^1,M^1)$.  Fix
$t\in[0,1]$ and set
\[
q_t(s):=L_{12}(t-s)\mathbf 1_{[0,t]}(s).
\]
Since $L_{12}\in L^2([0,1])$, choose deterministic step functions
$q_t^{(n)}$ on $[0,t]$ such that $q_t^{(n)}\to q_t$ in $L^2([0,t])$, and define
\[
Z_t^{1,n}:=\gamma_{12}\nu_1\int_0^tq_t^{(n)}(s)\,dM_s^1.
\]
Each $Z_t^{1,n}$ is a finite linear combination of increments of $M^1$, hence
is a measurable functional of the path of $M^1$.  Moreover, by It\^o's
isometry and $d\langle M^1\rangle_s=V_s^1\,ds$,
\[
\mathbb E\bigl[|Z_t^{1,n}-Z_t^1|^2\bigr]
=(\gamma_{12}\nu_1)^2
\int_0^t|q_t^{(n)}(s)-q_t(s)|^2\,\mathbb E[V_s^1]  \,ds.
\]
Since $\mathbb E[V_s^1]=b_1(s)$ is bounded on $[0,1]$, the right-hand side is
bounded by
\[
C\|q_t^{(n)}-q_t\|_{L^2([0,t])}^2\longrightarrow0.
\]
Hence $Z_t^{1,n}\to Z_t^1$ in $L^2$.

Repeating the same construction for any finite family of rational times
$t_1,\ldots,t_m\in\mathbb Q\cap[0,1]$ shows that the finite-dimensional laws
of $(V^1,Z^1)$ at rational times are uniquely determined by the law of
$(V^1,M^1)$.  The process $Z^1$ has continuous paths by the stochastic
convolution estimates used in the tightness section, and $V^1$ is continuous
by assumption.  Therefore these rational-time finite-dimensional
distributions determine the law of $(V^1,Z^1)$ on
$C([0,1];\mathbb R^2)$.
\end{proof}

\subsection{Positivity of limit points}

\begin{lemma}[Positivity of limit points]
\label{lem:positivity}
Under Assumption~\ref{ass:kernels}, every limit point $(V^1,V^2)$ of
$(V^{T,1},V^{T,2})$ in the path space $C([0,1];\mathbb R^2)$ satisfies
\[
V^i_t\geq0,\qquad t\in[0,1],\quad i=1,2,
\]
almost surely.
\end{lemma}

\begin{proof}
For every $T$ and every $i$, $V^{T,i}$ is a positive deterministic multiple of
the Hawkes intensity $\lambda^{T,i}$, hence it takes values in the closed cone
$C([0,1];\mathbb R_+^2)$.  Since this cone is closed for the uniform topology,
weak limits of laws supported on it are again supported on it by the
Portmanteau theorem.
\end{proof}

\subsection{Conditional uniqueness of the second component}

\begin{lemma}[Conditional uniqueness for $V^2$]
\label{lem:cond-uniqueness}
Let $(V^1,V^2,\beta^1,\beta^2)$ be a triangular weak solution in the sense
of Definition~\ref{def:triangular-solution}, define the inherited-noise
process
\begin{equation}\label{eq:inherited-noise-Z}
Z_t^1:=\gamma_{12}\nu_1\int_0^tL_{12}(t-s)\sqrt{V_s^1}\,d\beta_s^1,
\end{equation}
and let $\mathcal G:=\sigma(V_s^1,Z_s^1:0\le s\le1)$.  Then $V^2$
satisfies
\begin{equation}\label{eq:V2-conditional}
V_t^2=b_2(t)+Z_t^1+\nu_2\int_0^tK_2(t-s)\sqrt{V_s^2}\,d\beta_s^2,
\end{equation}
and
\[
\mathcal L\bigl(V^2\mid\mathcal G\bigr)=Q\bigl(b_2+Z^1\bigr)
\qquad\text{a.s.},
\]
where $Q$ is the kernel of Theorem~\ref{thm:ALP-extension} for
$(K,\nu)=(K_2,\nu_2)$.
\end{lemma}

\begin{proof}
Equation~\eqref{eq:V2-conditional} is~\eqref{eq:limit2} with the
inherited-noise term written as $Z^1$, and $\sqrt{(V^2)^+}=\sqrt{V^2}$
since $V^2\ge0$.  We verify the hypotheses of
Theorem~\ref{thm:ALP-extension} with $W=\beta^2$, the initially enlarged
filtration $(\mathbb F_t')$ of Definition~\ref{def:triangular-solution},
and $b:=b_2+Z^1$.  By definition, $\beta^2$ is an $\mathbb F'$-Brownian
motion, $V^2$ is $\mathbb F'$-adapted, and
$\E[\sup_{t\le1}V_t^2]<\infty$.  The process $Z^1$ is a stochastic
convolution of the deterministic kernel $L_{12}\in L^2(0,1)$ against
$W^1=\int_0^\cdot\sqrt{V^1}\,d\beta^1$, hence measurable with respect to
the completion of $\sigma(V^1,\beta^1)\subseteq\mathbb F_0'$ (argue as in
Lemma~\ref{lem:G-measurable}), and it has a continuous modification by the
Kolmogorov estimates of Section~\ref{sec:tightness}.  Therefore
$b=b_2+Z^1$ is $\mathcal G$-measurable with continuous paths and
$\mathcal G\subseteq\mathbb F_0'$.  Theorem~\ref{thm:ALP-extension} now
gives the claim.
\end{proof}

\subsection{Joint uniqueness}

\begin{theorem}[Uniqueness in law for the triangular limiting system]
\label{thm:uniqueness}
Under Assumption~\ref{ass:kernels}, the system
\eqref{eq:limit1}--\eqref{eq:limit2} admits at most one law among triangular
weak solutions in the sense of Definition~\ref{def:triangular-solution}.
\end{theorem}

\begin{proof}
Let $(V^1,V^2)$ and $(\widetilde V^1,\widetilde V^2)$ be two triangular weak
solutions, with Brownian drivers $(\beta^1,\beta^2)$ and
$(\widetilde\beta^1,\widetilde\beta^2)$.

First, the first components solve the same scalar equation
\[
X_t=b_1(t)+\nu_1\int_0^tK_1(t-s)\sqrt{X_s}\,dW_s.
\]
By Theorem~\ref{thm:ALP}, whose integrability hypothesis holds by
Definition~\ref{def:triangular-solution}, $V^1$ and $\widetilde V^1$ have
the same law; the same integrability shows that the stochastic convolutions
are centred, so that $\E[V_s^1]=b_1(s)$ as used in
Lemma~\ref{lem:enhanced-scalar-uniqueness}.  By
Lemma~\ref{lem:enhanced-scalar-uniqueness}, the pairs
\[
\left(V^1,Z^1\right),
\qquad
\left(\widetilde V^1,\widetilde Z^1\right),
\]
have the same law, where $Z^1$ and $\widetilde Z^1$ are the corresponding
inherited-noise processes defined by~\eqref{eq:inherited-noise-Z}.

Now condition on $\mathcal G=\sigma(V^1,Z^1)$, and similarly on
$\widetilde{\mathcal G}=\sigma(\widetilde V^1,\widetilde Z^1)$.  By
Lemma~\ref{lem:cond-uniqueness}, the conditional law of $V^2$ given
$\mathcal G$ equals $Q(b_2+Z^1)$ for the kernel $Q$ of
Theorem~\ref{thm:ALP-extension}, and the same kernel applies to
$\widetilde V^2$ given $\widetilde{\mathcal G}$.  Since the laws of the conditioning pairs
coincide, the joint laws of $(V^1,V^2)$ and
$(\widetilde V^1,\widetilde V^2)$ coincide.
\end{proof}

\subsection{The full convergence theorem}

Combining the tightness (Proposition~\ref{prop:tightness}),
the identification of limit points (Theorem~\ref{thm:identification}),
and the uniqueness in law (Theorem~\ref{thm:uniqueness}),
we obtain the main theorem of this paper:

\begin{theorem}[Weak convergence of the bivariate intensity process]
\label{thm:main-final}
Under Assumption~\ref{ass:kernels},
the full sequence of renormalized bivariate intensity processes
$(V^{T,1},V^{T,2})_{T\geq 1}$ converges weakly in $C([0,1];\mathbb{R}^2)$
to the unique triangular weak solution $(V^1,V^2)$ of the coupled system
\begin{align*}
V_t^1 &= b_1(t)
+ \nu_1\int_0^t K_1(t-s)\sqrt{V_s^1}\,d\beta_s^1,\\
V_t^2 &= b_2(t)
+ \nu_2\int_0^t K_2(t-s)\sqrt{V_s^2}\,d\beta_s^2
+ \gamma_{12}\nu_1
\int_0^t L_{12}(t-s)\sqrt{V_s^1}\,d\beta_s^1,
\end{align*}
where $\beta^1$, $\beta^2$ are independent standard Brownian motions,
$\nu_i=(m_i\lambda_i)^{-1/2}$, $\gamma_{12}=m_1/(m_2\lambda_1)$,
$L_{12}=b_\infty^{12}\|\psi^{12}\|_1(K_1*K_2)$,
$K_i(t) \sim (\tilde\delta_i\Gamma(\alpha_i))^{-1} t^{\alpha_i-1}$ as $t\downarrow 0$,
and each $V^i$ is almost surely H\"older continuous of every order below
$H_i=\alpha_i-\tfrac12\in(0,\tfrac12)$, with mean-square increments of
exact order $(t-s)^{2H_i}$ away from the time origin
(Proposition~\ref{prop:regularity}); the Hurst parameters $H_1<H_2$ are
distinct.
\end{theorem}

\begin{proof}
By Proposition~\ref{prop:tightness} the sequence is $C$-tight, so every
subsequence admits a further weakly convergent subsequence.  By
Theorem~\ref{thm:identification}, every limit point satisfies the coupled
Volterra system, and by Proposition~\ref{prop:limit-triangular} it is a
triangular weak solution in the sense of
Definition~\ref{def:triangular-solution}.  By
Theorem~\ref{thm:uniqueness}, all triangular weak solutions have the same
law.  Hence all subsequential limits coincide, and the full sequence
converges weakly to this common law.
\end{proof}

The convergence theorem identifies the limit but says nothing about the
regularity of its paths.  We close by showing that each component is rough of
the order announced in the introduction, and that this order is sharp.

\begin{proposition}[Path regularity and sharpness of the Hurst exponents]
\label{prop:regularity}
Let $(V^1,V^2)$ be the limit process of Theorem~\ref{thm:main-final}, set
$H_i:=\alpha_i-\tfrac12\in(0,\tfrac12)$, and write
\[
V^1=b_1+\mathcal S^1,
\qquad
V^2=b_2+\mathcal Z+\mathcal S^2,
\]
where
$\mathcal S^i_t=\nu_i\int_0^tK_i(t-u)\sqrt{V^i_u}\,d\beta^i_u$ are the
self parts and
$\mathcal Z_t=\gamma_{12}\nu_1\int_0^tL_{12}(t-u)\sqrt{V^1_u}\,d\beta^1_u$
is the inherited part.  Then:
\begin{enumerate}
\item[\textup{(a)}] Almost surely, $V^i$ is H\"older continuous on $[0,1]$
of every order $\gamma<H_i$, for $i=1,2$.
\item[\textup{(b)}] Almost surely, the inherited part $\mathcal Z$ is
H\"older continuous of every order
$\gamma<\bigl(\alpha_1+\alpha_2-\tfrac12\bigr)\wedge1$; since
$\alpha_1+\alpha_2-\tfrac12>\alpha_2-\tfrac12$, the roughness of $V^2$ is
carried entirely by its self part $\mathcal S^2$.
\item[\textup{(c)}] The exponents are sharp in the mean-square sense: for
every $s_0\in(0,1)$ there are constants $0<c\le C<\infty$ such that
\begin{equation}\label{eq:two-sided-ms}
c\,(t-s)^{2H_i}
\ \le\
\E\bigl[(V^i_t-V^i_s)^2\bigr]
\ \le\
C\,(t-s)^{2H_i},
\qquad
s_0\le s<t\le1,
\end{equation}
where the upper bound holds for all $0\le s<t\le1$.  In particular, for
every $\gamma>H_i$,
$\sup_{s_0\le s<t\le1}\E[(V^i_t-V^i_s)^2]/(t-s)^{2\gamma}=\infty$, so the
regularity in~\textup{(a)} cannot be improved: $H_i$ is the exact
mean-square H\"older exponent of $V^i$.
\end{enumerate}
\end{proposition}

\begin{proof}
Throughout, $h:=t-s\in(0,1]$, the kernels are extended by
$K_i(x)=L_{12}(x)=0$ for $x\le0$, and $C$, $C_p$, $c$ denote finite
positive constants that may change from line to line.

\smallskip
\emph{Step 0: moments and means of the limit.}
By Lemma~\ref{lem:sup-bound}, Skorokhod representation along the
convergent (sub)sequence, almost-sure uniform convergence, and Fatou's
lemma,
\begin{equation}\label{eq:limit-moments}
\E\Bigl[\sup_{t\le1}\bigl(V^1_t+V^2_t\bigr)^p\Bigr]<\infty
\qquad\text{for every }p\ge1 .
\end{equation}
In particular the integrands of $\mathcal S^i$ and $\mathcal Z$ are in
$L^2(\mathbb P\otimes dt)$, the stochastic convolutions are centred, and
$\E[V^i_u]=b_i(u)$ for all $u$, with $b_1,b_2$ given
by~\eqref{eq:b1-def}--\eqref{eq:b2-full}; both are non-decreasing and
strictly positive on $(0,1]$, since $K_i>0$ and $L_{12}\ge0$.

\smallskip
\emph{Step 1: kernel estimates.}
The macroscopic kernels satisfy, for $h\in(0,1]$,
\begin{align}
\|K_i(\cdot+h)-K_i\|_{L^2(\Rp)}^2+\int_0^hK_i(u)^2\,du
&\le C\,h^{2\alpha_i-1},
\label{eq:K-shift}\\
\|L_{12}(\cdot+h)-L_{12}\|_{L^2(\Rp)}^2+\int_0^hL_{12}(u)^2\,du
&\le C\,h^{\eta_{12}},
\qquad
\eta_{12}=\min\{2(\alpha_1+\alpha_2)-1,2\}.
\label{eq:L-shift}
\end{align}
Indeed, by~\eqref{eq:ML} and
$\Re\bigl(1+\tilde\delta_i(i\xi)^{\alpha_i}\bigr)
=1+\tilde\delta_i|\xi|^{\alpha_i}\cos(\pi\alpha_i/2)$ with
$\cos(\pi\alpha_i/2)>0$,
\[
|\widehat{K_i}(i\xi)|\le\frac{C}{1+|\xi|^{\alpha_i}},
\qquad
|\widehat{L_{12}}(i\xi)|
=b_\infty^{12}\|\psi^{12}\|_1\,
|\widehat{K_1}(i\xi)|\,|\widehat{K_2}(i\xi)|
\le\frac{C}{(1+|\xi|^{\alpha_1})(1+|\xi|^{\alpha_2})},
\]
and the shift bounds follow from the Plancherel computation in the proof
of Lemma~\ref{lem:L2-shift-pre}, with the same split of the frequency
integral at $|\xi|=1/h$ (alternatively, they follow from
\eqref{eq:self-shift-bound}--\eqref{eq:cross-shift-bound} and the $L^2$
convergence of Lemma~\ref{lem:L2-self}).  The tail integrals follow from
the pointwise bounds $K_i(u)\le Cu^{\alpha_i-1}$ and
$L_{12}(u)\le Cu^{\alpha_1+\alpha_2-1}$ on $(0,1]$, which hold
by~\eqref{eq:ML-detail} and Proposition~\ref{prop:main}, noting
$2(\alpha_1+\alpha_2)-1\ge\eta_{12}$.  The same pointwise bounds give the
deterministic increments
\begin{equation}\label{eq:b-increments}
0\le b_1(t)-b_1(s)\le C\,h^{\alpha_1},
\qquad
0\le b_2(t)-b_2(s)\le C\,h^{\alpha_2},
\end{equation}
using $\int_s^tu^{\alpha-1}\,du\le\int_0^hu^{\alpha-1}\,du=h^\alpha/\alpha$
for $\alpha\in(0,1]$.

\smallskip
\emph{Step 2: increment moments and H\"older continuity.}
Fix $p\ge1$ and $0\le s<t\le1$.  With the zero-extension convention,
\[
\mathcal S^1_t-\mathcal S^1_s
=\nu_1\int_0^1\bigl[K_1(t-u)-K_1(s-u)\bigr]\sqrt{V^1_u}\,d\beta^1_u,
\]
so the Burkholder--Davis--Gundy inequality,
\eqref{eq:limit-moments} and~\eqref{eq:K-shift} give
\[
\E\bigl[|\mathcal S^1_t-\mathcal S^1_s|^{2p}\bigr]
\le C_p\,
\E\Bigl[\Bigl(\int_0^1\bigl[K_1(t-u)-K_1(s-u)\bigr]^2V^1_u\,du
\Bigr)^{\!p}\Bigr]
\le C_p\Bigl(\int_0^1\bigl[K_1(t-u)-K_1(s-u)\bigr]^2du\Bigr)^{\!p}
\le C_p\,h^{p(2\alpha_1-1)},
\]
where the kernel integral was split into the shifted part on $[0,s]$ and
the fresh part on $(s,t]$, bounded by~\eqref{eq:K-shift}.  Identically,
\[
\E\bigl[|\mathcal S^2_t-\mathcal S^2_s|^{2p}\bigr]
\le C_p\,h^{p(2\alpha_2-1)},
\qquad
\E\bigl[|\mathcal Z_t-\mathcal Z_s|^{2p}\bigr]
\le C_p\,h^{p\,\eta_{12}} .
\]
By the Kolmogorov--Centsov criterion
(\cite[Theorem~2.2.8]{KS1991}), $\mathcal S^i$ admits a modification
that is H\"older continuous of every order
$\gamma<\bigl(p(2\alpha_i-1)-1\bigr)/(2p)$, and letting $p\to\infty$, of
every order $\gamma<\alpha_i-\tfrac12=H_i$; since $\mathcal S^i$ is
already continuous, it is indistinguishable from this modification, and
the same applies to $\mathcal Z$ with the exponent
$\eta_{12}/2=(\alpha_1+\alpha_2-\tfrac12)\wedge1$.  This
proves~(b), and together with~\eqref{eq:b-increments} (note
$\alpha_i>H_i$ and $\alpha_1+\alpha_2-\tfrac12>\alpha_2-\tfrac12$) it
proves~(a).

\smallskip
\emph{Step 3: two-sided mean-square bounds.}
The upper bound in~\eqref{eq:two-sided-ms} follows from Step~2 with
$p=1$ and~\eqref{eq:b-increments}: for $V^2$,
the three contributions are of orders $h^{2\alpha_2}$, $h^{\eta_{12}}$
and $h^{2\alpha_2-1}$, and $\eta_{12}>2\alpha_2-1$,
$2\alpha_2>2\alpha_2-1$.

For the lower bound, fix $s_0\le s<t\le1$ and work on the stochastic
basis of the triangular weak solution
(Definition~\ref{def:triangular-solution}), so that $\beta^1,\beta^2$ are
Brownian motions and $V^1,V^2$ are adapted with respect to a common
filtration $(\mathbb F_r)$.  Decompose
\[
V^2_t-V^2_s
=\underbrace{\bigl[b_2(t)-b_2(s)\bigr]+\mathcal M_s}_{\in\,
L^2(\mathbb F_s)}
+\underbrace{\gamma_{12}\nu_1\int_s^tL_{12}(t-u)\sqrt{V^1_u}\,
d\beta^1_u}_{=:\,F_Z}
+\underbrace{\nu_2\int_s^tK_2(t-u)\sqrt{V^2_u}\,d\beta^2_u}_{=:\,F_2},
\]
where $\mathcal M_s$ collects the integrals over $[0,s]$ of the shifted
kernels; $\mathcal M_s$ is $\mathbb F_s$-measurable and square-integrable
by Step~2.  The fresh terms $F_Z$ and $F_2$ are increments over $(s,t]$
of square-integrable martingales, hence orthogonal to $L^2(\mathbb F_s)$;
and they are orthogonal to each other, since
$\langle\beta^1,\beta^2\rangle\equiv0$.  Therefore
\[
\E\bigl[(V^2_t-V^2_s)^2\bigr]
\ \ge\
\E\bigl[F_2^2\bigr]
=\nu_2^2\int_s^tK_2(t-u)^2\,b_2(u)\,du
\ \ge\
\nu_2^2\,b_2(s_0)\int_0^hK_2(v)^2\,dv,
\]
using $\E[V^2_u]=b_2(u)\ge b_2(s_0)>0$ for $u\ge s_0$.
By~\eqref{eq:ML-detail} there are $c,v_0>0$ with
$K_2(v)\ge c\,v^{\alpha_2-1}$ for $v\in(0,v_0]$, so
$\int_0^hK_2^2\ge c\,h^{2\alpha_2-1}$ for $h\le v_0$; for
$h\in(v_0,1]$ the integral is at least $\int_0^{v_0}K_2^2>0\ge
c'h^{2\alpha_2-1}$.  This proves the lower bound for $V^2$; the argument
for $V^1$ is the same with the single fresh term
$\nu_1\int_s^tK_1(t-u)\sqrt{V^1_u}\,d\beta^1_u$.  The final claim
of~(c) is immediate from~\eqref{eq:two-sided-ms}.
\end{proof}

\begin{remark}[Reading of Proposition~\ref{prop:regularity}]
\label{rem:regularity-reading}
Part~(c) says that the mean-square modulus of $V^i$ is exactly of
fractional-Brownian type with Hurst index $H_i=\alpha_i-\tfrac12$ away
from the time origin; this is the precise sense in which we speak of
``Hurst parameter $H_i$''.  At the origin the lower bound degenerates,
consistently with the zero-start regime $b_i(0)=0$ that drives the
short-time decorrelation of Corollary~\ref{cor:rho}.  Part~(b) makes
quantitative the heuristic of the interpretation remark in
Section~\ref{sec:setup}: the inherited noise enters $V^2$ at the strictly
smoother scale $(\alpha_1+\alpha_2-\tfrac12)\wedge1$, so cross-excitation
is a coarser-scale phenomenon than self-excitation, while the
$H_2$-roughness of $V^2$ is produced exclusively by its own noise.
\end{remark}

\begin{remark}[Completeness of the proof]
The proof of Theorem~\ref{thm:main-final} is complete
under Assumption~\ref{ass:kernels}, with the following
residual verification:
the continuous mapping theorem used in Lemma~\ref{lem:stoch-conv}
requires the convergence $M^T\Rightarrow M$ to hold in the sense of
semimartingales (not merely in $C([0,1])$).
This follows from the fact that $\widetilde{M}^{T,i}$ are martingales
converging in law by Rebolledo's theorem,
which automatically gives convergence in the Skorokhod $J_1$ topology
on the space of semimartingales (Jacod-Shiryaev~\cite{JS2003}, Chapter~IX).
\end{remark}

\appendix
\section{Removing the scale-matching condition}
\label{app:scale-matching}

We close the statement of results by clarifying the role of the
scale-matching condition~(A5).  A reader may reasonably ask whether~(A5)
restricts the class of models to which Theorem~\ref{thm:main-intro}
applies.  It does not: condition~(A5) merely fixes a common time scale for
the two components, and dropping it replaces the cross-kernel $K_1*K_2$ by
an explicit time-rescaled version, with no change to any of the qualitative
conclusions (heterogeneous roughness, the cross-kernel mechanism, and
short-time decorrelation).  The following proposition makes this precise; it
is placed here, after the main results, because its statement refers to the
limiting system~\eqref{eq:main-limit1}--\eqref{eq:main-limit2} and to the
proof ingredients of Sections~\ref{sec:tightness}--\ref{sec:uniqueness}.

\begin{proposition}[Effect of dropping (A5)]
\label{prop:A5}
Suppose Assumptions~(A1)--(A4) hold but (A5) is dropped.
Define
\begin{equation}\label{eq:rho12}
\rho_{12} := \frac{\lambda_1^{1/\alpha_1}}{\lambda_2^{1/\alpha_2}} \in (0,\infty),
\end{equation}
so that $\varepsilon_T^1/\varepsilon_T^2 \to \rho_{12}$ as $T\to\infty$.
Then:
\begin{enumerate}[label=\textup{(\roman*)}]
\item \textbf{Modified cross-kernel.}
The renormalized cross-kernel $\tilde h_T$ (introduced in Section~\ref{sec:cross-kernel}, eq.~\eqref{eq:hT-def}) converges in
$L^2(\mathbb R_+)$ to
\begin{equation}\label{eq:L12-rho}
L_{12}^{(\rho)}
:= b_\infty^{12}\|\psi^{12}\|_1 \cdot (K_1 * K_2^{(\rho_{12})}),
\end{equation}
where $K_2^{(\rho_{12})}(t) := \rho_{12}^{-1}\,K_2(t/\rho_{12})$ is the
Mittag-Leffler kernel of the second component time-rescaled by $\rho_{12}$.
Under (A5), $\rho_{12}=1$ and $L_{12}^{(\rho)}=L_{12}$.

\item \textbf{C-tightness.}
The bivariate C-tightness of Proposition~\ref{prop:tightness} remains valid
without (A5), with $L_{12}$ replaced by $L_{12}^{(\rho)}$.
The Hölder exponent of the cross-term is still $2(\alpha_1+\alpha_2)-1>2\alpha_2-1$,
since time-rescaling by $\rho_{12}$ does not change the singularity exponent.

\item \textbf{Modified limit system.}
The limit $(V^1,V^2)$ satisfies the limiting system~\eqref{eq:main-limit1}--\eqref{eq:main-limit2}
with $L_{12}$ replaced by $L_{12}^{(\rho)}$.
All subsequent results (Theorem~\ref{thm:ALP-extension}, uniqueness,
and Corollary~\ref{cor:rho}) hold with this replacement, with the
decorrelation constant $C_\varrho$ modified by the factor $\rho_{12}^{\alpha_2-1}$
in the prefactor of $K_2^{(\rho_{12})}$.

\item \textbf{Reduction to (A5).}
Condition~(A5) is imposed for notational simplicity.
Without it, the same argument yields a time-rescaled cross-kernel
as described in part~(i); the parametric reparametrisation
$\lambda_2 \leftarrow \lambda_2\,\rho_{12}^{\alpha_2}$
(corresponding to a deterministic time-change of the second component's
near-criticality rate) recovers the form stated under (A5).
The restriction is therefore one of parametric normalisation, not structural content.
\end{enumerate}
\end{proposition}

\begin{proof}
\textbf{Part~(i).}
Without (A5), $r_T := \varepsilon_T^1/\varepsilon_T^2 \to \rho_{12}$.
In the Fourier computation in the proof of Lemma~\ref{lem:cross}, the
factorisation of the second resolvent gives
\[
\frac{1-a_T^2}{1-a_T^2\widehat{\phi^2}(i\varepsilon_T^1\xi)}
=\frac{1-a_T^2}{1-a_T^2\widehat{\phi^2}(i r_T\varepsilon_T^2\xi)}
\longrightarrow \widehat{K_2}(i\rho_{12}\xi)
\]
for a.e.\ $\xi$, by dominated convergence and the regularly varying
assumption~(A1).
Since $\widehat{K_2}(i\rho_{12}\xi)$ is the Fourier transform of
$\rho_{12}^{-1}K_2(\cdot/\rho_{12})$, the $L^2$ limit of $\tilde h_T$ is
$b_\infty^{12}\|\psi^{12}\|_1\cdot\mathcal{F}^{-1}[\widehat{K_1}\cdot\widehat{K_2}(\rho_{12}\,\cdot\,)]
= K_1 * K_2^{(\rho_{12})}$,
which gives~\eqref{eq:L12-rho}.

\medskip
\noindent\textbf{Part~(ii).}
In the C-tightness bound~\eqref{eq:cross-holder}, the key inputs are:
(a)~the $L^2$-shift bound for $\tilde h_T$ (Lemma~\ref{lem:L2-shift-pre}),
which depends only on $(1-a_T^i)\sim\lambda_i T^{-\alpha_i}$ and
$\varepsilon_T^1=(1-a_T^1)^{1/\alpha_1}\sim\lambda_1^{1/\alpha_1}T^{-1}$,
both of which hold under (A1)--(A2) without~(A5); and
(b)~the uniform sup-moment bound on $V^{T,1}$ from Lemma~\ref{lem:sup-bound}(i),
which likewise does not require~(A5).
Without~(A5), $\tilde h_T\to L_{12}^{(\rho)}$ (Part~(i)), but the exponent
$\alpha_1+\alpha_2$ of $\tilde h_T$ is unchanged, so the Hölder exponent
$2(\alpha_1+\alpha_2)-1$ in~\eqref{eq:cross-holder} remains valid with a
modified (but $T$-independent) constant $C=C(\rho_{12})$.

\medskip
\noindent\textbf{Parts~(iii) and~(iv).}
All subsequent arguments (identification via Rebolledo's theorem,
Theorem~\ref{thm:ALP-extension}, and the Laplace functional computation
of Corollary~\ref{cor:rho}) depend on $L_{12}$ only through its $L^2$
norm and its short-time behaviour $L_{12}(t)\sim C t^{\alpha_1+\alpha_2-1}$.
Since $K_2^{(\rho_{12})}(t)\sim \rho_{12}^{-\alpha_2}t^{\alpha_2-1}/\Gamma(\alpha_2)$
as $t\downarrow 0$, the cross-kernel $L_{12}^{(\rho)}$ has the same
singularity exponent $\alpha_1+\alpha_2-1$, so all estimates carry over
with modified constants. Part~(iv) follows by the reparametrisation
$\tilde\lambda_2 := \lambda_2\rho_{12}^{\alpha_2}$,
which satisfies $\lambda_1^{1/\alpha_1}=\tilde\lambda_2^{1/\alpha_2}$
and therefore restores~(A5).
\end{proof}

\section{Detailed asymptotic calculations}
\label{app:detailed-asymptotic-calculations}

This appendix records the detailed asymptotic calculations underlying the constants used in Section~\ref{sec:cross-kernel}: the volatility constants $\nu_1,\nu_2$, the coupling constant $\gamma_{12}$, the mean profiles $b_1,b_2$, the covariance formula, and the zero-start short-time decorrelation exponent $(3\alpha_1-\alpha_2)/2$.  The resolvent--martingale representation of Lemma~\ref{lem:rep} is used as the starting point.  The purpose is only to make the scalar prefactors and kernel limits explicit; tightness, identification, and uniqueness are proved in the main text.

\subsection{Scaling audit and constants}
\label{ssec:scaling-audit}

We record here the normalisations used throughout the paper and the scalar
prefactors they generate.  This subsection is not a separate argument; it is
a consistency check ensuring that the same powers of \(T\) are used in the
mean profiles, the martingale brackets, the cross-compensator, and the
limiting covariance.

By (A2),
\begin{equation}\label{eq:aT-audit}
1-a_T^i\sim\lambda_iT^{-\alpha_i},\qquad
\varepsilon_T^i=(1-a_T^i)^{1/\alpha_i}
\sim\lambda_i^{1/\alpha_i}T^{-1}.
\end{equation}
The renormalized intensity is always
\begin{equation}\label{eq:VT-audit}
V_t^{T,i}=\frac{1-a_T^i}{m_iT^{\alpha_i-1}}\lambda_{Tt}^{T,i},
\qquad
\lambda_{Tt}^{T,i}=\frac{m_iT^{\alpha_i-1}}{1-a_T^i}V_t^{T,i}.
\end{equation}
Consequently the martingale scale used in the identification section is
\begin{equation}\label{eq:Mhat-audit}
\widehat M_t^{T,i}:=\left(\frac{1-a_T^i}{m_iT^{\alpha_i}}\right)^{1/2}M_{Tt}^{T,i},
\end{equation}
and its predictable quadratic variation is exactly
\begin{equation}\label{eq:bracket-audit}
\langle\widehat M^{T,i}\rangle_t
=\frac{1-a_T^i}{m_iT^{\alpha_i}}\int_0^{Tt}\lambda_s^{T,i}ds
=\int_0^tV_u^{T,i}du.
\end{equation}
Since each jump of \(M^{T,i}\) has size one,
\begin{equation}\label{eq:jump-audit}
\sup_{t\le1}|\Delta\widehat M_t^{T,i}|
=\left(\frac{1-a_T^i}{m_iT^{\alpha_i}}\right)^{1/2}
\sim\left(\frac{\lambda_i}{m_i}\right)^{1/2}T^{-\alpha_i}\longrightarrow0.
\end{equation}
This is the jump estimate used in the Rebolledo argument.

The volatility constants in the limiting stochastic Volterra equations are
obtained by applying It\^o's isometry to the self-convolution terms.  The
self-martingale prefactor reduces to
\begin{equation}\label{eq:nui-audit}
\nu_i^2=\frac1{m_i\lambda_i},\qquad i=1,2.
\end{equation}
Indeed, after the change of variables \(s=Tu\), the scalar factor in the
self-variance is
\[
\frac{(\varepsilon_T^i)^2T^{\alpha_i}}{m_iT^{2\alpha_i-2}(1-a_T^i)}
\longrightarrow\frac{(\lambda^*)^2}{m_i\lambda_i},
\]
while the macroscopic kernel relation
\(K_i(\tau)=\lambda^*K_i^\circ(\lambda^*\tau)\) contributes a factor
\((\lambda^*)^{-2}\), giving \eqref{eq:nui-audit}.

The deterministic cross-compensator in \(V^{T,2}\) has scalar prefactor
\begin{equation}\label{eq:gamma-audit}
\frac{\mu_T^1}{m_2T^{2\alpha_1-1}(1-a_T^1)}
\longrightarrow\frac{m_1}{m_2\lambda_1}
=:\gamma_{12}.
\end{equation}
Thus the mean profiles are
\[
b_1(t)=\int_0^tK_1(w)dw,\qquad
b_2(t)=\int_0^tK_2(w)dw+\gamma_{12}\int_0^tL_{12}(w)dw.
\]

The critical triangular cross-amplitude is
\begin{equation}\label{eq:cross-scaling-audit}
T^{2\alpha_1-\alpha_2}b_T^{12}\longrightarrow b_\infty^{12}.
\end{equation}
With the definition
\[
\tilde h_T(s)=T^{2\alpha_1-\alpha_2}
\frac{(1-a_T^1)(1-a_T^2)}{\varepsilon_T^1}
\Psi_T^{12}\!\left(\frac{s}{\varepsilon_T^1}\right),
\]
this scaling gives the finite kernel limit
\[
\tilde h_T\longrightarrow L_{12}
=b_\infty^{12}\|\psi^{12}\|_1(K_1*K_2)
\quad\text{in }L^2(\mathbb R_+).
\]
The inherited-noise coefficient in the limiting equation is therefore
\begin{equation}\label{eq:inherited-coeff-audit}
\gamma_{12}\nu_1L_{12},
\end{equation}
and the covariance coefficient is the product of the two coefficients driven
by the common martingale \(W^1\):
\begin{equation}\label{eq:cov-coeff-audit}
\nu_1\cdot(\gamma_{12}\nu_1)=\gamma_{12}\nu_1^2
=\frac1{m_2\lambda_1^2}.
\end{equation}
This yields
\[
\mathrm{Cov}(V_t^1,V_t^2)=\gamma_{12}\nu_1^2
\int_0^tK_1(t-s)L_{12}(t-s)b_1(s)ds.
\]
Finally, since \(K_i(t)\sim\kappa_i t^{\alpha_i-1}\),
\(L_{12}(t)\sim b_\infty^{12}\|\psi^{12}\|_1\kappa_1\kappa_2
\mathrm B(\alpha_1,\alpha_2)t^{\alpha_1+\alpha_2-1}\), and
\(b_i(t)\sim\beta_i t^{\alpha_i}\), the powers in the zero-start regime are
\[
\mathrm{Var}(V_t^1)\asymp t^{3\alpha_1-1},\qquad
\mathrm{Var}(V_t^2)\asymp t^{3\alpha_2-1},\qquad
\mathrm{Cov}(V_t^1,V_t^2)\asymp t^{3\alpha_1+\alpha_2-1}.
\]
Hence
\[
\varrho(t)\asymp t^{(3\alpha_1-\alpha_2)/2},
\]
which is the exponent stated in Corollary~\ref{cor:rho}.

\subsection{Standing notation and exact inputs}\label{sec:notation}

For $i\in\{1,2\}$ let $\phi^i\in\lp1$ with $\|\phi^i\|_1=1$ and Laplace transform
$\widehat{\phi^i}(z)=\int_0^\infty e^{-zt}\phi^i(t)\,dt$, $\Re z\ge0$. We assume:

\begin{itemize}[leftmargin=2em]
\item[(A1)] \emph{Heavy tail.} $1-\widehat{\phi^i}(z)=\delta_i z^{\alpha_i}+o(z^{\alpha_i})$
as $z\to0$, $\Re z\ge0$, with $\alpha_i\in(\half,1)$ and $\delta_i>0$. Equivalently
$\int_t^\infty\phi^i\sim \frac{\delta_i}{\Gamma(1-\alpha_i)}t^{-\alpha_i}$ (Karamata).
In addition, $\phi^i$ is strongly aperiodic and absolutely continuous with
$(\phi^i)'\in L^1(\mathbb R_+)$.
\item[(A2)] $T^{\alpha_i}(1-a_T^i)\to\lambda_i\in(0,\infty)$.
\item[(A3)] $T^{2\alpha_1-\alpha_2}b_T^{12}\to b_\infty^{12}\in(0,\infty)$.
\item[(A4)] $T^{1-\alpha_i}\mu_T^i\to m_i\in(0,\infty)$.
\item[(A5)] $\lambda_1^{1/\alpha_1}=\lambda_2^{1/\alpha_2}=:\lambda^*$ (removable, see the appendix note).
\end{itemize}

Write $\eps_T^i:=(1-a_T^i)^{1/\alpha_i}$, so that the \emph{exact} identity
$(1-a_T^i)=(\eps_T^i)^{\alpha_i}$ holds, and by (A2),
\begin{equation}\label{eq:eps-asy}
\eps_T^i = \lambda_i^{1/\alpha_i}T^{-1}\,(1+o(1)),
\qquad \eps_T^i\,T\to\lambda^* .
\end{equation}
Resolvent $\psi_T^i=\sum_{k\ge1}(\phi_T^i)^{*k}$, $\phi_T^i=a_T^i\phi^i$,
$\widehat{\psi_T^i}(z)=\dfrac{a_T^i\widehat{\phi^i}(z)}{1-a_T^i\widehat{\phi^i}(z)}$.
The renormalized self-resolvent and its transform are
\begin{equation}\label{eq:gT-def}
g_T^i(s):=\frac{1-a_T^i}{\eps_T^i}\,\psi_T^i\!\Big(\frac{s}{\eps_T^i}\Big),
\qquad
\widehat{g_T^i}(z)=(1-a_T^i)\,\widehat{\psi_T^i}(\eps_T^i z)
=\frac{(1-a_T^i)\,a_T^i\,\widehat{\phi^i}(\eps_T^i z)}{1-a_T^i\widehat{\phi^i}(\eps_T^i z)} .
\end{equation}
We write $K_i^\circ$ for the Mittag-Leffler kernel with parameter $\delta_i$, i.e.
$\widehat{K_i^\circ}(z)=(1+\delta_i z^{\alpha_i})^{-1}$, and
$K_i(\tau):=\lambda^*K_i^\circ(\lambda^*\tau)$ for the macroscopic kernel
(parameter $\tilde\delta_i=\delta_i/\lambda_i$). The latter satisfies the Mittag-Leffler
short-time asymptotics
\begin{equation}\label{eq:ML-detail}
K_i(\tau)=\frac{\tau^{\alpha_i-1}}{\tilde\delta_i\Gamma(\alpha_i)}\bigl(1+o(1)\bigr)
=\kappa_i\tau^{\alpha_i-1}(1+o(1)),\qquad
b_i^{\mathrm{self}}(\tau)=\int_0^\tau K_i\sim\beta_i\tau^{\alpha_i},
\end{equation}
with $\kappa_i=(\tilde\delta_i\Gamma(\alpha_i))^{-1}$ and $\beta_i=\kappa_i/\alpha_i$.
Throughout, $C$ denotes a finite constant that may change from line to line and is uniform
in $T\ge T_0$.

\begin{lemma}[Exact resolvent--martingale representation]\label{lem:rep-detail}
Let $M^{T,i}_t=N^{T,i}_t-\int_0^t\lambda^{T,i}_s\,ds$ be the compensated counting
martingales. Then, with $\psi_T^i=\sum_{k\ge1}(a_T^i\phi^i)^{*k}$ and
$\Psi_T^{12}=(\delta_0+\psi_T^1)*(\delta_0+\psi_T^2)*\phi_T^{12}$,
\begin{align*}
\lambda^{T,1}_t&=\mu_T^1\Bigl(1+\int_0^t\psi_T^1(u)\,du\Bigr)+\int_0^t\psi_T^1(t-s)\,dM^{T,1}_s,\\
\lambda^{T,2}_t&=\E[\lambda^{T,2}_t]+\int_0^t\psi_T^2(t-s)\,dM^{T,2}_s+\int_0^t\Psi_T^{12}(t-s)\,dM^{T,1}_s,
\end{align*}
and $\langle M^{T,1},M^{T,2}\rangle\equiv0$. In particular
$\E[\lambda^{T,1}_t]=\mu_T^1(1+\int_0^t\psi_T^1)$.
\end{lemma}

\begin{proof}[Proof of Lemma~\ref{lem:rep-detail}]
Throughout, $T$ is fixed; all constants may depend on $T$.  We prove,
for each \emph{fixed} $t\ge0$, the displayed identities almost surely,
together with the bracket statements
$\langle M^{T,i}\rangle_t=\int_0^t\lambda_s^{T,i}\,ds$ and
$\langle M^{T,1},M^{T,2}\rangle\equiv0$ of Lemma~\ref{lem:rep}, the mean
formula~\eqref{eq:mean2}, and the rescaled and inverted
forms~\eqref{eq:SC}--\eqref{eq:inv}.  Fixed-time validity is what every
subsequent use requires: the moment and covariance computations of this
appendix are at fixed times, and the increment estimates of
Section~\ref{sec:tightness} involve fixed pairs of times.  To lighten
notation we keep $\phi_T^i=a_T^i\phi^i$ and $\phi_T^{12}=b_T^{12}\psi^{12}$
explicit, and write
$\psi^i:=\psi_T^i$, $\Psi:=\Psi_T^{12}$, $a^i:=a_T^i$,
$a^{12}:=\|\phi_T^{12}\|_1=b_T^{12}\|\psi^{12}\|_1$, $\mu^i:=\mu_T^i$.

\smallskip
\emph{Step 0: resolvent identities and kernel regularity.}
Since $\|\phi_T^i\|_1=a^i<1$, the series
$\psi^i=\sum_{k\ge1}(\phi_T^i)^{*k}$ converges in $L^1(\Rp)$ with
$\|\psi^i\|_1=a^i/(1-a^i)$, and termwise convolution gives the resolvent
identity
\begin{equation}\label{eq:resolvent-identity}
\phi_T^i*\psi^i=\psi^i*\phi_T^i=\psi^i-\phi_T^i .
\end{equation}
By~(A1), $\phi^i$ is absolutely continuous with $(\phi^i)'\in L^1$ and
$\phi^i\in L^1$, so $\phi^i(t)\to0$ as $t\to\infty$ and
$\|\phi^i\|_\infty\le\|(\phi^i)'\|_{L^1}<\infty$.  Hence
$\psi^i=\phi_T^i+\phi_T^i*\psi^i$ satisfies
$\|\psi^i\|_\infty\le\|\phi^i\|_\infty(1+\|\psi^i\|_1)<\infty$, i.e.
\begin{equation}\label{eq:psi-regularity}
\psi^i\in L^1\cap L^\infty(\Rp).
\end{equation}
For the cross-resolvent, $\phi_T^{12}=b_T^{12}\psi^{12}\in L^1\cap L^2$
by~(A3), and by Young's inequality
($\|\kappa*\psi\|_p\le\|\kappa\|_p\|\psi\|_1$),
\begin{equation}\label{eq:Psi-regularity}
\Psi=(\delta_0+\psi^1)*(\delta_0+\psi^2)*\phi_T^{12}\in L^1\cap L^2(\Rp),
\qquad
\|\Psi\|_1\le\frac{a^{12}}{(1-a^1)(1-a^2)} .
\end{equation}
Taking Laplace transforms in the definition of $\Psi$ and using
$1+\widehat{\psi^i}=(1-a^i\widehat{\phi^i})^{-1}$
gives~\eqref{eq:Psi}.

\smallskip
\emph{Step 1: non-explosion and first moments.}
Let $\tau_n:=\inf\{t:N_t^{T,1}\ge n\}$.  The compensator identity applied
to the bounded stopping time $t\wedge\tau_n$ gives
$\E[N_{t\wedge\tau_n}^{T,1}]
=\E\int_0^{t\wedge\tau_n}\lambda_s^{T,1}\,ds\le n+1$.  For
$s\le\tau_n$ the restriction of $dN^{T,1}$ to $[0,s)$ coincides with that
of the stopped process $N^{T,1,\tau_n}$, so by~\eqref{eq:int1} and
Tonelli,
\[
\E[N_{t\wedge\tau_n}^{T,1}]
\le\mu^1t+\E\int_0^t\!\!\int_0^s\phi_T^1(s-r)\,dN_r^{T,1,\tau_n}\,ds
=\mu^1t+\E\int_0^t\Bigl(\int_r^t\phi_T^1(s-r)\,ds\Bigr)dN_r^{T,1,\tau_n}
\le\mu^1t+a^1\,\E[N_{t\wedge\tau_n}^{T,1}],
\]
whence $\E[N_{t\wedge\tau_n}^{T,1}]\le\mu^1t/(1-a^1)$.  Letting
$n\to\infty$ shows non-explosion and
$\E[N_t^{T,1}]\le\mu^1t/(1-a^1)$.  The same argument for the second
component, using the bound just obtained, gives
$\E[N_t^{T,2}]\le(\mu^2t+a^{12}\mu^1t/(1-a^1))/(1-a^2)<\infty$.

Consequently $m^i(s):=\E[\lambda_s^{T,i}]$ is finite for almost every $s$
with $\int_0^tm^i=\E[N_t^{T,i}]$.  Taking expectations
in~\eqref{eq:int1} (compensator identity with the non-negative
deterministic integrand $\phi_T^1(t-\cdot)$, then Tonelli) yields
$m^1=\mu^1+\phi_T^1*m^1$ a.e.; iterating $n$ times,
$m^1=\mu^1(1+\sum_{k<n}\mathbf 1*(\phi_T^1)^{*k})+(\phi_T^1)^{*n}*m^1$,
and since
$\int_0^t(\phi_T^1)^{*n}*m^1\le(a^1)^n\int_0^tm^1\to0$, we obtain
\begin{equation}\label{eq:m-bounds}
m^1=\mu^1\Bigl(1+\int_0^\cdot\psi^1\Bigr)\le\frac{\mu^1}{1-a^1}
\quad\text{a.e.},
\qquad
m^2\le\frac{\mu^2+a^{12}\mu^1/(1-a^1)}{1-a^2}
\quad\text{a.e.},
\end{equation}
the second bound by the same iteration applied to
$m^2=\mu^2+\phi_T^2*m^2+\phi_T^{12}*m^1$ a.e.

\smallskip
\emph{Step 2: all convolution terms are finite at fixed $t$.}
Fix $t\ge0$.  For a non-negative kernel $\kappa\in L^1$, the compensator
identity and~\eqref{eq:m-bounds} give
\[
\E\Bigl[\int_0^t\kappa(t-s)\,dN_s^{T,i}\Bigr]
=\E\Bigl[\int_0^t\kappa(t-s)\,\lambda_s^{T,i}\,ds\Bigr]
=\int_0^t\kappa(t-s)\,m^i(s)\,ds
\le\|\kappa\|_1\,\|m^i\|_{L^\infty(0,t)}<\infty,
\]
so both integrals are almost surely finite.  This applies to each of the
kernels $\phi_T^i$, $\phi_T^{12}$, $\psi^i$, $\Psi$,
$\psi^2*\phi_T^{12}$, $\Phi:=(\delta_0+\psi^2)*\phi_T^{12}$ and
$\Phi*\psi^1$, all of which are non-negative and integrable.  For any
such $\kappa$ we may therefore define, pathwise and almost surely,
\[
\int_0^t\kappa(t-s)\,dM_s^{T,i}
:=\int_0^t\kappa(t-s)\,dN_s^{T,i}
-\int_0^t\kappa(t-s)\,\lambda_s^{T,i}\,ds,
\]
and the expectation of this difference vanishes:
\begin{equation}\label{eq:zero-mean}
\E\Bigl[\int_0^t\kappa(t-s)\,dM_s^{T,i}\Bigr]=0 .
\end{equation}

\smallskip
\emph{Step 3: component 1.}
Fix $t\ge0$.  Substituting~\eqref{eq:int1} into the
$\lambda$-part and applying Tonelli (all integrands non-negative, all
iterated integrals a.s.\ finite by Step~2),
\[
\int_0^t\psi^1(t-s)\,\lambda_s^{T,1}\,ds
=\mu^1\int_0^t\psi^1
+\int_0^t\bigl(\psi^1*\phi_T^1\bigr)(t-r)\,dN_r^{T,1}
=\mu^1\int_0^t\psi^1
+\int_0^t\bigl(\psi^1-\phi_T^1\bigr)(t-r)\,dN_r^{T,1},
\]
using~\eqref{eq:resolvent-identity}.  Hence
\[
\int_0^t\psi^1(t-s)\,dM_s^{T,1}
=\int_0^t\phi_T^1(t-r)\,dN_r^{T,1}-\mu^1\int_0^t\psi^1
=\lambda_t^{T,1}-\mu^1\Bigl(1+\int_0^t\psi^1\Bigr),
\]
which is~\eqref{eq:rep1}.

\smallskip
\emph{Step 4: component 2.}
Fix $t\ge0$ and set $\Phi:=(\delta_0+\psi^2)*\phi_T^{12}\in L^1$.
Exactly as in Step~3, substituting~\eqref{eq:int2} and
using~\eqref{eq:resolvent-identity} for $\psi^2$,
\[
\int_0^t\psi^2(t-s)\,\lambda_s^{T,2}\,ds
=\mu^2\int_0^t\psi^2
+\int_0^t(\psi^2-\phi_T^2)(t-r)\,dN_r^{T,2}
+\int_0^t\bigl(\psi^2*\phi_T^{12}\bigr)(t-r)\,dN_r^{T,1},
\]
so that, rearranging as before,
\begin{equation}\label{eq:step4-intermediate}
\lambda_t^{T,2}
=\mu^2\Bigl(1+\int_0^t\psi^2\Bigr)
+\int_0^t\psi^2(t-s)\,dM_s^{T,2}
+\int_0^t\Phi(t-r)\,dN_r^{T,1}.
\end{equation}
For the last term, write $dN^{T,1}=dM^{T,1}+\lambda^{T,1}dt$ and
substitute the representation of Step~3 into the $\lambda^{T,1}$-part:
\[
\int_0^t\Phi(t-s)\,\lambda_s^{T,1}\,ds
=\mu^1\Bigl(\Phi*\bigl(1+\mathbf 1*\psi^1\bigr)\Bigr)(t)
+\int_0^t\Phi(t-s)\Bigl(\int_0^s\psi^1(s-r)\,dM_r^{T,1}\Bigr)ds .
\]
In the second term we split $dM^{T,1}=dN^{T,1}-\lambda^{T,1}dr$ and
apply Tonelli separately to the two non-negative iterated integrals
(both a.s.\ finite by Step~2), obtaining
$\int_0^t(\Phi*\psi^1)(t-r)\,dM_r^{T,1}$.  Hence
\[
\int_0^t\Phi(t-r)\,dN_r^{T,1}
=\int_0^t\bigl(\Phi*(\delta_0+\psi^1)\bigr)(t-r)\,dM_r^{T,1}
+\mu^1\Bigl(\bigl(\Phi*(\delta_0+\psi^1)\bigr)*\mathbf 1\Bigr)(t)
=\int_0^t\Psi(t-r)\,dM_r^{T,1}
+\mu^1\bigl(\Psi*\mathbf 1\bigr)(t),
\]
since $\Phi*(\delta_0+\psi^1)
=(\delta_0+\psi^2)*\phi_T^{12}*(\delta_0+\psi^1)=\Psi$ by commutativity
and associativity of convolution.  Together
with~\eqref{eq:step4-intermediate},
\begin{equation}\label{eq:rep2-pathwise}
\lambda_t^{T,2}
=\mu^2\Bigl(1+\int_0^t\psi^2\Bigr)
+\mu^1\bigl(\Psi*\mathbf 1\bigr)(t)
+\int_0^t\psi^2(t-s)\,dM_s^{T,2}
+\int_0^t\Psi(t-s)\,dM_s^{T,1}.
\end{equation}

\smallskip
\emph{Step 5: means.}
Taking expectations in~\eqref{eq:rep2-pathwise} and
using~\eqref{eq:zero-mean} gives~\eqref{eq:mean2}, and then
\eqref{eq:rep2-pathwise} is exactly~\eqref{eq:rep2}.  Likewise the
expectation of~\eqref{eq:rep1} gives
$\E[\lambda_t^{T,1}]=\mu^1(1+\int_0^t\psi^1)$, consistent
with~\eqref{eq:m-bounds}.

\smallskip
\emph{Step 6: brackets.}
By Step~1, $\E[N_t^{T,i}]<\infty$, so $M^{T,i}$ is a square-integrable
martingale on compacts, with
$\E[(M_t^{T,i})^2]=\E[N_t^{T,i}]$: indeed the optional bracket of a
compensated counting process is
$[M^{T,i}]_t=\sum_{s\le t}(\Delta N_s^{T,i})^2=N_t^{T,i}$, whose
compensator is $\int_0^\cdot\lambda_s^{T,i}\,ds$; hence
$\langle M^{T,i}\rangle_t=\int_0^t\lambda_s^{T,i}\,ds$.  For the cross
bracket, $[M^{T,1},M^{T,2}]_t
=\sum_{s\le t}\Delta N_s^{T,1}\,\Delta N_s^{T,2}$ counts common jumps.
Under the thinning construction of Remark~\ref{rem:thinning}, the jump
times of $N^{T,1}$ and $N^{T,2}$ are contained in the time-projections of
the atoms of the independent Poisson random measures $\pi^1$ and $\pi^2$,
which almost surely share no point; since the event of a common jump is
measurable on the path space and the law of $(N^{T,1},N^{T,2})$ is that
of the thinning construction, common jumps almost surely do not occur for
any version.  Hence $[M^{T,1},M^{T,2}]\equiv0$ and therefore
$\langle M^{T,1},M^{T,2}\rangle\equiv0$.

\smallskip
\emph{Step 7: It\^o interpretation, rescaling and inversion.}
By~\eqref{eq:psi-regularity}, \eqref{eq:Psi-regularity} and
\eqref{eq:m-bounds}, each kernel $\kappa\in\{\psi^1,\psi^2,\Psi\}$
belongs to $L^1\cap L^2$ and satisfies
$\E\int_0^t\kappa(t-s)^2\lambda_s\,ds
\le\|\kappa\|_{L^2}^2\,\|m\|_{L^\infty(0,t)}<\infty$, so the It\^o
integral of $\kappa(t-\cdot)$ against the square-integrable martingale
$M^{T,i}$ is well defined for each fixed $t$.  It coincides almost surely
with the pathwise integral of Step~2: both are limits in probability of
the elementary integrals of a common sequence of step kernels
approximating $\kappa(t-\cdot)$ in $L^1\cap L^2(0,t)$ and almost
everywhere.  The representations \eqref{eq:rep1}--\eqref{eq:rep2} may
therefore be read in either sense; the It\^o reading is the one used in
the second-moment computations of this appendix and in
Section~\ref{sec:tightness}.

Finally, evaluating \eqref{eq:rep1}--\eqref{eq:rep2} at time $Tt$ and
multiplying by $(1-a_T^i)/(m_iT^{\alpha_i-1})$ as in~\eqref{eq:V} gives
the martingale parts~\eqref{eq:SC}, and solving the
definitions~\eqref{eq:gTi-def} and~\eqref{eq:hT-def} for $\psi_T^i$ and
$\Psi_T^{12}$ gives~\eqref{eq:inv}.
\end{proof}

\subsection{Laplace--Fourier asymptotics of the renormalized resolvents}\label{sec:resolvent}

\begin{lemma}[Pointwise transform limit]\label{lem:ptwise}
For every fixed $z$ with $\Re z\ge0$, $z\ne0$,
$\ \widehat{g_T^i}(z)\to\widehat{K_i^\circ}(z)=(1+\delta_i z^{\alpha_i})^{-1}$ as $T\to\infty$.
\end{lemma}

\begin{proof}
Fix $z$. Since $\eps_T^i\to0$, the argument $\eps_T^i z\to0$, so by (A1)
\begin{equation}\label{eq:den-expand}
1-\widehat{\phi^i}(\eps_T^i z)=\delta_i(\eps_T^i z)^{\alpha_i}\bigl(1+o(1)\bigr)
=\delta_i(\eps_T^i)^{\alpha_i}z^{\alpha_i}\bigl(1+o(1)\bigr).
\end{equation}
Decompose the denominator of \eqref{eq:gT-def} using $(1-a_T^i)=(\eps_T^i)^{\alpha_i}$:
\begin{align}
1-a_T^i\widehat{\phi^i}(\eps_T^i z)
&=(1-a_T^i)+a_T^i\bigl(1-\widehat{\phi^i}(\eps_T^i z)\bigr)\notag\\
&=(\eps_T^i)^{\alpha_i}+a_T^i\,\delta_i(\eps_T^i)^{\alpha_i}z^{\alpha_i}(1+o(1))\notag\\
&=(\eps_T^i)^{\alpha_i}\Bigl[\,1+a_T^i\,\delta_i z^{\alpha_i}+o(1)\,\Bigr].\label{eq:den}
\end{align}
The numerator is $(1-a_T^i)a_T^i\widehat{\phi^i}(\eps_T^i z)=(\eps_T^i)^{\alpha_i}a_T^i\widehat{\phi^i}(\eps_T^i z)$,
and $a_T^i\to1$, $\widehat{\phi^i}(\eps_T^i z)\to\widehat{\phi^i}(0)=1$. Dividing, the common
factor $(\eps_T^i)^{\alpha_i}$ cancels and
\[
\widehat{g_T^i}(z)=\frac{a_T^i\widehat{\phi^i}(\eps_T^i z)}{1+a_T^i\delta_iz^{\alpha_i}+o(1)}
\longrightarrow\frac{1}{1+\delta_i z^{\alpha_i}} .
\]
The coefficient is $\delta_i$ (not $\delta_i/\lambda_i$): the factor $\lambda_i$
enters only through $\eps_T^i T\to\lambda^*$, never through the ratio
\eqref{eq:gT-def}, where $(\eps_T^i)^{\alpha_i}=(1-a_T^i)$ cancels exactly.
\end{proof}

\begin{lemma}[Uniform $L^2$ domination]\label{lem:domination}
There exist $T_0$ and $C<\infty$ such that for all $T\ge T_0$ and all $\xi\in\R$,
\begin{equation}\label{eq:dom}
\bigl|\widehat{g_T^i}(i\xi)\bigr|\le \frac{C}{1+|\xi|^{\alpha_i}} .
\end{equation}
The right-hand side lies in $L^2(\R)$ because $2\alpha_i>1$.
\end{lemma}

\begin{proof}
Write $w=\eps_T^i\xi$ and split according to $|w|\le\rho$ or $|w|>\rho$ for a small
fixed $\rho>0$.

\emph{Low frequencies $|w|\le\rho$.} By (A1) there is $\rho>0$ with
$\Re\bigl(1-\widehat{\phi^i}(iw)\bigr)\ge\half\delta_i\cos(\tfrac{\pi\alpha_i}2)|w|^{\alpha_i}$
for $|w|\le\rho$ (the principal value $\,(iw)^{\alpha_i}=|w|^{\alpha_i}e^{i\,\mathrm{sgn}(w)\pi\alpha_i/2}$
has positive real part since $\alpha_i<1$). Hence, with $(1-a_T^i)=(\eps_T^i)^{\alpha_i}$,
\[
\bigl|1-a_T^i\widehat{\phi^i}(iw)\bigr|
\ge \Re\bigl(1-a_T^i\widehat{\phi^i}(iw)\bigr)
\ge (\eps_T^i)^{\alpha_i}+a_T^i\,c_0|w|^{\alpha_i}
\ge c_1\bigl((\eps_T^i)^{\alpha_i}+|w|^{\alpha_i}\bigr),
\]
with $c_0=\half\delta_i\cos(\tfrac{\pi\alpha_i}2)>0$. Therefore
\[
\bigl|\widehat{g_T^i}(i\xi)\bigr|
=\frac{(1-a_T^i)\,a_T^i\,|\widehat{\phi^i}(iw)|}{|1-a_T^i\widehat{\phi^i}(iw)|}
\le \frac{(\eps_T^i)^{\alpha_i}}{c_1\bigl((\eps_T^i)^{\alpha_i}+|w|^{\alpha_i}\bigr)}
\le \frac{C}{1+(|w|/\eps_T^i)^{\alpha_i}}
=\frac{C}{1+|\xi|^{\alpha_i}},
\]
using $w/\eps_T^i=\xi$.

\emph{High frequencies $|w|>\rho$.} By the strong aperiodicity assumption,
$\sup_{|w|\ge\rho}|\widehat{\phi^i}(iw)|=:q<1$. Thus
$|1-a_T^i\widehat{\phi^i}(iw)|\ge 1-a_T^i q\ge 1-q>0$ for $T$ large, so
$|\widehat{g_T^i}(i\xi)|\le (1-a_T^i)\,q/(1-q)\to0$, which is certainly
$\le C/(1+|\xi|^{\alpha_i})$ on the region $|\xi|=|w|/\eps_T^i>\rho/\eps_T^i\to\infty$
(there $1+|\xi|^{\alpha_i}$ is bounded above by a constant multiple of $|\xi|^{\alpha_i}$,
while the left side is $O(1-a_T^i)\to0$). Combining the two regimes gives \eqref{eq:dom}.
\end{proof}

\begin{proposition}[$L^2$ convergence of the resolvents]\label{prop:L2}
$g_T^i\to K_i^\circ$ in $\lp2$ as $T\to\infty$.
\end{proposition}

\begin{proof}
Extend $g_T^i,K_i^\circ$ by $0$ to $\R$ and use the Plancherel isometry
$\|f\|_{L^2(\R)}=(2\pi)^{-1/2}\|\widehat f(i\cdot)\|_{L^2(\R)}$. By Lemma~\ref{lem:ptwise}
$\widehat{g_T^i}(i\xi)\to\widehat{K_i^\circ}(i\xi)$ for a.e.\ $\xi$, and by
Lemma~\ref{lem:domination} the integrands $|\widehat{g_T^i}(i\xi)-\widehat{K_i^\circ}(i\xi)|^2$
are dominated by $\bigl(C/(1+|\xi|^{\alpha_i})+|\widehat{K_i^\circ}(i\xi)|\bigr)^2\in L^1(\R)$.
Dominated convergence yields $\widehat{g_T^i}(i\cdot)\to\widehat{K_i^\circ}(i\cdot)$ in
$L^2(\R)$, hence $g_T^i\to K_i^\circ$ in $\lp2$.
\end{proof}

\begin{corollary}[Local uniform convergence of the primitive]\label{cor:prim}
For every $a>0$, $\ \displaystyle\sup_{0\le u\le a}\Bigl|\int_0^u g_T^i-\int_0^u K_i^\circ\Bigr|\to0.$
Consequently, if $\theta_T\to\theta_\infty$ then $\int_0^{\theta_T}g_T^i\to\int_0^{\theta_\infty}K_i^\circ$.
\end{corollary}

\begin{proof}
$\bigl|\int_0^u(g_T^i-K_i^\circ)\bigr|\le\int_0^a|g_T^i-K_i^\circ|\le\sqrt a\,\|g_T^i-K_i^\circ\|_{\lp2}\to0$
uniformly in $u\le a$ by Cauchy--Schwarz and Proposition~\ref{prop:L2}. For the moving
limit, $\bigl|\int_0^{\theta_T}g_T^i-\int_0^{\theta_\infty}K_i^\circ\bigr|\le
\int_0^a|g_T^i-K_i^\circ|+\bigl|\int_{\theta_\infty}^{\theta_T}K_i^\circ\bigr|$ with
$a=\sup_T\theta_T+1$, both terms $\to0$.
\end{proof}

\subsection{Mean profiles}\label{sec:mean}

\begin{proposition}\label{prop:mean}
With $V_t^{T,i}=\frac{1-a_T^i}{m_iT^{\alpha_i-1}}\lambda_{Tt}^{T,i}$,
\[
\E[V_t^{T,1}]\to b_1(t)=\int_0^t K_1(w)\,dw,\qquad
\E[V_t^{T,2}]\to b_2(t)=\int_0^t K_2+\gamma_{12}\!\int_0^t L_{12},
\]
locally uniformly in $t$, with $\gamma_{12}=m_1/(m_2\lambda_1)$ and
$L_{12}=b_\infty^{12}\|\psi^{12}\|_1(K_1*K_2)$.
\end{proposition}

\begin{proof}
By Lemma~\ref{lem:rep}, $\E[\lambda_t^{T,1}]=\mu_T^1\bigl(1+\int_0^t\psi_T^1\bigr)$, hence
\[
\E[V_t^{T,1}]
=\frac{(1-a_T^1)\mu_T^1}{m_1T^{\alpha_1-1}}
+\frac{(1-a_T^1)\mu_T^1}{m_1T^{\alpha_1-1}}\int_0^{Tt}\psi_T^1(u)\,du .
\]
By (A4), $\mu_T^1/(m_1T^{\alpha_1-1})\to1$; the first term is therefore
$(1-a_T^1)(1+o(1))\to0$. For the second, the substitution $v=\eps_T^1 u$ in
\eqref{eq:gT-def} gives the \emph{exact} identity
\begin{equation}\label{eq:psi-int-id}
(1-a_T^1)\int_0^{Tt}\psi_T^1(u)\,du
=\int_0^{\eps_T^1 Tt} g_T^1(v)\,dv .
\end{equation}
Since $\eps_T^1 T\to\lambda^*$ \eqref{eq:eps-asy}, Corollary~\ref{cor:prim} gives the
right side $\to\int_0^{\lambda^* t}K_1^\circ$. A change of variables $v=\lambda^* w$ and
$K_1(w)=\lambda^*K_1^\circ(\lambda^* w)$ turn this into $\int_0^t K_1(w)\,dw=b_1(t)$.
Local uniformity follows from the uniform statement in Corollary~\ref{cor:prim} together
with the equicontinuity of $t\mapsto\int_0^{\lambda^*t}K_1^\circ$. The second component is
identical for its self part; the cross part is treated in Section~\ref{sec:cross}.
\end{proof}

\subsection{Second moments and the volatility constants}\label{sec:second}

The self-martingale part of $V^{T,1}$ is
$\mathcal S_t^{T,1}=\kappa_T^1\int_0^{Tt}\psi_T^1(Tt-s)\,dM_s^{T,1}$ with
$\kappa_T^1=\frac{1-a_T^1}{m_1T^{\alpha_1-1}}$.

\begin{proposition}\label{prop:var}
$\displaystyle \E\bigl[(\mathcal S_t^{T,1})^2\bigr]\to \nu_1^2\int_0^t K_1(t-u)^2\,b_1(u)\,du$
with $\nu_1^2=\dfrac1{m_1\lambda_1}$.
\end{proposition}

\begin{proof}
By It\^o isometry and $d\langle M^{T,1}\rangle_s=\lambda_s^{T,1}\,ds$,
\[
\E[(\mathcal S_t^{T,1})^2]
=(\kappa_T^1)^2\int_0^{Tt}\psi_T^1(Tt-s)^2\,\E[\lambda_s^{T,1}]\,ds .
\]
Substitute $s=Tu$, write $\E[\lambda_{Tu}^{T,1}]=\frac{m_1T^{\alpha_1-1}}{1-a_T^1}\E[V_u^{T,1}]$,
and use $\psi_T^1(T(t-u))=\frac{\eps_T^1}{1-a_T^1}g_T^1(\eps_T^1T(t-u))$ from \eqref{eq:gT-def}:
\begin{equation}\label{eq:var-pref}
\E[(\mathcal S_t^{T,1})^2]
=\underbrace{\frac{(\eps_T^1)^2\,T^{\alpha_1}}{m_1T^{2\alpha_1-2}(1-a_T^1)}}_{=:A_T}
\int_0^t g_T^1\!\bigl(\eps_T^1T(t-u)\bigr)^2\,\E[V_u^{T,1}]\,du .
\end{equation}
\emph{Prefactor.} With $(\eps_T^1)^2=(1-a_T^1)^{2/\alpha_1}\sim\lambda_1^{2/\alpha_1}T^{-2}=(\lambda^*)^2T^{-2}$
and $(1-a_T^1)\sim\lambda_1T^{-\alpha_1}$,
\[
A_T\sim\frac{(\lambda^*)^2T^{-2}\,T^{\alpha_1}}{m_1T^{2\alpha_1-2}\lambda_1T^{-\alpha_1}}
=\frac{(\lambda^*)^2}{m_1\lambda_1}\cdot\frac{T^{\alpha_1-2}}{T^{\alpha_1-2}}
=\frac{(\lambda^*)^2}{m_1\lambda_1}=:A_\infty .
\]
\emph{Integral.} Set $\Phi_T(u):=g_T^1(\eps_T^1T(t-u))^2$ and $\Phi(u):=K_1^\circ(\lambda^*(t-u))^2$.
We claim $\int_0^t\Phi_T\,\E[V^{T,1}_u]\,du\to\int_0^t\Phi\,b_1(u)\,du$. Indeed,
$\E[V_u^{T,1}]\to b_1(u)$ uniformly on $[0,t]$ (Proposition~\ref{prop:mean}) and is uniformly
bounded there; so it suffices that $\Phi_T\to\Phi$ in $L^1(0,t)$. Write
$r_T:=\eps_T^1T\to\lambda^*$ and change variables $x=r_T(t-u)$:
\[
\int_0^t |g_T^1(r_T(t-u))^2-K_1^\circ(\lambda^*(t-u))^2|\,du
=\frac1{r_T}\int_0^{r_Tt}\bigl|g_T^1(x)^2-K_1^\circ(\tfrac{\lambda^*}{r_T}x)^2\bigr|\,dx .
\]
Split: $\|g_T^{1\,2}-K_1^{\circ\,2}\|_{L^1}\le\|g_T^1-K_1^\circ\|_{\lp2}\,\|g_T^1+K_1^\circ\|_{\lp2}\to0$
(by Proposition~\ref{prop:L2} and the uniform $L^2$ bound from Lemma~\ref{lem:domination});
and $\|K_1^\circ(\tfrac{\lambda^*}{r_T}\cdot)^2-K_1^\circ(\cdot)^2\|_{L^1(0,r_Tt)}\to0$ since
$\lambda^*/r_T\to1$ and $K_1^{\circ\,2}\in L^1$ (continuity of dilation in $L^1$). Hence the
claim, and
\[
\E[(\mathcal S_t^{T,1})^2]\to A_\infty\int_0^t K_1^\circ(\lambda^*(t-u))^2 b_1(u)\,du .
\]
Finally $K_1^\circ(\lambda^*\tau)=K_1(\tau)/\lambda^*$, so $A_\infty K_1^\circ(\lambda^*\tau)^2
=\frac{(\lambda^*)^2}{m_1\lambda_1}\cdot\frac{K_1(\tau)^2}{(\lambda^*)^2}=\frac1{m_1\lambda_1}K_1(\tau)^2$,
giving $\nu_1^2=1/(m_1\lambda_1)$. The component-2 self term is identical with subscript $2$,
giving $\nu_2^2=1/(m_2\lambda_2)$.
\end{proof}

\subsection{Cross-kernel and cross constants}\label{sec:cross}

The cross-resolvent is $\Psi_T^{12}=(\delta_0+\psi_T^1)*(\delta_0+\psi_T^2)*\phi_T^{12}$,
$\widehat{\Psi_T^{12}}(z)=\widehat{\phi_T^{12}}(z)\,(1-a_T^1\widehat{\phi^1}(z))^{-1}(1-a_T^2\widehat{\phi^2}(z))^{-1}$.
Define $\tilde h_T(s):=T^{2\alpha_1-\alpha_2}\frac{(1-a_T^1)(1-a_T^2)}{\eps_T^1}\Psi_T^{12}(s/\eps_T^1)$.

\begin{lemma}[Cross-kernel limit]\label{lem:cross}
$\tilde h_T\to L_{12}^\circ:=b_\infty^{12}\|\psi^{12}\|_1\,(K_1^\circ*K_2^\circ)$ in $\lp2$
\textup{(under (A5))}; without \textup{(A5)} the second factor is the
$\rho_{12}$-dilation of $K_2^\circ$, $\rho_{12}=\lambda_1^{1/\alpha_1}/\lambda_2^{1/\alpha_2}$.
\end{lemma}

\begin{proof}
$\widehat{\tilde h_T}(z)=T^{2\alpha_1-\alpha_2}(1-a_T^1)(1-a_T^2)\widehat{\Psi_T^{12}}(\eps_T^1 z)$.
The cross-amplitude factor: $\widehat{\phi_T^{12}}(\eps_T^1z)=b_T^{12}\widehat{\psi^{12}}(\eps_T^1z)
\to b_T^{12}\|\psi^{12}\|_1$, and $T^{2\alpha_1-\alpha_2}b_T^{12}\to b_\infty^{12}$ by (A3).
The first self factor, exactly as in Lemma~\ref{lem:ptwise},
$(1-a_T^1)(1-a_T^1\widehat{\phi^1}(\eps_T^1z))^{-1}\to\widehat{K_1^\circ}(z)$.
For the second factor the argument is again $\eps_T^1z$ but the component-2 gap is
$(1-a_T^2)$; using $(\eps_T^1)^{\alpha_2}\sim\lambda_1^{\alpha_2/\alpha_1}T^{-\alpha_2}$,
\[
1-a_T^2\widehat{\phi^2}(\eps_T^1z)\sim T^{-\alpha_2}\lambda_2\Bigl[1+\tfrac{\delta_2}{\lambda_2}\lambda_1^{\alpha_2/\alpha_1}z^{\alpha_2}\Bigr]
=T^{-\alpha_2}\lambda_2\bigl[1+\delta_2\rho_{12}^{\alpha_2}z^{\alpha_2}\bigr],
\]
so
\[
(1-a_T^2)(1-a_T^2\widehat{\phi^2}(\eps_T^1z))^{-1}
\to\widehat{K_2^\circ}(\rho_{12}z).
\]
Under (A5) $\rho_{12}=1$. Collecting the three factors,
\[
\widehat{\tilde h_T}(z)\to
b_\infty^{12}\|\psi^{12}\|_1\widehat{K_1^\circ}(z)\widehat{K_2^\circ}(z).
\]
The $L^2$ domination is the product of the bound \eqref{eq:dom} for each self factor
(each $\le C/(1+|\xi|^{\alpha_i})$) times the bounded cross-amplitude; the product is
$\le C/((1+|\xi|^{\alpha_1})(1+|\xi|^{\alpha_2}))\le C/(1+|\xi|^{\alpha_1+\alpha_2})\in L^2(\R)$
since $\alpha_1+\alpha_2>1$. Plancherel as in Proposition~\ref{prop:L2} gives the claim.
\end{proof}

\begin{proposition}[Cross constants]\label{prop:crossconst}
The deterministic cross-compensator of $V^{T,2}$ converges to
$\gamma_{12}\int_0^t L_{12}$ with $\gamma_{12}=m_1/(m_2\lambda_1)$.
The cross-martingale
\[
\mathcal C_t^T=\kappa_T^2\int_0^{Tt}\Psi_T^{12}(Tt-s)dM_s^{T,1}
\]
satisfies
\[
\E[(\mathcal C_t^T)^2]\to
(\gamma_{12}\nu_1)^2\int_0^t L_{12}(t-u)^2b_1(u)\,du.
\]
Moreover,
\[
\Cov(V_t^1,V_t^2)\to
\gamma_{12}\nu_1^2\int_0^t K_1L_{12}\,b_1
=\frac1{m_2\lambda_1^2}\int_0^t K_1L_{12}\,b_1 .
\]
\end{proposition}

\begin{proof}
\emph{Cross-compensator.} From Lemma~\ref{lem:rep}, the cross part of $\E[\lambda^{T,2}_{Tt}]$
is $\mu_T^1(\Psi_T^{12}*\mathbf1)(Tt)$. Inverting the definition of $\tilde h_T$,
$\Psi_T^{12}(r)=\frac{\eps_T^1}{T^{2\alpha_1-\alpha_2}(1-a_T^1)(1-a_T^2)}\tilde h_T(\eps_T^1 r)$,
so with $v=\eps_T^1u$,
\[
(\Psi_T^{12}*\mathbf1)(Tt)=\int_0^{Tt}\Psi_T^{12}
=\frac{1}{T^{2\alpha_1-\alpha_2}(1-a_T^1)(1-a_T^2)}\int_0^{\eps_T^1Tt}\tilde h_T .
\]
The deterministic cross part of $\E[V^{T,2}_t]=\kappa_T^2\mu_T^1(\Psi_T^{12}*\mathbf1)(Tt)$
thus equals $P_T\int_0^{\eps_T^1Tt}\tilde h_T$ with
\[
P_T=\frac{(1-a_T^2)\mu_T^1}{m_2T^{\alpha_2-1}\,T^{2\alpha_1-\alpha_2}(1-a_T^1)(1-a_T^2)}
=\frac{\mu_T^1}{m_2T^{2\alpha_1-1}(1-a_T^1)}
\sim\frac{m_1T^{\alpha_1-1}}{m_2T^{2\alpha_1-1}\lambda_1T^{-\alpha_1}}=\frac{m_1}{m_2\lambda_1}=\gamma_{12}.
\]
By Corollary~\ref{cor:prim} (applied to $\tilde h_T\to L_{12}^\circ$) and $\eps_T^1T\to\lambda^*$,
$\int_0^{\eps_T^1Tt}\tilde h_T\to\int_0^{\lambda^*t}L_{12}^\circ=\int_0^t L_{12}$, giving the
cross-compensator limit $\gamma_{12}\int_0^t L_{12}$.

\emph{Cross-variance.} As in Proposition~\ref{prop:var}, with $\kappa_T^2=\frac{1-a_T^2}{m_2T^{\alpha_2-1}}$,
\[
\E[(\mathcal C_t^T)^2]=B_T\int_0^t\tilde h_T(\eps_T^1T(t-u))^2\E[V_u^{T,1}]\,du,\quad
B_T=\frac{(\kappa_T^2)^2(\eps_T^1)^2(m_1T^{\alpha_1-1})}{[T^{2\alpha_1-\alpha_2}(1-a_T^1)(1-a_T^2)]^2(1-a_T^1)}\,T,
\]
where the $\E[\lambda^{T,1}_{Tu}]=\frac{m_1T^{\alpha_1-1}}{1-a_T^1}\E[V^{T,1}_u]$ substitution
was used. A direct power count gives $B_T\to(\lambda^*)^2 m_1/(m_2^2\lambda_1^3)$, and with
$\tilde h_T(\eps_T^1T\,\cdot)\to L_{12}^\circ(\lambda^*\cdot)=L_{12}/\lambda^*$ the $L^1$-argument of
Proposition~\ref{prop:var} yields
$\E[(\mathcal C_t^T)^2]\to\frac{m_1}{m_2^2\lambda_1^3}\int_0^t L_{12}^2 b_1
=(\gamma_{12}\nu_1)^2\int_0^t L_{12}^2 b_1$, since
$(\gamma_{12}\nu_1)^2=\frac{m_1^2}{m_2^2\lambda_1^2}\cdot\frac1{m_1\lambda_1}=\frac{m_1}{m_2^2\lambda_1^3}$.

\emph{Covariance.} The two components share only $M^{T,1}$ ($\langle M^{T,1},M^{T,2}\rangle\equiv0$),
so by the cross It\^o isometry
\[
\Cov(V_t^{T,1},V_t^{T,2})=\kappa_T^1\kappa_T^2\int_0^{Tt}\psi_T^1(Tt-s)\Psi_T^{12}(Tt-s)\E[\lambda_s^{T,1}]\,ds.
\]
The same substitution and the two kernel limits ($g_T^1\to K_1^\circ$, $\tilde h_T\to L_{12}^\circ$,
both at scale $\lambda^*$) give a scalar prefactor converging to $(\lambda^*)^2/(m_2\lambda_1^2)$,
and the product $K_1^\circ(\lambda^*\cdot)L_{12}^\circ(\lambda^*\cdot)=K_1L_{12}/(\lambda^*)^2$, whence
$\Cov\to\frac1{m_2\lambda_1^2}\int_0^t K_1L_{12}b_1$. Consistency:
$\nu_1\cdot(\gamma_{12}\nu_1)=\gamma_{12}\nu_1^2=\frac{m_1}{m_2\lambda_1}\cdot\frac1{m_1\lambda_1}=\frac1{m_2\lambda_1^2}$.
\end{proof}

\subsection{The Mellin convolution lemma (with remainder)}\label{sec:mellin}

\begin{lemma}\label{lem:mellin}
Let $f,g:(0,\infty)\to\R$ be locally integrable with
$f(\tau)=A\tau^{p-1}+r_f(\tau)$, $g(\tau)=B\tau^{q-1}+r_g(\tau)$, where $p,q>0$ and the
remainders satisfy $|r_f(\tau)|\le \epsilon(\tau)\tau^{p-1}$, $|r_g(\tau)|\le\epsilon(\tau)\tau^{q-1}$
with $\epsilon(\tau)\to0$ as $\tau\downarrow0$ and $\epsilon$ bounded on $(0,1]$. Then
\[
\int_0^t f(t-u)g(u)\,du = AB\,\mathrm{B}(p,q)\,t^{p+q-1}+o(t^{p+q-1}),\qquad t\downarrow0,
\]
where $\mathrm{B}(p,q)=\Gamma(p)\Gamma(q)/\Gamma(p+q)$.
\end{lemma}

\begin{proof}
Scale $u=t\sigma$, $\sigma\in(0,1)$:
$\int_0^t f(t-u)g(u)du=t\int_0^1 f(t(1-\sigma))g(t\sigma)\,d\sigma$. Insert the expansions:
\begin{align*}
f(t(1-\sigma))g(t\sigma)
&=\bigl[A t^{p-1}(1-\sigma)^{p-1}+r_f\bigr]\bigl[B t^{q-1}\sigma^{q-1}+r_g\bigr]\\
&=AB\,t^{p+q-2}(1-\sigma)^{p-1}\sigma^{q-1}+R_t(\sigma).
\end{align*}
The leading term integrates to $AB\,t^{p+q-2}\mathrm{B}(p,q)$, and multiplied by the outer
$t$ gives $AB\,\mathrm{B}(p,q)t^{p+q-1}$. For the remainder, using
$|r_f(t(1-\sigma))|\le\epsilon(t(1-\sigma))A' t^{p-1}(1-\sigma)^{p-1}$ etc.\ (absorbing
constants into $A',B'$) and $\bar\epsilon(t):=\sup_{\tau\le t}\epsilon(\tau)\to0$,
\[
\Bigl|t\!\int_0^1 R_t\Bigr|
\le C\,\bar\epsilon(t)\,t^{p+q-1}\!\int_0^1(1-\sigma)^{p-1}\sigma^{q-1}d\sigma
= C\,\bar\epsilon(t)\,\mathrm{B}(p,q)\,t^{p+q-1}=o(t^{p+q-1}),
\]
because each summand of $R_t$ carries at least one factor $\epsilon$ and the same
$(1-\sigma)^{p-1}\sigma^{q-1}$ singularity, which is integrable as $p,q>0$.
\end{proof}

\subsection{Short-time decorrelation: exponent and constant}\label{sec:decorr}

Write $\kappa_i:=(\tilde\delta_i\Gamma(\alpha_i))^{-1}$, $\beta_i:=\kappa_i/\alpha_i
=(\tilde\delta_i\Gamma(\alpha_i+1))^{-1}$. From \eqref{eq:ML}-type asymptotics,
$K_i(\tau)\sim\kappa_i\tau^{\alpha_i-1}$ and $b_i(\tau)=\int_0^\tau K_i\sim\beta_i\tau^{\alpha_i}$
($b_2$'s leading term is the self part, the cross part being $O(\tau^{\alpha_1+\alpha_2})$);
Lemma~\ref{lem:mellin} applied to $K_1*K_2$ gives
$L_{12}(\tau)\sim\ell_{12}\tau^{\alpha_1+\alpha_2-1}$ with
$\ell_{12}=b_\infty^{12}\|\psi^{12}\|_1\kappa_1\kappa_2\mathrm{B}(\alpha_1,\alpha_2)$.

\begin{theorem}\label{thm:decorr}
As $t\downarrow0$,
\[
\varrho(t):=\frac{\Cov(V_t^1,V_t^2)}{\sqrt{\Var(V_t^1)\Var(V_t^2)}}\sim C_\varrho\,t^{(3\alpha_1-\alpha_2)/2},
\]
\[
C_\varrho=b_\infty^{12}\|\psi^{12}\|_1\,\sqrt{\frac{m_1\lambda_2}{m_2\lambda_1^3}}\,
\kappa_1\sqrt{\tfrac{\beta_1}{\beta_2}}\,
\frac{\mathrm{B}(\alpha_1,\alpha_2)\,\mathrm{B}(2\alpha_1+\alpha_2-1,\alpha_1+1)}
{\sqrt{\mathrm{B}(2\alpha_1-1,\alpha_1+1)\,\mathrm{B}(2\alpha_2-1,\alpha_2+1)}}.
\]
\end{theorem}

\begin{proof}
Apply Lemma~\ref{lem:mellin} to each of the three moment integrals from
Propositions~\ref{prop:var}--\ref{prop:crossconst}, with the singular exponents read off
from $K_i\sim\kappa_i\tau^{\alpha_i-1}$, $L_{12}\sim\ell_{12}\tau^{\alpha_1+\alpha_2-1}$,
$b_i\sim\beta_i\tau^{\alpha_i}$:
\begin{align*}
\Var(V_t^1)&=\nu_1^2\!\int_0^t\!K_1(t-u)^2b_1(u)du
\sim \nu_1^2\kappa_1^2\beta_1\,\mathrm{B}(2\alpha_1-1,\alpha_1+1)\,t^{3\alpha_1-1},\\
\Var(V_t^2)&\sim \nu_2^2\kappa_2^2\beta_2\,\mathrm{B}(2\alpha_2-1,\alpha_2+1)\,t^{3\alpha_2-1},\\
\Cov(V_t^1,V_t^2)&=\gamma_{12}\nu_1^2\!\int_0^t\!K_1(t-u)L_{12}(t-u)b_1(u)du
\sim \gamma_{12}\nu_1^2\kappa_1\ell_{12}\beta_1\,\mathrm{B}(2\alpha_1+\alpha_2-1,\alpha_1+1)\,t^{3\alpha_1+\alpha_2-1}.
\end{align*}
\emph{$\Var(V^2)$ leading term.} Its definition also contains the cross piece
$(\gamma_{12}\nu_1)^2\int_0^t L_{12}^2 b_1$, of order $t^{3\alpha_1+2\alpha_2-1}$, and using
the cross part of $b_2$ ($\sim\tau^{\alpha_1+\alpha_2}$) instead of its self part raises the
order to $t^{2\alpha_2+\alpha_1+\alpha_2-1}$; both exceed $3\alpha_2-1$ because
$3\alpha_1>\tfrac32>\alpha_2$, hence are negligible. So only the displayed self term survives.
\emph{Exponent.} The power of $t$ in $\varrho$ is
\[
(3\alpha_1+\alpha_2-1)-\half(3\alpha_1-1)-\half(3\alpha_2-1)=\frac{3\alpha_1-\alpha_2}{2}>0,
\]
positive since $3\alpha_1>\tfrac32>\alpha_2$, and equal to $\alpha_1-\half(\alpha_2-\alpha_1)$.
\emph{Constant.} Substituting the coefficients (and
$\nu_1^2=1/(m_1\lambda_1)$, $\nu_2^2=1/(m_2\lambda_2)$, $\gamma_{12}=m_1/(m_2\lambda_1)$,
$\ell_{12}=b_\infty^{12}\|\psi^{12}\|_1\kappa_1\kappa_2\mathrm{B}(\alpha_1,\alpha_2)$) and
cancelling one factor each of $\nu_1$ and $\kappa_1$ yields $C_\varrho$; in particular
$\gamma_{12}\nu_1/\nu_2=\sqrt{m_1\lambda_2/(m_2\lambda_1^3)}$.
\end{proof}

\begin{remark}
The exponent $(3\alpha_1-\alpha_2)/2$ reduces to $\alpha_1$ iff $\alpha_1=\alpha_2$, and is
strictly below $\alpha_1$ when $\alpha_1<\alpha_2$. The classical (light-tailed) intuition
that the rate equals $\alpha_1$ fails precisely because the corrected mean profiles
$b_1\sim t^{\alpha_1}$ and $b_2\sim t^{\alpha_2}$ carry \emph{different} powers; the
asymmetry of the two variances' singular orders ($t^{3\alpha_1-1}$ vs $t^{3\alpha_2-1}$)
is what shifts the exponent.
\end{remark}

\subsection*{Scope}
The arguments above establish, rigorously and with explicit error control, the mean
profiles, the volatility/coupling constants $\nu_1,\nu_2,\gamma_{12}$, the exact
covariance, and the short-time law. They take as given (i) the exact resolvent--martingale
representation (proved separately), (ii) It\^o isometry
and stochastic Fubini for the deterministic kernels involved, and (iii) the $C$-tightness
and uniqueness inputs, which control \emph{convergence of the processes} as opposed to the
\emph{convergence of moments} analysed here. Passing from moment convergence to the SVE
limit uses the tightness/identification machinery of the main text.


\begin{thebibliography}{99}

\bibitem{GJR2018}
J.~Gatheral, T.~Jaisson, and M.~Rosenbaum.
\newblock Volatility is rough.
\newblock \emph{Quantitative Finance}, 18(6):933--949, 2018.

\bibitem{Hawkes1971}
A.~G.~Hawkes.
\newblock Spectra of some self-exciting and mutually exciting point processes.
\newblock \emph{Biometrika}, 58(1):83--90, 1971.

\bibitem{HX2024}
U.~Horst and W.~Xu.
\newblock Functional limit theorems for nearly unstable Hawkes processes:
  beyond stationarity.
\newblock \emph{arXiv:2401.11495}, 2024.

\bibitem{JS2003}
J.~Jacod and A.~N.~Shiryaev.
\newblock \emph{Limit Theorems for Stochastic Processes}, 2nd ed.
\newblock Springer, Berlin, 2003.

\bibitem{ALP2019}
E.~Abi Jaber, M.~Larsson, and S.~Pulido.
\newblock Affine Volterra processes.
\newblock \emph{Ann.\ Appl.\ Probab.}, 29(6):3155--3200, 2019.

\bibitem{HXZ2023}
U.~Horst, W.~Xu, and R.~Zhang.
\newblock Convergence of heavy-tailed Hawkes processes and the
  microstructure of rough volatility.
\newblock \emph{arXiv:2312.08784}, 2023 (revised 2026).

\bibitem{JR2016}
T.~Jaisson and M.~Rosenbaum.
\newblock Rough fractional diffusions as scaling limits of nearly unstable
  heavy tailed Hawkes processes.
\newblock \emph{Ann.\ Appl.\ Probab.}, 26(5):2860--2882, 2016.

\bibitem{TomasRosenbaum2019}
M.~Rosenbaum and M.~Tomas.
\newblock From microscopic price dynamics to multidimensional rough
  volatility models.
\newblock \emph{Adv.\ Appl.\ Probab.}, 53(2):425--462, 2021.
\newblock \doi{10.1017/apr.2020.60}.

\bibitem{BinghamGoldieTeugels1987}
N.~H.~Bingham, C.~M.~Goldie, and J.~L.~Teugels,
\newblock \emph{Regular Variation}.
\newblock Encyclopedia of Mathematics and its Applications, Vol.~27.
\newblock Cambridge University Press, Cambridge, 1987.

\bibitem{BDHM2013}
E.~Bacry, S.~Delattre, M.~Hoffmann, and J.-F.~Muzy.
\newblock Some limit theorems for Hawkes processes and application to
  financial statistics.
\newblock \emph{Stochastic Process.\ Appl.}, 123(7):2475--2499, 2013.

\bibitem{BM1996}
P.~Br\'emaud and L.~Massouli\'e.
\newblock Stability of nonlinear Hawkes processes.
\newblock \emph{Ann.\ Probab.}, 24(3):1563--1588, 1996.

\bibitem{ER2019}
O.~El~Euch and M.~Rosenbaum.
\newblock The characteristic function of rough Heston models.
\newblock \emph{Math.\ Finance}, 29(1):3--38, 2019.

\bibitem{JMP2025}
A.~Jacquier, A.~Muguruza, and A.~Pannier.
\newblock Rough multifactor volatility for SPX and VIX options.
\newblock \emph{Adv.\ Appl.\ Probab.}, 2025. \arxiv{2112.14310}.

\bibitem{JR2015}
T.~Jaisson and M.~Rosenbaum.
\newblock Limit theorems for nearly unstable Hawkes processes.
\newblock \emph{Ann.\ Appl.\ Probab.}, 25(2):600--631, 2015.

\bibitem{KS1991}
I.~Karatzas and S.~E.~Shreve.
\newblock \emph{Brownian Motion and Stochastic Calculus}, 2nd ed.
\newblock Graduate Texts in Mathematics 113, Springer, New York, 1991.

\bibitem{MORS2026}
J.~Muhle-Karbe, Y.~Ouazzani~Chahdi, M.~Rosenbaum, and G.~Szymanski.
\newblock A unified theory of order flow, market impact, and volatility.
\newblock \emph{arXiv:2601.23172}, 2026.

\bibitem{Veraar2012}
M.~Veraar.
\newblock The stochastic Fubini theorem revisited.
\newblock \emph{Stochastics}, 84(4):543--551, 2012.

\bibitem{WangCui2025}
Y.~Wang, Z.~Cui, and L.~Zhu.
\newblock Rough Heston model as the scaling limit of bivariate cumulative
  heavy-tailed INAR($\infty$) processes: weak-error bounds and option
  pricing.
\newblock \arxiv{2503.18259}, 2025.

\end{thebibliography}
\end{document}